\documentclass[12pt]{article}
\usepackage{mathrsfs}
\usepackage{amsfonts}
\usepackage{enumitem}
\usepackage{amsmath, amsfonts, amssymb, amsthm, amscd}
\usepackage[all]{xy}

\theoremstyle{plain}
\newtheorem{thm}{Theorem}[section]
\newtheorem{theorem}[thm]{Theorem}
\newtheorem{lemma}[thm]{Lemma}
\newtheorem{corollary}[thm]{Corollary}
\newtheorem{proposition}[thm]{Proposition}

\theoremstyle{remark}
\newtheorem{definition}[thm]{Definition}

\newtheorem{remark}[thm]{Remark}

\newtheorem{defn-thm}[thm]{Definition-Theorem}
\renewcommand{\bar}{\overline}

\renewcommand{\phi}{\varphi}

\newcommand{\C}{{\mathbb C}}

\newcommand{\T}{{\mathcal T}}

\renewcommand{\tilde}{\widetilde}

\newcommand{\g}{{\mathfrak g}}

\newcommand\independent{\protect\mathpalette{\protect\independenT}{\perp}}
\def\independenT#1#2{\mathrel{\rlap{$#1#2$}\mkern2mu{#1#2}}}

\def\sZ{\mathcal{Z}_{m}}
\def\Z{\mathcal{Z}_{m}}
\def\ZZ{\mathcal{Z}_{m}^{H}}

\def\P{\Phi}
\def\PP{\Phi_{m}^{H}}
\def\TT{\T_{m}^{H}}

\def\C{\mathbb{C}}

\def\T{\mathcal{T}}

\def\U{\mathcal {U}}

\title{Hodge metric completion of the moduli space of Calabi--Yau manifolds}

\author{Kefeng Liu and Yang Shen}

\begin{document}
\date{}

\maketitle

\vspace{-20pt}

\begin{abstract} 
In this paper, it is proved that the Hodge metric completion of
the moduli space of polarized and marked Calabi--Yau manifolds, i.e. the Torelli space, is
a complex affine manifold. 
As applications we prove that the period map from the Torelli
 space and the extended period map from its completion space, both are injective into the period domain, and that the completion space is a bounded domain of holomorphy with a complete K\"ahler-Einstein metric.
As a corollary we show that the period map from the moduli space of polarized Calabi--Yau manifolds
with level $m$ structure is also injective.
\end{abstract}

\parskip=5pt
\baselineskip=15pt





\setcounter{section}{-1}
\section{Introduction}
A compact projective manifold $M$ of complex dimension $n$ with $n\geq 3$ is called
Calabi--Yau in this paper, if it has a trivial canonical bundle and
satisfies $H^i(M,\mathcal{O}_M)=0$ for $0<i<n$.
We fix a lattice $\Lambda$ with an pairing $Q_{0}$, where $\Lambda$ is isomorphic to $H^n(M_{0},\mathbb{Z})/\text{Tor}$ for some fixed Calabi--Yau manifold $M_{0}$, and $Q_0$ is the intersection pairing.
A polarized and
marked Calabi--Yau manifold is a triple consisting of a Calabi--Yau
manifold $M$, an ample line bundle $L$ over $M$, and a marking $\gamma$ defined as an isometry of the lattices
$$\gamma :\, (\Lambda, Q_{0})\to (H^n(M,\mathbb{Z})/\text{Tor},Q).$$

Let $\mathcal{Z}_m$ be a smooth irreducible component of the moduli space of polarized Calabi--Yau manifolds with level $m$ structure with $m\geq 3$, which is constructed by Popp, Viehweg, and Szendr\"oi. For example, see Section 2 of \cite{sz}.  We define the Teichm\"uller space of Calabi--Yau manifolds to be the universal cover of $\mathcal{Z}_m$, which will be proved to be independent of the choice of $m$.
We will denote by $\T$ the Teichm\"uller space of Calabi--Yau manifolds.

Let $\mathcal{T}'$ be a smooth irreducible component of the moduli space of
equivalence classes of marked and polarized Calabi--Yau manifolds.  We call $\mathcal{T}'$ the Torelli space of Calabi--Yau manifolds in this paper.  We will assume that both $\mathcal{Z}_m$ and $\mathcal{T}'$ contain the polarized Calabi--Yau manifold that we start with. We will see that both the Teichm\"uller space $\T$ and the Torelli space $\T'$ are covering spaces of $\sZ$, therefore $\T$ is also the universal cover of $\T'$.  See Section \ref{TandPhi} for details. The Torelli space $\T'$ is also called the framed moduli as discussed in \cite{Beau}. We will see that the Torelli space $\T'$ is the most natural space to consider the period map and to study the Torelli problem.

We summarize the relations of these spaces in the following commutative diagram
\begin{equation}\label{inr cover1}
\xymatrix{
\T \ar[dr]^-{\pi} \ar[dd]^-{\pi_m} &\\
&\T'\ar[dl]^-{\pi_m'},\\
\sZ &}
\end{equation}
with $\pi_{m}$, $\pi'_{m}$ and $\pi$ the corresponding covering maps.
Therefore both the Teichm\"uller space $\T$ and the Torelli space $\T'$ are connected  complex manifolds with
$$\dim_{{\mathbb{C}}}{\mathcal{T}}=\dim_{{\mathbb{C}}}{\mathcal{T}'}=h^{n-1,1}(M)=N,$$
where $h^{n-1,1}(M)=\dim_{\mathbb{C}} H^{n-1,1}(M)$. There exist  universal families
$\mathcal{U}\rightarrow\mathcal{T}$ and $\mathcal{U}'\rightarrow\mathcal{T}'$ over the corresponding spaces as the pull-backs of the universal family $\mathcal{X}_{\mathcal{Z}_m}\rightarrow \mathcal{Z}_m$ as introduced in \cite{sz}. The markings give canonical identifications of the $n$-th cohomology groups of fibers of the universal family. We refer the reader to Section \ref{TandPhi} for details of the discussions about moduli, Torelli and Teichm\"uller spaces, as well as the period maps for Calabi--Yau manifolds.

Let $D$ be the period domain of polarized Hodge structures of
the $n$-th primitive cohomology of $M$.
Let us denote the period map on the smooth moduli space $\sZ$ by $$\Phi_{\mathcal{Z}_m}:\, \mathcal{Z}_{m}\rightarrow D/\Gamma,$$where
$\Gamma$ denotes the global monodromy group which acts
properly and discontinuously on $D$.
Then we can lift the period map $\Phi_{\mathcal{Z}_m}$ to the period map $\Phi:\, \T \to D$ from the universal cover $\T$ of $\sZ$, such that the following diagram commutes
$$\xymatrix{
\T \ar[r]^-{\Phi} \ar[d]^-{\pi_m} & D\ar[d]^-{\pi_{D}}\\
\sZ \ar[r]^-{\Phi_{\sZ}} & D/\Gamma.
}$$
Similarly we define the period map $$\Phi':\, \T'\to D$$ on the Torelli space $\T'$ by the definition of marking, such that the above diagram and diagram \eqref{inr cover1} fit into the following commutative diagram
\begin{equation}\label{inr cover2}\xymatrix{
\T \ar[dr]^-{\pi}\ar[rr]^-{\P} \ar[dd]^-{\pi_m} && D\ar[dd]^-{\pi_{D}}\\
&\T'\ar[ur]^-{\P'}\ar[dl]_-{\pi'_{m}}&\\
\Z \ar[rr]^-{\Phi_{\Z}} && D/\Gamma.
}
\end{equation}

There is a natural metric, called Hodge metric $h$, on $D$, which is a complete homogeneous metric induced from the Killing form as studied in detail in \cite{GS}.  By local Torelli theorem of Calabi--Yau manifolds, both $\Phi_{\mathcal{Z}_m}$ and $\Phi$ are nondegenerate, and it is shown in \cite{Lu} that the pull-backs of $h$ on $\mathcal{Z}_m$ and $\mathcal{T}$ are both well-defined K\"ahler metrics. Clearly $\P'$ is also nondegenerate, and the pull-back of the Hodge metric to $\T'$ by $\P'$ also defines a K\"ahler metric on the Torelli space $\T'$. These K\"ahler metrics are still called the Hodge metrics in this paper.

One of our crucial constructions is the global holomorphic affine structures on the Teichm\"uller space and the Hodge metric completion space of the Torelli space. Here we first outline the construction of affine structure on the Teichm\"uller space to give the reader some basic ideas of our method. See Section \ref{affine on T} for details.
We first fix a base point $p\in \mathcal{T}$ with its Hodge structure $\{H^{k,n-k}_p\}_{k=0}^n$ as the reference Hodge structure. With this fixed base point $\Phi(p)=o\in D$, we identify the unipotent subgroup $N_+=\text{exp}(\mathfrak{n}_{+})$ with its orbit in $\check{D}$  and define
$$\check{\mathcal{T}}=\Phi^{-1}(N_+\cap D)\subseteq \mathcal{T}.$$
We then prove that $\mathcal{T}\backslash\check{\mathcal{T}}$ is an analytic subvariety of $\T$, and the image of $\check{\T}$ under the period map $\P$ is bounded in $N_+$ which is identified to $\mathfrak{n}_+\simeq \text{T}_o^{1,0}D$ as a complex Euclidean space with induced metric from the Hodge metric on $D$.

The proof of the boundedness combines an argument of Harish-Chandra in proving that Hermitian symmetric spaces are bounded domains in complex Euclidean spaces,  and several fundamental properties of variation of Hodge structure from geometry and the geometry of period maps. By applying the Riemann extension theorem, we conclude that $$\Phi(\mathcal{T})\subseteq N_+\cap D$$ is a bounded subset in $N_+$. In this case we will simply say that the period map $\Phi$ is bounded. Here the Euclidean structure on $\mathfrak{n}_+$ and $N_+$ are induced from the Hodge metric on $D$ by the identification of $\mathfrak{n}_+$
to the holomorphic tangent space $T^{1,0}_oD$ of $D$ at the base point $o$. We also consider $N_+$ as a complex Euclidean space such that the exponential map $$\text{exp}:\, \mathfrak{n}_+\to N_+$$ is an isometry.

We then consider the abelian subalgebra $\mathfrak{a}\subset \mathfrak{n}_{+}$, which is defined by the image of the differential of the period map at the base point $p\in \T$, and $$A=\text{exp}(\mathfrak{a})\subset N_{+}$$ the corresponding abelian Lie group. We consider $\mathfrak{a}$  and $A$ as an Euclidean subspace of $\mathfrak{n}_+$ of $N_+$ respectively.
Denote the projection map by $$P:\, N_{+}\cap D\to A\cap D$$ and define $\Psi= P\circ \Phi$.

With local Torelli theorem for Calabi--Yau manifolds we then show that the holomorphic map
$$\Psi: \,{\mathcal{T}}\rightarrow A\cap D \subset A\simeq \C^{N}$$
 is nondegenerate on $\mathcal{T}$. Thus $\Psi: \,\mathcal{T}\rightarrow A\cap D$ induces a holomorphic affine structure on $\mathcal{T}$ by pulling back the affine structure on $\C^N$.



To proceed further, we must consider the Hodge metric completions of moduli spaces. More precisely, consider a smooth projective compactification $\bar{\mathcal{Z}}_{m}$ such that $\mathcal{Z}_m$ is Zariski open in $\bar{\mathcal{Z}}_{m}$ and the complement $\bar{\mathcal{Z}}_{m}\backslash\mathcal{Z}_m$ is a divisor of normal crossings. Let $\mathcal{Z}^H_{m}$ be the Hodge metric completion of the smooth moduli space $\mathcal{Z}_m$ inside $\bar{\mathcal{Z}}_{m}$, and let $\mathcal{T}^H_{m}$ be the universal cover of $\mathcal{Z}^H_{m}$ with the universal covering map $$\pi_{m}^H:\,\mathcal{T}^H_{m}\rightarrow \mathcal{Z}^H_{m}.$$ 

We will show that  $\mathcal{Z}^H_{m}$ can be identified to the Griffiths completion space of finite monodromy as introduced in Theorem 9.6 in \cite{Griffiths3}. In this paper, the space $\T^H_m$ will be called the Hodge metric completion space. In fact, we will prove that $\T^H_m$ is the completion space of the Torelli space with respect to the Hodge metric on it. Sometimes we simply call $\T^H_m$ the completion space for convenience. Lemma \ref{cidim} shows that $\mathcal{Z}^H_{m}$ is a connected and complete complex manifold, and thus $\mathcal{T}^H_{m}$ is a connected and simply connected complete complex manifold.

By using level structure and Serre's lemma, we can show that $\mathcal{Z}^H_{m}\setminus \Z$ consists of the points around which the so-called Picard-Lefschetz transformations are trivial for $m\geq 3$. Therefore we also have the following commutative diagram by choosing a lifting map $i_m$,
\begin{align}\label{introdiagram}
\xymatrix{\mathcal{T}\ar[r]^{i_{m}}\ar[d]^{\pi_m}&\mathcal{T}^H_{m}\ar[d]^{\pi_{m}^H}\ar[r]^{{\Phi}^{H}_{m}}&D\ar[d]^{\pi_D}\\
\mathcal{Z}_m\ar[r]^{i}&\mathcal{Z}^H_{m}\ar[r]^{{\Phi}_{{\mathcal{Z}_m}}^H}&D/\Gamma,
}
\end{align}
where $\Phi^H_{{\mathcal{Z}_m}}$ is the continuous extension of the period map $$\Phi_{\mathcal{Z}_m}:\,\mathcal{Z}_m\rightarrow D/\Gamma$$ and $i$ is the inclusion map. Also $i_m$ is a lifting of $i\circ \pi_m$,  $\Phi^H_{m}$ is a lifting of $\Phi^H_{{\mathcal{Z}_m}}\circ \pi_{m}^H$. We prove in Lemma \ref{choice} that there is a suitable choice of $i_m$ and $\Phi^H_{m}$ such that $\Phi=\Phi^H_{m}\circ i_m$.

As a corollary of the boundedness of the period map $\P$, we know that $\Phi^H_{m}$ is actually a bounded holomorphic map from $\mathcal{T}^H_{m}$ to $N_+\cap D$. Our first result is the following theorem.
\begin{theorem}\label{defnofTH}
For any $m\geq 3$, the complete complex manifold $\mathcal{T}^H_{m}\simeq A\cap D$ is a complex affine manifold, which is a bounded domain in $A\simeq \mathbb{C}^N$. Moreover, the holomorphic map $$\Phi_{m}^H:\, \mathcal{T}^H_{m}\rightarrow  N_+\cap D$$ is an injection. As a consequence, the complex manifolds $\mathcal{T}^H_{m}$ and $\mathcal{T}^H_{{m'}}$ are biholomorphic to each other for any $m, m'\geq 3$.
\end{theorem}
This theorem allows us to get rid of the subscript  $m$, and define the complete complex manifold $\mathcal{T}^H$ with respect to the Hodge metric by $\mathcal{T}^H=\mathcal{T}^H_{m}$, the holomorphic map $i_{\mathcal{T}}: \,\mathcal{T}\to \mathcal{T}^H$ by $i_{\mathcal{T}}=i_m$, and the extended period map $$\Phi^H:\, \mathcal{T}^H\rightarrow D$$ by $\Phi^H=\Phi^H_{m}$ for any $m\geq 3$.
Then diagram \eqref{introdiagram} becomes
\begin{align}\label{introdiagram'}
\xymatrix{\mathcal{T}\ar[r]^{i_\T}\ar[d]^{\pi_m}&\mathcal{T}^H\ar[d]^{\pi_{m}^H}\ar[r]^{{\Phi}^{H}}&D\ar[d]^{\pi_D}\\
\mathcal{Z}_m\ar[r]^{i}&\mathcal{Z}^H_{m}\ar[r]^{{\Phi}_{{\mathcal{Z}_m}}^H}&D/\Gamma.
}
\end{align}
By these definitions, Theorem \ref{defnofTH} implies that $\mathcal{T}^H$ is a complex affine manifold and that $$\Phi^H:\,\mathcal{T}^H\to N_+\cap D$$ is a holomorphic injection.

It is easy to show that the map $i_\T=i_m$ is a covering map onto its image. Then we prove that the Torelli space  $\T'$ is biholomorphic to the image
$\T_0=i_{\T}(\T)$ of $i_{\T}$ in $\T^{H}$. From this we define an injective map $$\pi^{0}:\, \T'\to \T^{H}$$ such that diagram \eqref{introdiagram'}
together with diagram \eqref{inr cover2} gives the following commutative diagram
\begin{equation}\label{intr main diagram}
\xymatrix{
\T \ar[dr]^-{\pi}\ar[rr]^-{i_{\T}} \ar[dd]^-{\pi_m} &&\T^{H}  \ar[dd]^-{\pi_{m}^{H}}\ar[rr]^-{\P^{H}}  &&D\ar[dd]^-{\pi_{D}}\\
&\T'\ar[ur]^-{\pi^{0}}\ar[dl]_-{\pi'_{m}}\ar[urrr]^-{\P'}&&&\\
\Z \ar[rr]^-{i} &&\ZZ \ar[rr]^-{\P_{\Z}^{H}} &&D/\Gamma .}
\end{equation}

With the injectivity of $\P^{H}$ and $\pi^{0}$, we have the global Torelli theorem on the Torelli space of Calabi--Yau manifolds.

\begin{theorem}\label{introcor2}
The global Torelli theorem holds on the Torelli space of Calabi--Yau manifolds, i.e.
the period map $$\P':\, \T'\to D$$ is injective. Moreover, the complete complex affine manifold $\mathcal{T}^H$ is the completion space of $\mathcal{T}'$ with respect to the Hodge metric, and it is a bounded domain in $A\simeq \mathbb{C}^N$.
\end{theorem}

For the above two theorems,  we remark that one technical difficulty of our arguments is to prove that $\mathcal{T}_m^H$ is independent of $m$,  and that the
image $\T_0=i_{\T}(\T)$ of $i_{\T}$ is biholomorphic to the Torelli space $\T'$. For this purpose we  introduce the smooth complete manifold $\mathcal{T}^H_{m}$
equipped with the Hodge metric. Moreover we prove the existence of an affine structure on $\TT$, which is given by pulling back the affine structure on $\C^N$
through the holomorphic map
$$\Psi_{m}^{H}:\, \TT \to A\cap D\subset A\simeq \C^{N},$$
where $\Psi_{m}^{H}=P\circ \PP$ is shown to be nondegenerate by using affine structures on $\TT$ and $A$, and the local Torelli theorem for Calabi--Yau manifolds.

Then by using the completeness of $\TT$ with the Hodge metric and Corollary 2 of Griffiths and Wolf in \cite{GW}, we show that $\Psi_{m}^{H}$ is a covering map. Finally,
we prove that $\Psi_{m}^{H}$ is injective by giving two proofs. The first proof is to directly prove that $A\cap D$ is simply connected. In fact we show that
$A\cap D$ is diffeomorphic to a Euclidean space as explicitly described at the beginning of the proof of Lemma \ref{abounded} following Harish-Chandra's proof.
The second proof is to use the affine structure and the completeness of $\TT$ with the Hodge metric, as well as the fact that $\Psi^{H}_{m}$ is an affine map, which is reduced to proving that the straight line segment in $\T^H_m$
connecting the base point to any other point must be mapped to the straight line segment connecting two different points in the Euclidean space $A\simeq \C^N$. From this we deduce that
for $m\geq 3$, $\T^H_m$ is biholomorphic to  $A\cap D$.

To show that $i_{\T}(\T)$ is biholomorphic to $\T'$, we first note that $$\mathcal{T}_0:=i_{\T}(\T)=i_m(\mathcal{T}),$$ for any $m\ge 3$, is a Zariski open submanifold of $\T^H_m=\T^H$. It is not hard to show that $i_{\mathcal{T}}: \,\mathcal{T}\to\mathcal{T}_0$ is a covering map. Moreover, we prove that $\T_{0}$ is also the covering space of the Torelli space $\T'$ with the covering map $\pi_{0}:\, \T_{0}\to \T'$, which fits into the commutative diagram
$$\xymatrix{
\T_{0}\ar[d]^-{\pi_{0}}\ar[rr]^-{\P^{H}|_{\T_{0}}}&& D\\
\T'\ar[urr]^-{\P'} &&.
}$$
Here the markings and level structures of the Calabi--Yau manifolds come into play substantially. Since $\P^{H}$ is injective, so is the restricted map $\P^{H}|_{\T_{0}}$. Therefore the covering map $\pi_{0}$ is also injective from the  commutative diagram above, and $\pi_{0}:\, \T_{0}\to \T'$ is a biholomoprhic map, which implies that $\P'$ is injective.

As a direct corollary, we have the global Torelli theorem on the moduli space $\Z$ of polarized Calabi--Yau manifolds with level $m$ structure for any $m\geq 3$.

\begin{corollary}\label{intr Global Torelli theorem'}
The global Torelli theorem holds on the moduli space $\Z$ of polarized Calabi--Yau manifolds with level $m$ structure, i.e. the period map $$\Phi_{\Z}:\, \Z \rightarrow D/\Gamma$$ is injective.
\end{corollary}

Moreover, we prove the following result about the geometry of the Hodge metric completion space $\T^H$.
\begin{theorem}The completion space $\mathcal{T}^H$ is a bounded domain of holomorphy in $\mathbb{C}^N$; thus there exists a complete K\"ahler--Einstein metric on $\mathcal{T}^H$.
\end{theorem}
To prove this theorem, we construct a plurisubharmonic exhaustion function on $\mathcal{T}^H$ by using Theorem \ref{general f}, the completeness of $\mathcal{T}^H$, and the injectivity of $\Phi^H$. This shows that $\mathcal{T}^H$ is a bounded domain of holomorphy in $\mathbb{C}^N$.
The existence of the K\"ahler-Einstein metric follows directly from a theorem in Mok--Yau in \cite{MokYau}. 

This paper is organized as follows. In Section \ref{TandPhi}, we review the definition of the period domain of polarized Hodge structures and briefly describe the constructions of the moduli space $\sZ$, the Teichm\"uller space $\T$ and the Torelli space $\T'$ of Calabi--Yau manifolds, as well as the definitions of the period maps and the Hodge metrics on the moduli space and the Teichm\"uller space respectively. We also introduce the Hodge metric completion space $\TT$ and study the extended map $\PP$ on  $\TT$.
In Section \ref{boundedness of period maps}, we show that the image of the period map $$\P:\, \T \to D$$ lies in $N_+\cap D$ as a bounded subset in the Euclidean space $N_+$ with Euclidean metric induced from the Hodge metric.

In Section \ref{Holomorphic affine}, we prove that there exists a global holomorphic affine structure on $\T$, as well as an affine structure on $\TT$.
Then we show that the extended period map $$\Phi^H_{m}: \,\mathcal{T}^H_{m}\rightarrow D$$ is injective. In Section \ref{completionspace}, we define the completion space $\mathcal{T}^H$ and the extended period map $\Phi^H$. We then show that $\mathcal{T}^H$ is the Hodge metric completion space of the Torelli space $\mathcal{T}'$, which is also a complex affine manifold, and that $\Phi^H$ is a holomorphic injection, which extends the period map $$\Phi': \,\mathcal{T}'
\rightarrow D.$$
From this the global Torelli theorems for polarized Calabi--Yau manifolds on the Torelli space and the moduli spaces with level $m$ structures for any $m\geq3$ follow directly. Finally, we prove that $\mathcal{T}^H$ is a bounded domain of holomorphy in $\mathbb{C}^N$, and thus it admits a complete K\"ahler--Einstein metric. In Appendix \ref{topological lemmas}, we include two elementary topological lemmas.  One lemma shows the existence of the choices of $\Phi^H_{m}$ and $i_m$ satisfying $\Phi=\Phi^H_{m}\circ i_{m}$, and the other lemma relates the fundamental group of the moduli space $\mathcal{Z}_m$ to that of its completion space $\mathcal{Z}^H_{m}$. 

In previous versions of this paper, we erroneously identified the Teichm\"uller space with the Torelli space by quoting a wrong lemma, which gives a wrong proof that the Torelli space is simply connected. We corrected this error, and will show that the Torelli space $\T'$ and its Hodge metric completion space $\T^H_m$ are the most natural moduli spaces to understand period maps and the global Torelli problems.

\subsubsection*{Acknowledgement} We would like to thank Professors Robert Friedman, Mark Green, Phillip Griffiths, Si Li, Bong Lian, Eduard Looijenga, Wilfried Schmid, Andrey Todorov, Veeravalli Varadarajan and Shing-Tung Yau for sharing many of their ideas.
\section{Moduli, Teichm\"uller, Torelli spaces and the period map}\label{TandPhi}
In Section \ref{section period domain}, we recall the definition and
some basic properties of the period domain. In Section \ref{section
construction of Tei}, we discuss the constructions of the
Teichm\"ulller space and the Torelli space of Calabi--Yau manifolds based on the works of
Popp \cite{Popp}, Viehweg \cite{v1} and Szendr\"oi \cite{sz} on the
moduli spaces of Calabi--Yau manifolds. In Section \ref{section
period map}, we define the period maps from the Teichm\"{u}ller space and the Torelli space
to the period domain.
In Section \ref{Hodge metric completion}, we define the Hodge metric completion spaces of Calabi--Yau manifolds and study the extended period map over the completion space.
We remark that most of the results in this
section are standard and can be found from the literature we refer in the
subjects.
\subsection{Period domain of polarized Hodge structures}\label{section period
domain}
We first review the construction of the period domain of
polarized Hodge structures.
We refer the reader to $\S3$ in \cite{schmid1} for more details.

A pair $(M,L)$ consisting of a Calabi--Yau manifold $M$ of complex dimension $n$ with $n\geq 3$ and an ample
line bundle $L$ over $M$ is called a {polarized Calabi--Yau
manifold}. By abusing notation, the Chern class of $L$ will also be
denoted by $L$ and thus $L\in H^2(M,\mathbb{Z})$.
We fix a lattice $\Lambda$ with a pairing $Q_{0}$, where $\Lambda$ is isomorphic to $H^n(M_{0},\mathbb{Z})/\text{Tor}$ for some Calabi--Yau manifold $M_{0}$ and $Q_{0}$ is defined by the cup-product.
For a polarized Calabi--Yau manifold $(M,L)$, we define a marking $\gamma$ as an isometry of the lattices
\begin{equation}\label{marking}
\gamma :\, (\Lambda, Q_{0})\to (H^n(M,\mathbb{Z})/\text{Tor},Q).
\end{equation}

\begin{definition}
Let the pair $(M,L)$ be a polarized Calabi--Yau manifold, then we call the
triple $(M,L,\gamma)$ a polarized and
marked Calabi--Yau manifold.
\end{definition}

For a polarized and marked Calabi--Yau manifold $M$ with background
smooth manifold $X$, the marking identifies
$H^n(M,\mathbb{Z})/\text{Tor}$ isometrically to the fixed lattice $\Lambda$. This
gives us a canonical identification of the middle dimensional de
Rham cohomology of $M$ to that of the background manifold $X$, that is,
\begin{equation*}
H^n(M,\mathbb{F})\cong H^n(X,\mathbb{F}),
\end{equation*}
where the coefficient ring $\mathbb{F}$ can be ${\mathbb{Q}}$, ${\mathbb{R}}$ or
${\mathbb{C}} $. Since the polarization $L$ is an integer class, it defines a map
\begin{equation*}
L:\, H^n(X,{\mathbb{Q}})\to H^{n+2}(X,{\mathbb{Q}}),\quad\quad A\mapsto L\wedge A.
\end{equation*}
We denote by $H_{pr}^n(X)=\ker(L)$ the primitive cohomology groups,
where the coefficient ring can also be ${\mathbb{Q}}$, ${\mathbb{R}}$ or ${\mathbb{C}
}$. We let $$H_{pr}^{k,n-k}(M)=H^{k,n-k}(M)
\cap H_{pr}^n(X,{\mathbb{C}})$$ and denote its dimension by
$h^{k,n-k}$. We have the Hodge decomposition
\begin{align}  \label{cl10}
H_{pr}^n(X,{\mathbb{C}})=H_{pr}^{n,0}(M)\oplus\cdots\oplus
H_{pr}^{0,n}(M),
\end{align}
such that $H_{pr}^{n-k,k}(M)=\bar{H_{pr}^{k,n-k}(M)}.$
It is easy to see that for a polarized Calabi--Yau manifold, since $H^2(M, {\mathcal O}_M)=0$, we have
$$H_{pr}^{n,0}(M)= H^{n,0}(M), \ H_{pr}^{n-1,1}(M)= H^{n-1,1}(M).$$
The Poincar\'e bilinear form $Q$ on $H_{pr}^n(X,{\mathbb{Q}})$ is
defined by
\begin{equation*}
Q(u,v)=(-1)^{\frac{n(n-1)}{2}}\int_X u\wedge v
\end{equation*}
for any $d$-closed $n$-forms $u,v$ on $X$.
Furthermore, $Q $ is nondegenerate and can be extended to
$H_{pr}^n(X,{\mathbb{C}})$ bilinearly. Moreover, it also satisfies the
Hodge-Riemann bilinear relations
\begin{eqnarray}
Q\left ( H_{pr}^{k,n-k}(M), H_{pr}^{l,n-l}(M)\right )=0
\text{ unless } k+l=n,\label{cl30}\\
\left (\sqrt{-1}\right )^{2k-n}Q\left ( v,\bar v\right )>0
\text{ for }v\in H_{pr}^{k,n-k}(M)\setminus\{0\}. \label{cl40}
\end{eqnarray}

Let $f^k=\sum_{i=k}^nh^{i,n-i}$, denote $f^0=m$, and
$$F^k=F^k(M)=H_{pr}^{n,0}(M)\oplus\cdots\oplus H_{pr}^{k,n-k}(M),$$
from which we have the decreasing filtration
$$H_{pr}^n(X,{\mathbb{C}})=F^0\supset\cdots\supset F^n.$$ We know that
\begin{align}
&\dim_{\mathbb{C}} F^k=f^k,  \label{cl45}\\
& H^n_{pr}(X,{\mathbb{C}})=F^{k}\oplus \bar{F^{n-k+1}}, \quad\text{and}\quad
H_{pr}^{k,n-k}(M)=F^k\cap\bar{F^{n-k}}.
\end{align}
In terms of the Hodge filtration, the Hodge-Riemann relations \eqref{cl30} and \eqref{cl40} are
\begin{align}
& Q\left ( F^k,F^{n-k+1}\right )=0, \quad\text{and}\quad\label{cl50}\\
& Q\left ( Cv,\bar v\right )>0 \quad\text{if}\quad v\ne 0,\label{cl60}
\end{align}
where $C$ is the Weil operator given by $Cv=\left (\sqrt{-1}\right
)^{2k-n}v$ for $v\in H_{pr}^{k,n-k}(M)$. The period domain $D$
for polarized Hodge structures with data \eqref{cl45} is the space
of all such Hodge filtrations
\begin{equation*}
D=\left \{ F^n\subset\cdots\subset F^0=H_{pr}^n(X,{\mathbb{C}})\mid %
\eqref{cl45}, \eqref{cl50} \text{ and } \eqref{cl60} \text{ hold}
\right \}.
\end{equation*}
The compact dual $\check D$ of $D$ is
\begin{equation*}
\check D=\left \{ F^n\subset\cdots\subset F^0=H_{pr}^n(X,{\mathbb{C}})\mid %
\eqref{cl45} \text{ and } \eqref{cl50} \text{ hold} \right \}.
\end{equation*}
The period domain $D\subseteq \check D$ is an open subset.
From the definition of period domain we naturally get
the {Hodge bundles} on $\check{D}$ by associating to each point
in $\check{D}$ the vector spaces $\{F^k\}_{k=0}^n$ in the Hodge
filtration of that point. Without confusion we will also denote by $F^k$
the bundle with $F^k$ as the fiber for each $0\leq k\leq n$.
\begin{remark}Here we would like to remark the notation change for the primitive cohomology groups.
As mentioned above, for a polarized Calabi--Yau manifold we have
$$H_{pr}^{n,0}(M)= H^{n,0}(M), \ H_{pr}^{n-1,1}(M)= H^{n-1,1}(M).$$
For simplicity, we will also use $H^n(M,\mathbb{C})$ and $H^{k, n-k}(M)$ to denote the primitive cohomology groups $H^n_{pr}(M,\mathbb{C})$ and $H_{pr}^{k, n-k}(M)$ respectively. Moreover, cohomology will mean primitive cohomology in the rest of the paper.
\end{remark}
\subsection{Moduli, Teichm\"uller  and Torelli spaces}\label{section construction of Tei}

We first recall the concept of universal family of compact complex
manifolds in deformation theory. We refer to page~8-10 in \cite{su} and page~94 in \cite{Popp}
for equivalent definitions and more details.

A family of compact complex manifolds $\pi:\,\mathcal{W}\to\mathcal{B}$ is called \textit{versal} at a point $p\in \mathcal{B}$ if it satisfies the following conditions:
\begin{enumerate}
\item If given a complex analytic family $\iota:\,\mathcal{V}\to \mathcal{S}$ of compact complex manifolds with a point $s\in \mathcal{S}$ and a biholomorphic map $$f_0:\,V=\iota^{-1}(s)\to U=\pi^{-1}(p),$$ then there exists a holomorphic map $g$ from a neighbourhood $\mathcal{N}\subseteq \mathcal{S}$ of the point $s$ to $\mathcal{B}$ and a holomorphic map $f:\,\iota^{-1}(\mathcal{N})\to \mathcal{W}$ with $\iota^{-1}(\mathcal{N})\subseteq \mathcal{V}$ such that they satisfy that $g(s)=p$ and $f|_{\mathcal{\iota}^{-1}(s)}=f_0$
with the following commutative diagram
\[\xymatrix{\iota^{-1}(\mathcal{N})\ar[r]^{f}\ar[d]^{\iota}&\mathcal{W}\ar[d]^{\pi}\\
\mathcal{N}\ar[r]^{g}&\mathcal{B}.
}\]
\item For all $g$ satisfying the above condition, the tangent map $(dg)_s$ is uniquely determined.
\end{enumerate}
The family $\pi:\,\mathcal{W}\to\mathcal{B}$ is called \textit{universal} at a point $p\in \mathcal{B}$ if (1) is satisfied and (2) is replaced by
\begin{itemize}
\item[(2')] The map $g$ is uniquely determined.
\end{itemize}
If a family $\pi:\,\mathcal{W}\to\mathcal{B}$ is versal (universal resp.) at every point $p\in\mathcal{B}$, then it is called a \textit{versal family} (\textit{universal family} resp.) on $\mathcal{B}$.


Let $(M,L)$ be a polarized Calabi--Yau manifold.
Recall that a marking of $(M,L)$ is defined as an isometry
$$\gamma :\, (\Lambda, Q_{0})\to (H^n(M,\mathbb{Z})/\text{Tor},Q).$$
For any integer $m\geq 3$, we follow the definition of Szendr\"oi \cite{sz}
 to define an $m$-equivalent relation of two markings on $(M,L)$ by
$$\gamma\sim_{m} \gamma' \text{ if and only if } \gamma'\circ \gamma^{-1}-\text{Id}\in m \cdot\text{End}(H^n(M,\mathbb{Z})/\text{Tor}),$$
and denote $[\gamma]_{m}$ to be the set of all the equivalent classes of such $\gamma$.
Then we call $[\gamma]_{m}$ a level $m$
structure on the polarized Calabi--Yau manifold.

Two triples $(M, L, [\gamma]_{m})$ and $(M', L', [\gamma']_{m})$ are equivalent if there exists a biholomorphic map $f:\,M\to M'$ such that
\begin{align*}
f^*L'&=L,\\
f^*\gamma' &\sim_{m}\gamma,
\end{align*}
where $f^*\gamma'$ is given by $$\gamma':\, (\Lambda, Q_{0})\to (H^n(M',\mathbb{Z})/\text{Tor},Q)$$ composed with
$$f^*:\, (H^n(M',\mathbb{Z})/\text{Tor},Q)\to (H^n(M,\mathbb{Z})/\text{Tor},Q).$$
We denote by $[M, L, [\gamma]_{m}]$ the equivalent class of the polarized Calabi--Yau manifolds with level $m$ structure $(M, L, [\gamma]_{m})$.

For deformation of
polarized Calabi--Yau manifold with level $m$ structure, we reformulate Theorem~2.2 in
\cite{sz} as the following theorem, in which we only put the statements we need in this paper. One can
also look at \cite{Popp} and \cite{v1} for more details about the
construction of moduli spaces of Calabi--Yau manifolds.

\begin{theorem}\label{Szendroi theorem 2.2}
Let $[M,L, [\gamma]_{m}]$ be a polarized Calabi--Yau manifold with level
$m$ structure $[\gamma]_{m}$ for $m\geq 3$. Then there exists a connected quasi-projective complex manifold
$\mathcal{Z}_m$ with a universal family of Calabi--Yau manifolds,
\begin{align}\label{szendroi versal family}
f_{m}:\,\mathcal{X}_{\mathcal{Z}_m}\rightarrow \mathcal{Z}_m,
\end{align}
which contains $[M,L, [\gamma]_{m}]$ as a fiber and is polarized by an ample line bundle
$\mathcal{L}_{\mathcal{Z}_m}$ on $\mathcal{X}_{\mathcal{Z}_m}$.
\end{theorem}

As discussed in \cite{sz}, $\Z$ is a smooth irreducible component of the moduli space of polarized Calabi--Yau manifolds with level $m$ structure.  For $m\ge 3$, we denote by $\T^{m}$ the universal covering space of $\Z$, with the covering map $\pi^{m}:\, \T^{m}\to \Z$. Then we can pull-back the universal family $f_{m}:\,\mathcal{X}_{\mathcal{Z}_m}\rightarrow \mathcal{Z}_m$ to get an analytic family $\phi^{m}:\, \U^{m}\to \T^{m}$ via the covering map $\pi^{m}$, which is also universal since universal family is a local property.

We claim that $\T^{m}$ is independent of the choice of $m$. In fact,
let $m_1, m_2$ be two different integers
$\ge3$, and let $\T^{m_1}$ and $\T^{m_2}$ be the corresponding universal covering spaces with the universal families $$\phi^{m_1}:\,
\U^{m_1} \to \T^{m_1} \  \ \text{and}\ \  \phi^{m_2}:\, \U^{m_2} \to \T^{m_2}$$
respectively. Then for any point $p\in \T^{m_1}$ and the fiber
$M_p=(\phi^{m_1})^{-1}(p)$ over $p$, there exists $q\in \T^{m_2}$ such
that $N_q=(\phi^{m_2})^{-1}(q)$ is biholomorphic to $M_p$. 

By the
definition of universal family, we can find a local neighborhood
$U_p$ of $p$ and a holomorphic map $$h_p:\, U_p\to \T^{m_2},$$
$p\mapsto q$ such that the map $h_p$ is uniquely determined. Since
$\T^{m_1}$ is simply-connected, all the local holomorphic maps $$\{
h_p:\, U_p\to \T^{m_2},\, p\in \T^{m_1}\}$$ patches together to give
a global holomorphic map $h:\, \T^{m_1}\to \T^{m_2}$ which is
well-defined. Moreover $h$ is unique since it is unique on each
local neighborhood of $\T^{m_1}$. Similarly we have a holomorphic
map $h':\, \T^{m_2}\to \T^{m_1}$ which is also unique. Then $h$ and
$h'$ are inverse to each other by the uniqueness of $h$ and $h'$.
Therefore $\T^{m_1}$ and $\T^{m_2}$ are biholomorphic.

From now on we denote by $\T$ the universal covering space of $\Z$ for any $m\ge 3$, as it is independent of $m$. We call $\T$ the {\em Teichm\"uller space of Calabi--Yau manifolds}.
We also denote by $\phi:\,\U\rightarrow \mathcal{T}$ the pull-back family of the family (\ref{szendroi versal family}) via the covering $\pi_m:\, \T \to \Z$.
In summary, we have proved the following proposition.

\begin{proposition}\label{imp}
The Teichm\"uller space $\mathcal{T}$ of Calabi--Yau manifolds is a connected and simply connected complex
manifold, and the family
\begin{align}\label{versal family over Teich}
\phi:\,\mathcal{U}\rightarrow\mathcal{T}
\end{align}
which contains $M$ as a fiber, is a universal family.
\end{proposition}

We remark that the family $\phi:\, \mathcal{U}\rightarrow \mathcal{T}$ being universal at each point is essentially due to the local Torelli theorem for Calabi--Yau manifolds. In fact, the Kodaira-Spencer map of the family $\mathcal{U}\rightarrow \mathcal{T}$
\begin{align*}
\kappa:\,T^{1,0}_p\mathcal{T}\to {H}^{0,1}(M_p,T^{1,0}M_p),
\end{align*}
is an isomorphism for each $p\in\mathcal{T}$. Then by theorems in page~9 of \cite{su}, we conclude that $\mathcal{U}\to \mathcal{T}$ is versal at each $p\in \mathcal{T}$.
Since
$$H^{0}(M,\Theta_{M})=H^{0}(M,\Omega^{n-1}_{M})=H^{{n-1,0}}(M)=0,$$
we conclude from Theorem 1.6 of \cite{su} that $\mathcal{U}\to \mathcal{T}$ is universal at each $p\in \mathcal{T}$.
Moreover, the well-known Bogomolov-Tian-Todorov result (\cite{tian1} and \cite{tod1}) implies that $$\dim_{\mathbb{C}}(\mathcal{T})=N=h^{n-1,1}.$$ We refer the reader to
Chapter 4 in \cite{km1} for more details about deformation of complex structures and the Kodaira-Spencer map. 

Recall that a polarized and marked Calabi--Yau manifold is a triple $(M, L, \gamma)$, where $M$ is a Calabi--Yau manifold, $L$ is a polarization on $M$, and $\gamma$ is a marking $$\gamma :\, (\Lambda, Q_{0})\to (H^n(M,\mathbb{Z})/\text{Tor},Q).$$
Two triples $(M, L, \gamma)$ and $(M', L', \gamma')$ are equivalent if there exists a biholomorphic map $f:\,M\to M'$ with
\begin{align*}
f^*L'&=L,\\
f^*\gamma' &=\gamma,
\end{align*}
where $f^*\gamma'$ is given by $$\gamma':\, (\Lambda, Q_{0})\to (H^n(M',\mathbb{Z})/\text{Tor},Q)$$ composed with
$$f^*:\, (H^n(M',\mathbb{Z})/\text{Tor},Q)\to (H^n(M,\mathbb{Z})/\text{Tor},Q).$$
We denote by $[M, L, \gamma]$ the equivalent class of the polarized and marked Calabi--Yau manifold $(M, L, \gamma)$.

In this paper, we define the Torelli space as follows.

\begin{definition}
The Torelli space $\T'$ of Calabi--Yau manifolds is the irreducible smooth component of the moduli space of the equivalent classes of polarized and marked Calabi--Yau manifolds, which contains $[M, L,\gamma]$.
\end{definition}
By mapping $[M, L, \gamma]$ to $[M, L, [\gamma]_{m}]$, we have a natural covering map $$\pi_m':\,\mathcal{T}'\to\mathcal{Z}_m.$$ From this we see easily that $\T'$ is a smooth and connected complex manifold.
We also get a pull-back universal family $\phi':\, \U'\to \T'$ on the Torelli space $\T'$ via the covering map $\pi_m'$.

Recall that we have defined the Teichm\"uller space $\T$ to be the universal covering space of $\Z$ with covering map $\pi_{m}:\, \T_{m}\to \Z$. Therefore we can lift $\pi_{m}$ via the covering map $\pi_m':\,\mathcal{T}'\to\mathcal{Z}_m$ to get a covering map $\pi:\, \T\to \T'$, such that the following diagram commutes.
\begin{equation}\label{cover1}
\xymatrix{
\T \ar[dr]^-{\pi} \ar[dd]^-{\pi_m} &\\
&\T'\ar[dl]^-{\pi_m'}\\
\sZ &\ \ .}
\end{equation}

\subsection{The period maps}\label{section period map}
For the family $f_m : \mathcal{X}_{\mathcal{Z}_{m}} \to \mathcal{Z}_{m}$, we denote each fiber by $$[M_s,L_{s},[\gamma_{s}]_{m}]=f_m^{-1}(s)$$ and $F_s^k=F^k(M_s)$ for any $s\in \mathcal{Z}_{m}$. With some fixed point $s_0\in \mathcal{Z}_m$, the period map is defined as a morphism $\Phi_{\mathcal{Z}_m} :\mathcal{Z}_m \to D/\Gamma$ by
\begin{equation}\label{perioddefinition}
s \mapsto \tau^{[\gamma_s]}(F^n_s\subseteq\cdots\subseteq F^0_s)\in D,
\end{equation}
where $\tau^{[\gamma_s]}$ is an isomorphism between the complex vector spaces
$$\tau^{[\gamma_s]}:\, H^n(M_s,\mathbb{C})\to H^n(M_{s_0},\mathbb{C}),$$
which depends only on the homotopy class $[\gamma_s]$ of the curve $\gamma_s$ between $s$ and $s_0$. Then the period map is well-defined with respect to the monodromy representation
$$\rho : \pi_1(\mathcal{Z}_m)\to \Gamma \subseteq \text{Aut}(H_{\mathbb{Z}},Q).$$ It is well-known that the period map has the following properties:
\begin{enumerate}
\item locally liftable;
\item holomorphic, i.e. $\partial F^i_z/\partial \bar{z}\subset F^i_z$, $0\le i\le n$;
\item Griffiths transversality: $\partial F^i_z/\partial z\subset F^{i-1}_z$, $1\le i\le n$.
\end{enumerate}

From (1) and the fact that $\T$ is the universal cover of $\sZ$, we can lift the period map to $\Phi : \T \to D$ such that the diagram
$$\xymatrix{
\T \ar[r]^-{\Phi} \ar[d]^-{\pi_m} & D\ar[d]^-{\pi}\\
\mathcal{Z}_m \ar[r]^-{\Phi_{\mathcal{Z}_m}} & D/\Gamma
}$$
is commutative.

From the definition of marking in \eqref{marking}, we also have a well-defined period map $\P':\, \T'\to D$ from the Torelli space $\T'$ by defining
\begin{equation}\label{defn of P'}
p \mapsto \gamma_{p}^{-1}(F^n_p\subseteq\cdots\subseteq F^0_p)\in D,
\end{equation}
where the triple $[M_{p}, L_{p}, \gamma_{p}]$ is the fiber over $p\in \T'$ of the analytic family $\mathcal{U}'\to \T'$,  and the marking $\gamma_{p}$ is an
isometry from a fixed lattice $\Lambda$ to $H^{n}(M_{p},\mathbb{Z})/\text{Tor}$, which extends $\C$-linearly to an isometry from $H=\Lambda\otimes_{\mathbb{Z}}\C$
to $H^{n}(M_{p},\mathbb{C})$. Here 
$$\gamma_{p}^{-1}(F^n_p\subseteq\cdots\subseteq F^0_p)=\gamma_{p}^{-1}(F^n_p)\subseteq\cdots\subseteq\gamma_{p}^{-1}(F^0_p)=H$$
denotes a Hodge filtration of $H$.

Then we have the following commutative diagram
\begin{align}\label{periods}
\xymatrix{
\T \ar[dr]^-{\pi}\ar[rr]^-{\P} \ar[dd]^-{\pi_m} && D\ar[dd]^-{\pi_{D}}\\
&\T'\ar[ur]^-{\P'}\ar[dl]_-{\pi'_{m}}&\\
\Z \ar[rr]^-{\Phi_{\Z}} && D/\Gamma,
}
\end{align}
where the maps $\pi_{m}$, $\pi'_{m}$ and $\pi$ are all natural covering maps between the corresponding spaces as in \eqref{cover1}.

Before closing this section, we prove a lemma concerning the monodromy group $\Gamma$.
\begin{lemma}\label{trivial monodromy}
Let $\gamma$ be the image of some element of $\pi_1(\mathcal{Z}_m)$ in $\Gamma$ under the monodromy representation. Suppose that $\gamma$ is finite, then $\gamma$ is trivial. Therefore for $m\geq 3$, we can assume that $\Gamma$ is torsion-free and $D/\Gamma$ is smooth.
\end{lemma}
\begin{proof}
Let us look at the period map locally as $\Phi_{\mathcal{Z}_m}:\,\Delta^* \to D/\Gamma$. Assume that $\gamma$ is the monodromy action corresponding to the generator of the fundamental group of $\Delta^*$. We lift the period map to $\Phi :\, {\mathbb H}\to D$, where ${\mathbb H}$ is the upper half plane and the covering map from ${\mathbb H}$ to $\Delta^*$ is
$$z\mapsto \exp(2\pi \sqrt{-1} z).$$
Then $\Phi(z+1)=\gamma\Phi(z)$ for any $z\in {\mathbb H}$. Since $\Phi(z+1)$ and $\Phi(z)$ correspond to the same point in $\mathcal{Z}_m$, by the definition of $\mathcal{Z}_m$ we have
$$\gamma\equiv \text{I} \text{ mod }(m).$$
But $\gamma$ is also in $\text{Aut}(H_\mathbb{Z})$, applying Serre's lemma \cite{Serre60} or Lemma 2.4 in \cite{sz}, we have $\gamma=\text{I}$.
\end{proof}

\subsection{Hodge metric completion and extended period maps}\label{Hodge metric completion}

By the work of Viehweg in \cite{v1}, we know that $\mathcal{Z}_m$ is quasi-projective and consequently we can find a smooth projective compactification $\bar{\mathcal{Z}}_{m}$ such that $\mathcal{Z}_m$ is Zariski open in $\bar{\mathcal{Z}}_{m}$ and the complement $\bar{\mathcal{Z}}_{m}\backslash\mathcal{Z}_m$ is a divisor of normal crossings. Therefore, $\mathcal{Z}_m$ is dense and open in $\bar{\mathcal{Z}}_{m}$ with the complex codimension of the complement $\bar{\mathcal{Z}}_{m}\backslash \mathcal{Z}_m$ at least one. 

In \cite{GS}, Griffiths and Schmid studied the {Hodge metric} on the period domain $D$. We denote it by $h$. In particular, this Hodge metric is a complete homogeneous metric induced by the Killing form.
By local Torelli theorem for Calabi--Yau manifolds, we know that the period maps $\Phi_{\mathcal{Z}_m}, \Phi$ both have nondegenerate differentials everywhere.
Thus it follows from \cite{GS} that the pull-backs of $h$ by $\Phi_{\mathcal{Z}_m}$ and $\Phi$ to $\mathcal{Z}_m$ and $\mathcal{T}$ respectively are both well-defined K\"ahler metrics. By abuse of notation, we still call these pull-back metrics the {Hodge metrics}. 
Let us denote $\mathcal{Z}^H_{m}$ to be the completion of $\mathcal{Z}_m$ with respect to the Hodge metric. By definition $\mathcal{Z}_m^H$ is the smallest complete space with respect to the Hodge metric that contains $\mathcal{Z}_m$. Then $\mathcal{Z}^H_{m}\subseteq\bar{\mathcal{Z}}_{m}$ and the complex codimension of the complement $\mathcal{Z}^H_{m}\backslash \mathcal{Z}_m$ is at least one.

Now we recall some basic properties about metric completion space we will use in this paper. We know that the metric completion space of a connected space is still connected. Therefore, $\mathcal{Z}_m^H$ is connected.

Suppose $(X, d)$ is a metric space with metric $d$. Then the metric completion space of $(X, d)$ is unique in the following sense: if $\bar{X}_1$ and $\bar{X}_2$ are complete metric spaces that both contain $X$ as a dense subset, then there exists an isometry $$f: \bar{X}_1\rightarrow \bar{X}_2$$ such that $f|_{X}$ is the identity map on $X$. Moreover, the metric completion space $\bar{X}$ of $X$ is the smallest complete metric space containing $X$ in the sense that any other complete space that contains $X$ as a subspace must also contains $\bar{X}$ as a subspace.
Hence the Hodge metric completion space $\mathcal{Z}_m^H$ is unique up to isometry, although the compact space $\bar{\mathcal{Z}}_{m}$ may not be unique. This means that our definition of $\mathcal{Z}_m^H$ is intrinsic.

Moreover, suppose $\bar{X}$ is the metric completion space of the metric space $(X, d)$. If there is a continuous map $f: \,X\rightarrow Y$ which is a local isometry with $Y$ a complete space, then there exists a continuous extension $\bar{f}:\,\bar{X}\rightarrow{Y}$ such that $\bar{f}|_{X}=f$. Since $D/\Gamma$ together with the Hodge metric $h$ is complete, we can extend the period map to a continuous map $$\Phi_{\mathcal{Z}_m}^H:\, \mathcal{Z}_m^H\to D/\Gamma.$$

In order to understand $\mathcal{Z}^H_{m}$ and the extended period map $\Phi_{\mathcal{Z}_m}^H$ well, we introduce another two extensions of $\mathcal{Z}_m$, which will be proved to be biholomorphic to $\mathcal{Z}^H_{m}$.

Let $\mathcal{Z}_m'\supseteq \mathcal{Z}_m$ be the maximal subset of $\bar{\mathcal Z}_m$ to which the period map $\Phi_{\mathcal{Z}_m}:\, \mathcal{Z}_m\to D/\Gamma$ extends continuously and let $\Phi_{\mathcal{Z}_m'} : \, \mathcal{Z}_m'\to D/\Gamma$ be the extended map. Then one has the commutative diagram
\begin{equation*}
\xymatrix{
\mathcal{Z}_m \ar@(ur,ul)[r]+<16mm,4mm>^-{\Phi_{\mathcal{Z}_m}}\ar[r]^-{i} &\mathcal{Z}_m' \ar[r]^-{\Phi_{\mathcal{Z}_m'}} & D/\Gamma.
}
\end{equation*}
with $i :\, \mathcal{Z}_m\to \mathcal{Z}_m'$ the inclusion map.

Since $\bar{\mathcal Z}_m\setminus \mathcal{Z}_m$ is a divisor with simple normal crossings, for any point in $\bar{\mathcal Z}_m\setminus \mathcal{Z}_m$ we can find a neighborhood $U$ of that point, which is isomorphic to a polycylinder $\Delta^n$, such that
$$U\cap \mathcal{Z}_m\simeq (\Delta^*)^k\times \Delta^{N-k}.$$

Let $T_i$, $1\le i\le k$ be the image of the $i$-th fundamental group of $(\Delta^*)^k$ under the monodromy representation, then the $T_i$'s are called the Picard-Lefschetz transformations.
Let us define the subspace $\mathcal{Z}_m''\subset \bar{\mathcal{Z}}_{m}$ which contains $\mathcal{Z}_m$ and the points in $\bar{\mathcal{Z}}_{m}\setminus \mathcal{Z}_m$ around which the Picard-Lefschetz transformations are of finite order, hence trivial by Lemma \ref{trivial monodromy}.

With the above preparations, we are ready to prove the following lemma.

\begin{lemma}\label{cidim} We have 
$\mathcal{Z}_m'=\mathcal{Z}_m''=\mathcal{Z}_m^H$ which is an open complex submanifold of $\bar{\mathcal Z}_m$ with $\text{codim}_{\C}(\bar{\mathcal Z}_m\setminus \mathcal{Z}_m^H)\ge 1$.
The subset $\mathcal{Z}_m^H\setminus \mathcal{Z}_m$ consists of the points around which the Picard-Lefschetz transformations are trivial.
Moreover the extended period map $$\Phi_{\mathcal{Z}_m}^H:\, \mathcal{Z}_m^H\to D/\Gamma$$ is proper and holomorphic.
\end{lemma}
\begin{proof}
From Theorem 9.6 in \cite{Griffiths3} and its proof, or Corollary 13.4.6 in \cite{CMP}, we know that $\mathcal{Z}_m''$ is open and dense in $\bar{\mathcal{Z}}_{m}$ and the period map $\Phi_{\mathcal{Z}_m}$ extends to a holomorphic map $$\Phi_{\mathcal{Z}_m''} :\, \mathcal{Z}_m'' \to D/\Gamma$$ which is proper.
In fact, as proved in Theorem 3.1 of \cite{sz}, which follows directly from Propositions 9.10 and 9.11 of \cite{Griffiths3}, $\bar{\mathcal{Z}}_{m}\backslash \Z''$ consists of the components of divisors in $\bar{\mathcal{Z}}_{m}$ whose Picard-Lefschetz transformations are of infinite order, and therefore
$\Z''$ is a Zariski open submanifold in $\bar{\mathcal{Z}}_{m}$.
Hence by the definition of $\Z'$, we know that $\Z''\subseteq \Z'$.

Conversely, for any point $q\in \mathcal{Z}_m'$ with image $u=\Phi_{\mathcal{Z}'_m}(q)\in D/\Gamma$, we can choose the points $q_{k}\in \mathcal{Z}_m$, $k=1,2,\cdots$ such that $q_{k}\longrightarrow q$ with images $u_{k}=\Phi_{\mathcal{Z}_m}(q_{k})\longrightarrow u$ as $k\longrightarrow \infty$. Since $\Phi_{\mathcal{Z}_m''}:\, \mathcal{Z}_m''\to D/\Gamma$ is proper, the sequence $$\{q_{k}\}_{k=1}^{\infty}\subset (\Phi_{\mathcal{Z}_m''})^{-1}(\{u_{k}\}_{k=1}^{\infty})$$ has a limit point $q$ in $\mathcal{Z}_m''$, that is to say $q\in \mathcal{Z}_m''$ and $\mathcal{Z}_m'\subseteq \mathcal{Z}_m''$.

Therefore we have proved that $\mathcal{Z}_m'=\mathcal{Z}_m''$ and $\mathcal{Z}_m'\setminus \mathcal{Z}_m$ consists of the points around which the Picard-Lefschetz transformations are trivial. From Theorem (9.5) in \cite{Griffiths3}, we get that the extended period map $$\Phi_{\mathcal{Z}'_m}:\, \mathcal{Z}_m'\to D/\Gamma$$ is a proper holomorphic mapping, which is called Griffiths extension of $\Phi_{\mathcal{Z}_m}$ by Sommese in \cite{Sommese}. Now we only need to prove that $\mathcal{Z}_m'=\mathcal{Z}_m^H$.

Since we already have the extension $$\Phi_{\mathcal{Z}_m}^{H}:\, \mathcal{Z}_m^H\to D/\Gamma,$$ we see that $\mathcal{Z}_m^H\subseteq \mathcal{Z}_m'$ by the definition of $\mathcal{Z}_m'$. Conversely the points in $\mathcal{Z}_m'\setminus \mathcal{Z}_m$ are mapped into $D/\Gamma$, hence have finite distance from some fixed point in $\mathcal{Z}_m$ with respect to the Hodge metric. Therefore $\mathcal{Z}_m'\subseteq \mathcal{Z}_m^H$.
\end{proof}

%

Let $\mathcal{T}^{H}_{m}$ be the universal cover of $\mathcal{Z}^H_{m}$ with the universal covering map $$\pi_{m}^H:\, \mathcal{T}^{H}_{m}\rightarrow \mathcal{Z}_m^H.$$ Thus $\mathcal{T}^H_{m}$ is a connected and simply connected complete complex manifold with respect to the Hodge metric. We will call $\mathcal{T}^H_{m}$ the {Hodge metric completion space with level $m$ structure}.
Recall that the Teichm\"uller space $\mathcal{T}$ is the universal cover of the moduli space $\mathcal{Z}_m$ with the universal covering map denoted by $\pi_m:\, \mathcal{T}\to \mathcal{Z}_m$.
Thus we have the following commutative diagram
\begin{align}\label{cover maps} \xymatrix{\mathcal{T}\ar[r]^{i_{m}}\ar[d]^{\pi_m}&\mathcal{T}^H_{m}\ar[d]^{\pi_{m}^H}\ar[r]^{{\Phi}^{H}_{m}}&D\ar[d]^{\pi_D}\\
\mathcal{Z}_m\ar[r]^{i}&\mathcal{Z}^H_{m}\ar[r]^{{\Phi}_{{\mathcal{Z}_m}}^H}&D/\Gamma,
}
\end{align}
where $i$ is the inclusion map, ${i}_{m}$ is a lifting map of $i\circ\pi_m$, $\pi_D$ is the covering map and ${\Phi}^{H}_{m}$ is a lifting map of ${\Phi}_{{\mathcal{Z}_m}}^H\circ \pi_{m}^H$. In particular, $\Phi^H_{m}$ is a continuous map from $\mathcal{T}^H_{m}$ to $D$.

We notice that the lifting maps ${i}_{{\mathcal{T}}}$ and ${\Phi}^H_{m}$ are not unique, but Lemma \ref{choice} in the appendix of this paper implies that there exist suitable choices of $i_{m}$ and $\Phi_{m}^H$ such that $\Phi=\Phi^H_{m}\circ i_{m}$. We will fix the choices of $i_{m}$ and $\Phi^{H}_m$ such that $\Phi=\Phi^H_{m}\circ i_m$ in the rest of the paper. Without confusion of notations, we denote $\mathcal{T}_m:=i_{m}(\mathcal{T})$ and the restriction map $\Phi_m=\Phi^H_{m}|_{\mathcal{T}_m}$. Then we also have $\Phi=\Phi_m\circ i_m$.
\begin{proposition}\label{opend}The image $\mathcal{T}_m$ equals to the preimage $(\pi^H_{m})^{-1}(\mathcal{Z}_{m})$, and $i_m: \, \T \to \T_m$ is a covering map.
\end{proposition}
\begin{proof}
From diagram (\ref{cover maps}), we have that $\pi^H_{m}(i_m(\mathcal{T}))=i(\pi_m(\mathcal{T}))=\mathcal{Z}_{m}.$
Therefore, $\mathcal{T}_m=i_m(\mathcal{T})\subseteq (\pi^H_{m})^{-1}(\mathcal{Z}_{m})$. For the other direction, we need to show that for any point $q\in (\pi^H_{m})^{-1}(\mathcal{Z}_{m})\subseteq\mathcal{T}^H_{m}$, we have $$q\in i_m(\mathcal{T})=\mathcal{T}_m.$$

Let $p=\pi^H_{m}(q)\in \mathcal{Z}_{m}$, Let $x_1\in \pi_m^{-1}(p)\subseteq \mathcal{T}$ be an arbitrary point, then $\pi^H_{m}(i_m(x_1))=i(\pi_m(x_1))=p$ and hence $i_m(x_1)\in (\pi^H_{m})^{-1}(p)\subseteq \mathcal{T}^H_{m}$.

As $\mathcal{T}^H_{m}$ is a connected complex manifold, $\mathcal{T}^H_{m}$ is path connected. Therefore, for $i_m(x_1), q\in \mathcal{T}^H_{m}$, there exists a curve $\gamma:\,[0,1]\to\mathcal{T}^H_{m}$ with $\gamma(0)=i_m(x_1)$ and $\gamma(1)=q$. Then the composition $\pi^H_{m}\circ \gamma$ gives a loop on $\mathcal{Z}^H_{m}$ with $\pi^H_{m}\circ \gamma(0)=\pi^H_{m}\circ \gamma(1)=p$. Lemma \ref{fund} in the appendix implies that there is a loop $\Gamma$ on $\mathcal{Z}_m$ with $\Gamma(0)=\Gamma(1)=p$ such that
\begin{align*}
[i\circ \Gamma]=[\pi^H_{m}\circ \gamma]\in \pi_1(\mathcal{Z}^H_{m}),
\end{align*} where $\pi_1(\mathcal{Z}^H_{m})$ denotes the fundamental group of $\mathcal{Z}^H_{m}$.

Because $\mathcal{T}$ is the universal cover of $\mathcal{Z}_m$, there is a unique lifting map $\tilde{\Gamma}:\,[0,1]\to \mathcal{T}$ with $\tilde{\Gamma}(0)=x_1$ and $\pi_m\circ\tilde{\Gamma}=\Gamma$.
Again since $\pi^H_{m}\circ i_m=i\circ \pi_m$, we have
\begin{align*}
\pi^H_{m}\circ i_m\circ \tilde{\Gamma}=i\circ \pi_m\circ\tilde{\Gamma}=i\circ\Gamma:\,[0,1]\to \mathcal{Z}_m.
\end{align*}
Therefore $[\pi^H_{m}\circ i_m\circ \tilde{\Gamma}]=[i\circ \Gamma]\in \pi_1(\mathcal{Z}_m)$, and the two curves $i_m\circ \tilde{\Gamma}$ and $\gamma$ have the same starting points $i_m\circ \tilde{\Gamma}(0)=\gamma(0)=i_m(x_1)$. Then the homotopy lifting property of the covering map $\pi^H_{m}$ implies that $i_m\circ \tilde{\Gamma}(1)=\gamma(1)=q$. Therefore, $q\in i_m(\mathcal{T})$, as needed.

To show that $i_m$ is a covering map, note that for any small enough open neighborhood $U$ in $\T_{m}$, the restricted map $$\pi_m^H|_{U} :\, U \to V=\pi_m^H(U)\subset \Z$$ is biholomorphic, and there exists a disjoint union $\cup_{i}V_{i}$ of open subsets in $\T$ such that $\cup_{i}V_{i}=(\pi_{m})^{-1}(V)$ and $\pi_{m}|_{V_{i}}:\, V_{i}\to V$ is biholomorphic. From the commutativity of the diagram \eqref{cover maps}, we have that $\cup_{i}V_{i}=(i_{m})^{-1}(U)$ and $$i_{m}|_{V_{i}}:\, V_{i}\to U$$ is biholomorphic. Therefore $i_m:\, \T \to \T_m$ is also a covering map.
\end{proof}

We remark  that as the lifts of the holomorphic maps $i$ and $\Phi_{\mathcal{Z}_m}^H$ to universal covers, both $i_m$ and ${\Phi}^{H}_{m}$ are easily seen to be holomorphic maps.   This fact can also be proved  by using Theorem 9.6 in \cite{Griffiths3} and the Riemann extension theorem.

Actually we have proved that $i_m$ is a covering map, so it is holomorphic. Another proof to show $i_m$ is holomorphic is to use the geometric structures of $\ZZ$ which we briefly describe as follows.

Proposition \ref{opend} implies that $\T_m$ is an open complex submanifold of $\T_{m}^{H}$ and $\text{codim}_{\C}(\T_{m}^{H}\setminus \T_m)\ge 1$.
Since $\Phi=\Phi^H_{m}\circ i_m$ is holomorphic, $$\Phi_{m}=\Phi_{m}^{H}|_{\T_m} : \, \T_m\to D$$ is also holomorphic. Since $\P_m$ has a continuous extension $\P^H_m$, we apply the Riemann extension theorem, for which  we only need to
show that $\TT\setminus \T_m$ is an analytic subvariety of $\TT$.
In fact, since  $\bar{\mathcal{Z}}_m\setminus\Z$ is a union of simple normal crossing divisors, from Lemma \ref{trivial monodromy}, we see that the subset $\ZZ\setminus\Z$ consists of normal crossing divisors in $\ZZ$ around which the monodromy group is trivial.  Therefore $\ZZ\setminus\Z$ is an analytic subvariety of $\ZZ$. On the other hand, from Proposition \ref{opend}, we know that  $\TT\setminus \T_m$ is the inverse image of $\ZZ\setminus \Z$ under the covering map $$\pi^{H}_m:\, \TT \to \ZZ,$$ this implies that $\TT\setminus \T_m$ is an analytic subvariety of $\TT$.

Note that the fact that $\ZZ\setminus\Z$, therefore $\TT\setminus \T_m$, is analytic subvariety is contained in Theorem 9.6 in \cite{Griffiths3}. We refer the readers to Page 156 of \cite{Griffiths3} for the related discussions.

We will call $\P^H_m:\, \T_m^H \to D$ the extended period map.

\begin{lemma}\label{extendedtransversality}
The extended period map $\Phi_m^H :\, \T_m^H \to D$ satisfies the Griffiths
transversality.
\end{lemma}
\begin{proof}
Let $\text{T}^{1,0}_hD$ be the horizontal subbundle.
Since $\Phi_m^H : \T_m^H \to D$ is a holomorphic map, the tangent map
$$d\Phi_m^H : \, \text{T}^{1,0}\T_m^H \to \text{T}^{1,0}D$$
is at least continuous. We only need to show that the image of $(d\Phi_m^H)$ is contained in the horizontal tangent bundle $\text{T}^{1,0}_hD$.

Since $\text{T}^{1,0}_hD$ is closed in $\text{T}^{1,0}D$, $(d\Phi_m^H)^{-1}(\text{T}^{1,0}_hD)$ is closed in $\text{T}^{1,0}\T_m^H$. But $\Phi_m^H|_{\T_m}$ satisfies the Griffiths transversality, i.e. $(d\Phi_m^H)^{-1}(\text{T}^{1,0}_hD)$ contains $\text{T}^{1,0}\T_m$, which is open in $\text{T}^{1,0}\T_m^H$. Hence $(d\Phi_m^H)^{-1}(\text{T}^{1,0}_hD)$ contains the closure of $\text{T}^{1,0}\T_m$, which is $\text{T}^{1,0}\T_m^H$.
\end{proof}

\section{Boundedness of the period maps}\label{boundedness of period maps}
In Section \ref{preliminary}, we review some basic properties of the period domain from Lie group and Lie algebra point of view.
We fix a base point $p\in\mathcal{T}$ and introduce the unipotent group $N_+\subseteq \check{D}$, which is biholomorphic to complex Euclidean space $\mathbb{C}^d$.
In Section \ref{boundedness of Phi}, we define $\check{\T}=\Phi^{-1}(N_+\cap D)$, and consider, in terms of the notations of Lie algebras in Section \ref{preliminary},  $$\mathfrak{p}_+=\mathfrak{p}/(\mathfrak{p}\cap
\mathfrak{b})=\mathfrak{p}\cap\mathfrak{n}_+ \subseteq
\mathfrak{n}_+$$
and $\text{exp}(\mathfrak{p}_+)\subseteq N_+$ as complex Euclidean subspaces. 

Let  $$P_+ :\, N_+\cap D \to \text{exp}(\mathfrak{p}_+)\cap D$$ be the induced projection
map and $\P_+=P_+ \circ \Phi|_{\check{\T}}$.  We first prove that the image of $$\P_+ :\, \check{\T} \to \text{exp}(\mathfrak{p}_+)\cap
D$$ is bounded in $\text{exp}(\mathfrak{p}_+)$ with respect to the Euclidean metric on $\text{exp}(\mathfrak{p}_+)\subseteq N_+$. In fact we actually prove that $\text{exp}(\mathfrak{p}_+)\cap
D$ is bounded in $\text{exp}(\mathfrak{p}_+)$. Then we prove the boundedness of the image of $$\Phi:\, \check{\T} \to N_{+}\cap D$$ in $N_+$ by proving the finiteness of the map $P_+|_{\Phi(\check{\T})}$.
Finally we prove that $\T\setminus \check{\T}$ is an analytic subvariety of $\T$ with $\mbox{codim}_\C (\T\setminus \check{\T})\ge 1$, and apply the Riemann extension theorem to get the boundedness of the image of $$\Phi:\, {\T} \to N_{+}\cap D$$ in $N_+.$

As mentioned in the introduction, when the images of these maps are bounded sets in the corresponding complex  Euclidean spaces, we will simply say that these maps are bounded.

\subsection{Period domain from Lie algebras and Lie groups}\label{preliminary}Let us briefly recall some properties of the period domain from Lie group and Lie algebra point of view. All of the results in this section is well-known to the experts in the subject. The purpose to give details is to fix notations. One may either skip this section or refer to \cite{GS} and \cite{schmid1} for most of the details.

The orthogonal group of the bilinear form $Q$ in the definition of Hodge structure is a linear algebraic group, defined over $\mathbb{Q}$. Let us simply denote $H_{\mathbb{C}}=H_{pr}^n(M, \mathbb{C})$ and $H_{\mathbb{R}}=H_{pr}^n(M, \mathbb{R})$. The group over ${\mathbb C}$ is
\begin{align*}
G_{\mathbb{C}}=\{ g\in GL(H_{\mathbb{C}})|~ Q(gu, gv)=Q(u, v) \text{ for all } u, v\in H_{\mathbb{C}}\},
\end{align*}
which acts on $\check{D}$ transitively. The group of real points in $G_{\mathbb{C}}$ is
\begin{align*}
G_{\mathbb{R}}=\{ g\in GL(H_{\mathbb{R}})|~ Q(gu, gv)=Q(u, v) \text{ for all } u, v\in H_{\mathbb{R}}\},
\end{align*}
which acts transitively on $D$ as well.

Consider the period map $\Phi:\, \mathcal{T}\rightarrow D$. Fix a point $p\in \mathcal{T}$ with the image $$o:=\Phi(p)=\{F^n_p\subset \cdots\subset F^{0}_p\}\in D.$$The points $p\in \mathcal{T}$ and $o\in D$ will be referred as the base points or the reference points. A linear transformation $g\in G_{\mathbb{C}}$ preserves the base point if and only if $gF^k_p=F^k_p$ for each $k$. Thus it gives the identification
\begin{align*}
\check{D}\simeq G_\mathbb{C}/B\quad\text{with}\quad B=\{ g\in G_\mathbb{C}|~ gF^k_p=F^k_p, \text{ for any } k\}.
\end{align*}
Similarly, one obtains an analogous identification
\begin{align*}
D\simeq G_\mathbb{R}/V\hookrightarrow \check{D}\quad\text{with}\quad V=G_\mathbb{R}\cap B,
\end{align*}
where the embedding corresponds to the inclusion $$
G_\mathbb{R}/V=G_{\mathbb{R}}/G_\mathbb{R}\cap B\subseteq G_\mathbb{C}/B.$$

The Lie algebra $\mathfrak{g}$ of the complex Lie group $G_{\mathbb{C}}$ can be described as
\begin{align*}
\mathfrak{g}&=\{X\in \text{End}(H_\mathbb{C})|~ Q(Xu, v)+Q(u, Xv)=0, \text{ for all } u, v\in H_\mathbb{C}\}.
\end{align*}
It is a simple complex Lie algebra, which contains
$\mathfrak{g}_0=\{X\in \mathfrak{g}|~ XH_{\mathbb{R}}\subseteq H_\mathbb{R}\}$
as a real form, i.e. $\mathfrak{g}=\mathfrak{g}_0\oplus i \mathfrak{g}_0.$ With the inclusion $G_{\mathbb{R}}\subseteq G_{\mathbb{C}}$, $\mathfrak{g}_0$ becomes the Lie algebra of $G_{\mathbb{R}}$. One observes that the reference Hodge structure $\{H^{k, n-k}_p\}_{k=0}^n$ of $H^n(M,{\mathbb{C}})$ induces a Hodge structure of weight zero on
$\mathfrak{g}$, namely,
\begin{align*}
\mathfrak{g}=\bigoplus_{k\in \mathbb{Z}} \mathfrak{g}^{k, -k}\quad\text{with}\quad\mathfrak{g}^{k, -k}=\{X\in \mathfrak{g}|XH_p^{r, n-r}\subseteq H_p^{r+k, n-r-k}\}.
\end{align*}

Since the Lie algebra $\mathfrak{b}$ of $B$ consists of those $X\in \mathfrak{g}$ that preserves the reference Hodge filtration $\{F_p^n\subset\cdots\subset F^0_p\}$, one thus has
\begin{align*}
 \mathfrak{b}=\bigoplus_{k\geq 0} \mathfrak{g}^{k, -k}.
\end{align*}
The Lie algebra $\mathfrak{v}_0$ of $V$ is
$$\mathfrak{v}_0=\mathfrak{g}_0\cap \mathfrak{b}=\mathfrak{g}_0\cap \mathfrak{b}\cap\bar{\mathfrak{b}}=\mathfrak{g}_0\cap \mathfrak{g}^{0, 0}.$$
With the above isomorphisms, the holomorphic tangent space of $\check{D}$ at the base point is naturally isomorphic to $\mathfrak{g}/\mathfrak{b}$.

Let us consider the nilpotent Lie subalgebra $$\mathfrak{n}_+:=\oplus_{k\geq 1}\mathfrak{g}^{-k,k}.$$ Then one gets the holomorphic isomorphism $\mathfrak{g}/\mathfrak{b}\cong \mathfrak{n}_+$. We take the unipotent group $$N_+=\exp(\mathfrak{n}_+).$$

As $\text{Ad}(g)(\mathfrak{g}^{k, -k})$ is in  $\bigoplus_{i\geq
k}\mathfrak{g}^{i, -i} \text{ for each } g\in B,$ the subspace
$\mathfrak{b}\oplus\linebreak \mathfrak{g}^{-1, 1}/\mathfrak{b}\subseteq
\mathfrak{g}/\mathfrak{b}$ defines an Ad$(B)$-invariant subspace. By
left translation via $G_{\mathbb{C}}$,
$\mathfrak{b}\oplus\mathfrak{g}^{-1,1}/\mathfrak{b}$ gives rise to a
$G_{\mathbb{C}}$-invariant holomorphic subbundle of the holomorphic
tangent bundle. It will be denoted by $\text{T}^{1,0}_{h}\check{D}$,
and will be referred to as the horizontal tangent subbundle. One can
check that this construction does not depend on the choice of the
base point. The horizontal tangent subbundle, restricted to $D$,
determines a subbundle $\text{T}_{h}^{1, 0}D$ of the holomorphic
tangent bundle $\text{T}^{1, 0}D$ of $D$. 

The
$G_{\mathbb{C}}$-invariance of $\text{T}^{1, 0}_{h}\check{D}$
implies the $G_{\mathbb{R}}$-invariance of $\text{T}^{1, 0}_{h}D$.
Note that the horizontal tangent subbundle $\text{T}_{h}^{1, 0}D$
can also be constructed as the associated bundle of the principle
bundle $V\to G_\mathbb{R} \to D$ with the adjoint representation of
$V$ on the space
$\mathfrak{b}\oplus\mathfrak{g}^{-1,1}/\mathfrak{b}$. 

 As another
interpretation of the horizontal bundle in terms of the Hodge
bundles ${F}^{k}\to \check{D}$, $0\le k \le n$, one has
\begin{align}\label{horizontal}
\text{T}^{1, 0}_{h}\check{D}\simeq \text{T}^{1, 0}\check{D}\cap \bigoplus_{k=1}^{n}\text{Hom}({F}^{k}/{F}^{k+1}, {F}^{k-1}/{F}^{k}).
\end{align}

A holomorphic mapp $\Psi: \,M\rightarrow \check{D}$ of a complex manifold $M$ into $\check{D}$ is called {horizontal}, if the tangent map $$d\Psi:\, \text{T}^{1,0}M \to \text{T}^{1,0}\check{D}$$ takes values in $\text{T}^{1,0}_h\check{D}$.
The period map $\Phi: \, \T\rightarrow D$ is horizontal due to Griffiths transversality.

Let us introduce the notion of an adapted basis for the given Hodge decomposition or the Hodge filtration.
For any $p\in \mathcal{T}$ and $f^k=\dim F^k_p$ for any $0\leq k\leq n$, we call a basis $$\xi=\left\{ \xi_0, \xi_1, \dots, \xi_N, \dots, \xi_{f^{k+1}}, \dots, \xi_{f^k-1}, \dots,\xi_{f^{2}}, \dots, \xi_{f^{1}-1}, \xi_{f^{0}-1} \right\}$$ of $H^n(M_p, \mathbb{C})$ an \textit{adapted basis for the given Hodge decomposition}
$$H^n(M_p, {\mathbb{C}})=H^{n, 0}_p\oplus H^{n-1, 1}_p\oplus\cdots \oplus H^{1, n-1}_p\oplus H^{0, n}_p, $$
if it satisfies $
H^{k, n-k}_p=\text{Span}_{\mathbb{C}}\left\{\xi_{f^{k+1}}, \dots, \xi_{f^k-1}\right\}$ with $\dim H^{k,n-k}_p=f^k-\linebreak f^{k+1}$.

We call a basis
\begin{align*}
\zeta=\{\zeta_0, \zeta_1, \dots, \zeta_N, \dots, \zeta_{f^{k+1}}, \dots, \zeta_{f^k-1}, \dots, \zeta_{f^2}, \dots, \zeta_{f^1-0}, \zeta_{f^0-1}\}
\end{align*}
of $H^n(M_p, \mathbb{C})$ an \textit{adapted basis for the given Hodge filtration}
\begin{align*}
F^n\subseteq F^{n-1}\subseteq\cdots\subseteq F^0
\end{align*}
if it satisfies $$F^{k}=\text{Span}_{\mathbb{C}}\{\zeta_0, \dots, \zeta_{f^k-1}\}$$ with $\text{dim}_{\mathbb{C}}F^{k}=f^k$.
Moreover, unless otherwise pointed out, the matrices in this paper are $m\times m$ matrices, where $m=f^0$. The blocks of the $m\times m$ matrix $T$ is set as follows:
for each $0\leq \alpha, \beta\leq n$, the $(\alpha, \beta)$-th block $T^{\alpha, \beta}$ is
\begin{align}\label{block}
T^{\alpha, \beta}=\left[T_{ij}(\tau)\right]_{f^{-\alpha+n+1}\leq i \leq f^{-\alpha+n}-1, \ f^{-\beta+n+1}\leq j\leq f^{-\beta+n}-1},
\end{align} where $T_{ij}$ is the entries of
the matrix $T$, and $f^{n+1}$ is defined to be zero. In particular, $T =[T^{\alpha,\beta}]$ is called a \textit{block lower triangular matrix} if
$T^{\alpha,\beta}=0$ whenever $\alpha<\beta$.
\begin{remark}\label{identify n+}
We remark that by fixing a base point, we can identify the above quotient Lie groups or Lie algebras with their orbits in the corresponding quotient Lie algebras or Lie groups. For example, $\mathfrak{n}_+\cong \mathfrak{g}/\mathfrak{b}$, $$\mathfrak{g}^{-1,1}\cong\mathfrak{b}\oplus \mathfrak{g}^{-1,1}/\mathfrak{b},$$ and $N_+\cong N_+B/B\subseteq \check{D}$, since $N_{+}\cap B=\{Id\}.$
We can also identify a point $$\Phi(p)=\{ F^n_p\subseteq F^{n-1}_p\subseteq \cdots \subseteq F^{0}_p\}\in D$$ with its Hodge decomposition $\bigoplus_{k=0}^n H^{k, n-k}_p$.

 With any fixed adapted basis of the corresponding Hodge decomposition for the base point, we have matrix representations of the elements in the above Lie groups and Lie algebras. For example, the elements in $N_+$ can be realized as nonsingular block lower triangular matrices with identity blocks in the diagonal; elements in $B$ can be realized as nonsingular block upper triangular matrices.
\end{remark}

              We shall review and collect some facts about the structure of simple Lie algebra $\mathfrak{g}$ in our case. Again one may refer to \cite{GS} and \cite{schmid1} for more details. Let $\theta: \,\mathfrak{g}\rightarrow \mathfrak{g}$ be the Weil operator, which is defined by
\begin{align*}\theta(X)=(-1)^p X\quad \text{ for } X\in \mathfrak{g}^{p,-p}.
\end{align*}
Then $\theta$ is an involutive automorphism of $\mathfrak{g}$, and is defined over $\mathbb{R}$. The $(+1)$ and $(-1)$ eigenspaces of $\theta$ will be denoted by $\mathfrak{k}$ and $\mathfrak{p}$ respectively. Moreover, set
\begin{align*}
\mathfrak{k}_0=\mathfrak{k}\cap \mathfrak{g}_0, \quad \mathfrak{p}_0=\mathfrak{p}\cap \mathfrak{g}_0.
\end{align*}The fact that $\theta$ is an involutive automorphism implies
\begin{align*}
\mathfrak{g}=\mathfrak{k}\oplus\mathfrak{p}, \quad \mathfrak{g}_0=\mathfrak{k}_0\oplus \mathfrak{p}_0, \quad
[\mathfrak{k}, \,\mathfrak{k}]\subseteq \mathfrak{k}, \quad [\mathfrak{p},\,\mathfrak{p}]\subseteq\mathfrak{k}, \quad [\mathfrak{k}, \,\mathfrak{p}]\subseteq \mathfrak{p}.
\end{align*}

Let us consider
$\mathfrak{g}_c=\mathfrak{k}_0\oplus \sqrt{-1}\mathfrak{p}_0.$
Then $\mathfrak{g}_c$ is a real form for $\mathfrak{g}$.
Recall that the killing form $B(\cdot, \,\cdot)$ on $\mathfrak{g}$ is defined by
\begin{align*}B(X,Y)=\text{Trace}(\text{ad}(X)\circ\text{ad}(Y)) \quad \text{for } X,Y\in \mathfrak{g}.
\end{align*}
A semisimple Lie algebra is compact if and only if its Killing form is negative definite. Thus it is not hard to check that $\mathfrak{g}_c$ is actually a compact real form of $\mathfrak{g}$, while $\mathfrak{g}_0$ is a noncompact real form.
Recall that $G_{\mathbb{R}}\subseteq G_{\mathbb{C}}$ is the subgroup which correpsonds to the subalgebra $\mathfrak{g}_0\subseteq\mathfrak{g}$. Let us denote the connected subgroup $G_c\subseteq G_{\mathbb{C}}$ which corresponds to the subalgebra $\mathfrak{g}_c\subseteq\mathfrak{g}$. Let us denote the complex conjugation of $\mathfrak{g}$ with respect to the compact real form $\mathfrak{g}_{c}$ by $\tau_c$, and the complex conjugation of $\mathfrak{g}$ with respect to the noncompact real form $\mathfrak{g}_{0}$ by $\tau_0$.

The intersection $K=G_c\cap G_{\mathbb{R}}$ is then a compact subgroup of $\mathbb{G}_{\mathbb{R}}$, whose Lie algebra is $\mathfrak{k}_0=\mathfrak{g}_{\mathbb{R}}\cap \mathfrak{g}_c$.
With the above notations, Schmid showed in \cite{schmid1} that $K$ is a maximal compact subgroup of $G_{\mathbb{R}}$, and it meets every connected component of $G_{\mathbb{R}}$. Moreover, $V=G_{\mathbb{R}}\cap B\subseteq K$.

We know that in our cases, $G_{\mathbb{C}}$ is a connected simple Lie group, $B$ is a parabolic subgroup in $G_{\mathbb{C}}$ with $\mathfrak{b}$ as its Lie subalgebra.
The Lie algebra $\mathfrak{b}$ has a unique maximal nilpotent ideal $\mathfrak{n}_-$. It is not hard to see that
\begin{align*}\mathfrak{g}_c\cap \mathfrak{n}_-=\mathfrak{n}_-\cap \tau_c(\mathfrak{n}_-)=0.
\end{align*}
By using Bruhat's lemma, one concludes $\mathfrak{g}$ is spanned by the parabolic subalgebras $\mathfrak{b}$ and $\tau_c(\mathfrak{b})$. Moreover $\mathfrak{v}=\mathfrak{b}\cap \tau_c(\mathfrak{b})$, $\mathfrak{b}=\mathfrak{v}\oplus \mathfrak{n}_-$. In particular, we also have
$$\mathfrak{n}_+=\tau_c(\mathfrak{n}_-).$$

As remarked in $\S1$ in \cite{GS} of Griffiths and Schmid, one gets that $\mathfrak{v}$ must have the same rank as $\mathfrak{g}$, as $\mathfrak{v}$ is  the intersection of the two parabolic subalgebras $\mathfrak{b}$ and $\tau_c(\mathfrak{b})$. Moreover, $\mathfrak{g}_0$ and $\mathfrak{v}_0$ are also of equal rank, since they are real forms of $\mathfrak{g}$ and $\mathfrak{v}$ respectively. Therefore, we can choose a Cartan subalgebra $\mathfrak{h}_0$ of $\mathfrak{g}_0$ such that $\mathfrak{h}_0\subseteq \mathfrak{v}_0$ is also a Cartan subalgebra of $\mathfrak{v}_0$. Since $\mathfrak{v}_0\subseteq \mathfrak{k}_0$, we also have $\mathfrak{h}_0\subseteq\mathfrak{k}_0$. A Cartan subalgebra of a real Lie algebra is a maximal abelian subalgebra. Therefore $\mathfrak{h}_0$ is also a maximal abelian subalgebra of $\mathfrak{k}_0$, hence $\mathfrak{h}_0$ is a Cartan subalgebra of $\mathfrak{k}_0$. 

Summarizing the above discussion, we get the following proposition.
\begin{proposition}\label{cartaninK}There exists a Cartan subalgebra $\mathfrak{h}_0$ of $\mathfrak{g}_0$ such that $\mathfrak{h}_0\subseteq \mathfrak{v}_0\subseteq \mathfrak{k}_0$, and $\mathfrak{h}_0$ is also a Cartan subalgebra of $\mathfrak{k}_0$. \end{proposition}

\begin{remark}As an alternate proof of Proposition \ref{cartaninK}, to show that $\mathfrak{k}_0$ and $\mathfrak{g}_0$ have equal rank, one only needs to notice that $\mathfrak{g}_0$ in our case is one of the following real simple Lie algebras: $\mathfrak{sp}(2l, \mathbb{R})$, $\mathfrak{so}(p, q)$ with $p+q$ odd, or $\mathfrak{so}(p, q)$ with $p$ and $q$ both even. One may refer to \cite{Griffiths1} and \cite{su} for more details.\end{remark}
Proposition \ref{cartaninK} implies that the simple Lie algebra $\mathfrak{g}_0$ in our case is a simple Lie algebra of first category as defined in \cite[$\S4$]{Sugi}. In the following, we will briefly derive the result of a simple Lie algebra of first category in \cite[Lemma 3]{Sugi1}. One may also refer to \cite[Lemma 2.2.12, page~141--142]{Xu} for the same result.

Let us still use the above notations of the Lie algebras we consider. By Proposition 4, we can take $\mathfrak{h}_0$ to be a Cartan subalgebra of $\mathfrak{g}$ such that $$\mathfrak{h}_0\subseteq \mathfrak{v}_0\subseteq\mathfrak{k}_0$$ and $\mathfrak{h}_0$ is also a Cartan subalgebra of $\mathfrak{k}_0$. Let us denote $\mathfrak{h}$ to be the complexification of $\mathfrak{h}_0$. Then $\mathfrak{h}$ is a Cartan subalgebra of $\mathfrak{g}$ such that $\mathfrak{h}\subseteq \mathfrak{v}\subseteq \mathfrak{k}$.

Write $\mathfrak{h}_0^*=\text{Hom}(\mathfrak{h}_0, \mathbb{R})$ and $\mathfrak{h}^*_{\mathbb{R}}=\sqrt{-1}\mathfrak{h}^*_0. $ Then $\mathfrak{h}^*_{\mathbb{R}}$ can be identified with $\mathfrak{h}_{\mathbb{R}}:=\sqrt{-1}\mathfrak{h}_0$ by the duality with respect to the restriction of the Killing form $B$ of $\mathfrak{g}$ to $\mathfrak{h}_{\mathbb{R}}$.
 Let $\rho\in\mathfrak{h}_{{\mathbb{R}}}^*\simeq \mathfrak{h}_{{\mathbb{R}}}$, one can define the following subspace of $\mathfrak{g}$
\begin{align*}
\mathfrak{g}^{\rho}=\{x\in \mathfrak{g}\mid [h, \,x]=\rho(h)x\quad\text{for all } h\in \mathfrak{h}\}.
\end{align*}
An element $\varphi\in\mathfrak{h}_{{\mathbb{R}}}^*\simeq\mathfrak{h}_{{\mathbb{R}}}$ is called a root of $\mathfrak{g}$ with respect to $\mathfrak{h}$ if $\mathfrak{g}^\varphi\neq \{0\}$.

Let $\Delta\subseteq \mathfrak{h}^{*}_{\mathbb{R}}\simeq\mathfrak{h}_{{\mathbb{R}}}$ denote the space of nonzero $\mathfrak{h}$-roots. Then each root space
\begin{align*}
\mathfrak{g}^{\varphi}=\{x\in\mathfrak{g}\mid [h,x]=\varphi(h)x \text{ for all } h\in \mathfrak{h}\}
\end{align*} belongs to some $\varphi\in \Delta$ is one-dimensional over $\mathbb{C}$, generated by a root vector~$e_{{\varphi}}$.

Since the involution $\theta$ is a Lie-algebra automorphism fixing $\mathfrak{k}$, we have $$[h, \theta(e_{{\varphi}})]=\varphi(h)\theta(e_{{\varphi}})$$ for any $h\in \mathfrak{h}$ and $\varphi\in \Delta.$ Thus $\theta(e_{{\varphi}})$ is also a root vector belonging to the root $\varphi$, so $e_{{\varphi}}$ must be an eigenvector of $\theta$. It follows that there is a decomposition of the roots $\Delta$ into $\Delta_{\mathfrak{k}}\cup\Delta_{\mathfrak{p}}$ of compact roots and noncompact roots with root spaces $\mathbb{C}e_{{\varphi}}\subseteq \mathfrak{k}$ and $\mathfrak{p}$ respectively.
The adjoint representation of $\mathfrak{h}$ on $\mathfrak{g}$ determins a decomposition
\begin{align*}
\mathfrak{g}=\mathfrak{h}\oplus \sum_{\varphi\in\Delta}\mathfrak{g}^{\varphi}.
\end{align*}There also exists a Weyl base $\{h_i, 1\leq i\leq l;\,\,e_{{\varphi}}, \text{ for any } \varphi\in\Delta\}$ with $l=\text{rank}(\mathfrak{g})$ such that
$\text{Span}_{\mathbb{C}}\{h_1, \dots, h_l\}=\mathfrak{h}$, $\text{Span}_{\mathbb{C}}\{e_{\varphi}\}=g^{\varphi}$ for each $\varphi\in \Delta$, and
\begin{align}
&\tau_c(h_i)=\tau_0(h_i)=-h_i, \quad \text{ for any }1\leq i\leq l;\nonumber \\
&\tau_c(e_{{\varphi}})=\tau_0(e_{{\varphi}})=-e_{{-\varphi}}, \quad \text{for any } \varphi\in \Delta_{\mathfrak{k}};\label{bar}\\
&\tau_0(e_{{\varphi}})=-\tau_c(e_{{\varphi}})=e_{{-\varphi}}, \quad\text{for any }\varphi\in \Delta_{\mathfrak{p}},\nonumber
\end{align}
and
\begin{align}
\mathfrak{k}_0&=\mathfrak{h}_0+\sum_{\varphi\in \Delta_{\mathfrak{k}}}\mathbb{R}(e_{{\varphi}}-e_{{-\varphi}})+\sum_{\varphi\in\Delta_{\mathfrak{k}}} \mathbb{R}\sqrt{-1}(e_{{\varphi}}+e_{{-\varphi}});\label{rootk}\\
\mathfrak{p}_0&=\sum_{\varphi\in \Delta_{\mathfrak{p}}}\mathbb{R}(e_{{\varphi}}+e_{{-\varphi}})+\sum_{\varphi\in\Delta_{\mathfrak{p}}} \mathbb{R}\sqrt{-1}(e_{{\varphi}}-e_{{-\varphi}}).\label{rootp}
\end{align}
First we have the following lemma.
\begin{lemma}\label{puretype}Let $\Delta$ be the set of $\mathfrak{h}$-roots as above. Then for each root $\varphi\in \Delta$, there is an integer $-n\leq k\leq n$ such that $e_{\varphi}\in \mathfrak{g}^{k,-k}$. In particular, if $e_{{\varphi}}\in \mathfrak{g}^{k, -k}$, then $\tau_0(e_{{\varphi}})\in \mathfrak{g}^{-k,k}$ for any $-n\leq k\leq n$.
\end{lemma}
\begin{proof}
Let $\varphi$ be a root, and $e_{\varphi}$ be the generator of the root space $\mathfrak{g}^{\varphi}$, then $e_{\varphi}=\sum_{k=-n}^n e^{-k,k}$, where $e^{-k,k}\in \mathfrak{g}^{-k,k}$. Because $\mathfrak{h}\subseteq \mathfrak{v}\subseteq \mathfrak{g}^{0,0}$, the Lie bracket $[e^{-k,k}, h]\in \mathfrak{g}^{-k,k}$ for each $k$. Then the condition $[e_{\varphi}, h]=\varphi (h)e_{\varphi}$ implies that
\begin{align*}
\sum_{k=-n}^n[e^{-k,k},h]=\sum_{k=-n}^n\varphi(h)e^{-k,k}\quad \text{ for each } h\in \mathfrak{h}.
\end{align*}
By comparing the type, we get
\begin{align*} [e^{-k,k},h]=\varphi(h)e^{-k,k}\quad \text{ for each } h\in \mathfrak{h}.
\end{align*}
Therefore $e^{-k,k}\in g^{\varphi}$ for each $k$. As $\{e^{-k,k}, \}_{k=-n}^n$ forms a linear independent set, but $g^{\varphi}$ is one dimensional, thus there is only one $k$ with $-n\leq k\leq n$ and $e^{-k,k}\neq 0$.
\end{proof}

Let us now introduce a lexicographic order (\cite[page 41]{Xu} or \cite[page 416]{Sugi}) in the real vector space $\mathfrak{h}_{{\mathbb{R}}}$ as follows: we fix an ordered basis $e_1, \dots, e_l$ for $\mathfrak{h}_{{\mathbb{R}}}$. Then for any $h=\sum_{i=1}^l\lambda_ie_i\in\mathfrak{h}_{{\mathbb{R}}}$, we call $h>0$ if the first nonzero coefficient is positive, that is, if $$\lambda_1=\cdots=\lambda_k=0, \  \lambda_{k+1}>0$$ for some $1\leq k<l$. For any $h, h'\in\mathfrak{h}_{{\mathbb{R}}}$, we say $h>h'$ if $h-h'>0$, $h<h'$ if $h-h'<0$ and $h=h'$ if $h-h'=0$.
In particular, let us identify the dual spaces $\mathfrak{h}_{{\mathbb{R}}}^*$ and $\mathfrak{h}_{{\mathbb{R}}}$, thus $\Delta\subseteq \mathfrak{h}_{{\mathbb{R}}}$. 

Let us first choose a maximal linearly independent subset $\{e_1, \dots, e_s\}$ of $\Delta_{\mathfrak{p}}$, then a maximal linearly independent subset $\{e_{s+1}, \dots, e_{l}\}$ of $\Delta_{\mathfrak{k}}$. Then $\{e_1, \dots, e_s, e_{s+1}, \dots, e_l\}$ forms a basis for $\mathfrak{h}_{{\mathbb{R}}}^*$ since $\text{Span}_{\mathbb{R}}\Delta=\mathfrak{h}_{{\mathbb{R}}}^*$. We define the above lexicographic order in $\mathfrak{h}_{{\mathbb{R}}}^*\simeq \mathfrak{h}_{{\mathbb{R}}}$ by using the ordered basis $\{e_1,\dots, e_l\}$. In this way, we also define
\begin{align*}
\Delta^+=\{\varphi>0: \,\varphi\in \Delta\};\quad \Delta_{\mathfrak{p}}^+=\Delta^+\cap \Delta_{\mathfrak{p}}.
\end{align*}
Similarly we define $\Delta^-$, $\Delta^-_{\mathfrak{p}}$, $\Delta^{+}_{\mathfrak{k}}$, and $\Delta^-_{\mathfrak{k}}$. Then one can conclude the following lemma from Lemma 2.2.10 and Lemma 2.2.11 at pp.141 in \cite{Xu},
\begin{lemma}\label{RootSumNonRoot}Using the above notation, we have
\begin{align*}(\Delta_{\mathfrak{k}}+\Delta^{\pm}_{\mathfrak{p}})\cap \Delta\subseteq\Delta_{\mathfrak{p}}^{\pm}; \quad (\Delta_{\mathfrak{p}}^{\pm}+\Delta_{\mathfrak{p}}^{\pm})\cap\Delta=\emptyset.
\end{align*}
If one defines $$\mathfrak{p}^{\pm}=\sum_{\varphi\in\Delta^{\pm}_{\mathfrak{p}}} \mathfrak{g}^{\varphi}\subseteq\mathfrak{p},$$ then $\mathfrak{p}=\mathfrak{p}^+\oplus\mathfrak{p}^-$ and $
[\mathfrak{p}^{\pm}, \, \mathfrak{p}^{\pm}]=0,$ $ [\mathfrak{p}^+, \,\mathfrak{p}^-]\subseteq\mathfrak{k}$, $[\mathfrak{k}, \,\mathfrak{p}^{\pm}]\subseteq\mathfrak{p}^{\pm}. $
\end{lemma}
\begin{definition}Two different roots $\varphi, \psi\in \Delta$ are said to be strongly orthogonal if and only if $\varphi\pm\psi\notin\Delta\cup \{0\}$, which is denoted by $\varphi \independent\psi$.
\end{definition}
For the real simple Lie algebra $\mathfrak{g}_0=\mathfrak{k}_0\oplus\mathfrak{p}_0$ which has a Cartan subalgebra $\mathfrak{h}_0$ in $\mathfrak{k}_0$, the maximal abelian subspace of $\mathfrak{p}_0$ can be described as in the following lemma, which is a slight extension of a lemma of Harish-Chandra in \cite{HC}. One may refer to \cite[Lemma 3]{Sugi1} or \cite[Lemma 2.2.12, page 141--142]{Xu} for more details. For reader's convenience we give the detailed proof.
\begin{lemma}\label{stronglyortho}There exists a set of strongly orthogonal noncompact positive roots $\Lambda=\{\varphi_1, \dots, \varphi_r\}\subseteq\Delta^+_{\mathfrak{p}}$ such that
\begin{align*}
\mathfrak{A}_0=\sum_{i=1}^r\mathbb{R}\left(e_{{\varphi_i}}+e_{{-\varphi_i}}\right)
\end{align*}
is a maximal abelian subspace in $\mathfrak{p}_0$.
\end{lemma}
\begin{proof}
Let $\varphi_1$ be the minimum in $\Delta_\mathfrak{p}^+$, and $\varphi_2$ be the minimal element in $\{\varphi\in\Delta_{\mathfrak{p}}^+:\,\varphi\independent \varphi_1\}$, then we obtain inductively an maximal ordered set of roots $\Lambda=\{\varphi_1,\dots, \varphi_r\}\subseteq \Delta_\mathfrak{p}^+$, such that for each $1\leq k\leq r$
\begin{align*}
\varphi_k=\min\{\phi\in\Delta_{\mathfrak{p}}^+:\,\varphi\independent \varphi_j \text{ for } 1\leq j\leq k-1\}.
\end{align*}
Because $\varphi_i\independent \varphi_j$ for any $1\leq i<j\leq r$, we have $[e_{{\pm\varphi_i}}, e_{{\pm\varphi_j}}]=0$. Therefore $$\mathfrak{A}_0=\sum_{i=1}^r\mathbb{R}\left(e_{{\varphi_i}}+e_{{-\varphi_i}}\right)$$ is an abelian subspace of $\mathfrak{p}_0$. Also because a root can not be strongly orthogonal to itself, the ordered set $\Lambda$ contains distinct roots. Thus $\dim_\mathbb{R}\mathfrak{A}_0=r$.

Now we prove that $\mathfrak{A}_0$ is a maximal abelian subspace of $\mathfrak{p}_0$. Suppose towards a contradiction that there was a nonzero vector $X\in\mathfrak{p}_0$ as follows
$$X=\!\!\sum_{\alpha\in\Delta_{\mathfrak{p}}^+\setminus \Lambda}\!\!\lambda_\alpha\left( e_{{\alpha}}+e_{{-\alpha}}\right)+\!\!\sum_{\alpha\in\Delta_{\mathfrak{p}}^+\setminus \Lambda}\!\!\mu_\alpha\sqrt{-1}\left( e_{{\alpha}}-e_{{-\alpha}}\right), \quad\text{where }\lambda_\alpha, \mu_\alpha\in \mathbb{R}, $$ 
such that $[X, e_{{\varphi_i}}\!+\!e_{{-\varphi_i}}]\!=\!0$ for each $1\!\leq\! i\!\leq\! r$.
We denote $c_\alpha\!=\!\lambda_\alpha\!+\!\sqrt{-1}\mu_\alpha$. Because $X\neq 0$, there exists $\psi\in\Delta_{\mathfrak{p}}^+\setminus \Lambda$ with $c_{\psi}\neq0$.
Also $\psi$ is not strongly orthogonal to $\varphi_i$ for some $1\leq i\leq r$. Thus we may first define $k_{\psi}$ for each $\psi$ with $c_{\psi}\neq 0$ as the following: $$k_{\psi}=\min_{1\leq i\leq r}\{i:  \psi \text{ is not strongly orthogonal to } \varphi_i\}. $$ Then we know that $1\leq k_\psi\leq r$ for each $\psi$ with $c_\psi\neq 0$. We define $k$ to be the following,
\begin{align}\label{definitionofk}
k=\min_{\psi\in \Delta^+_{\mathfrak{p}}\setminus\Lambda\text{ with }c_{\psi}\neq0}\{k_{\psi}\}.
\end{align}
Here, we are taking the minimum over a finite set in \eqref{definitionofk} and $1\leq k\leq r$. Moreover, we get the following non-empty set,
\begin{align}\label{definitionofSk}S_k=\{\psi\in \Delta^+_{\mathfrak{p}}\setminus \Lambda: c_{\psi}\neq 0 \text{ and } k_{\psi}=k\}\neq \emptyset.\end{align}

Recall the notation $N_{\beta, \gamma}$ for any $\beta, \gamma\in \Delta$ is defined as follows: if $\beta+\gamma\in \Delta\cup\{0\}$, $N_{\beta,\gamma}$ is defined such that
$[e_{\beta}, e_{\gamma}]=N_{\beta,\gamma}e_{\beta+\gamma};$
if $\beta+\gamma\notin \Delta\cup \{0\}$ then one defines $N_{\beta, \gamma}=0.$
Now let us take $k$ as defined in \eqref{definitionofk} and consider the Lie bracket
\begin{align*}
0&=[X, e_{{\varphi_k}}+e_{{-\varphi_k}}]\\
&=\sum_{\psi\in\Delta_{\mathfrak{p}}^+\setminus \Lambda}\big(c_\psi(N_{{\psi,\varphi_k}}e_{{\psi+\varphi_k}}+N_{{\psi,-\varphi_k}}e_{{\psi-\varphi_k}})\\
&\qquad\qquad\ \ +\bar{c}_\psi(N_{{-\psi,\varphi_k}}e_{{-\psi+\varphi_k}}+N_{{-\psi,-\varphi_k}}e_{{-\psi-\varphi_k}})\big).
\end{align*} As $[\mathfrak{p}^{\pm}, \,\mathfrak{p}^{\pm}]=0$, we have $\psi+\varphi_k \notin \Delta$ and $-\psi-\varphi_k \notin\Delta$ for each $\psi\in \Delta_{\mathfrak{p}}^+$. Hence, $N_{{\psi,\varphi_k}}=N_{{-\psi,-\varphi_k}}=0$ for each $\psi\in\Delta_{\mathfrak{p}}^+$. Then we have the simplified expression
\begin{align}\label{equal0}
0=[X, e_{{\varphi_k}}+e_{{-\varphi_k}}]=\sum_{\psi\in\Delta_{\mathfrak{p}}^+\setminus \Lambda}\left(c_{\psi}N_{{\psi,-\varphi_k}}e_{{\psi-\varphi_k}}
+\bar{c}_\psi N_{{-\psi,\varphi_k}}e_{{-\psi+\varphi_k}}\right).
\end{align}

 Now let us take $\psi_0\in S_k\neq \emptyset$. Then $c_{\psi_0}\neq 0$. By the definition of $k$, we have $\psi_0$ is not strongly orthogonal to $\varphi_k$ while $\psi_0+\varphi_k\notin \Delta\cup\{0\}$. Thus we have $\psi_0-\varphi_k \in\Delta\cup\{0\}$. Therefore  $c_{\psi_{0}}N_{{\psi_0,-\varphi_k}}e_{{\psi_0-\varphi_k}}\neq 0$. Since $0=\linebreak{} [X, e_{{\varphi_k}}+e_{{-\varphi_k}}]$, there must exist one element $\psi_0'\neq \psi_0\in\Delta_{\mathfrak{p}}^+\setminus \Lambda$ such that $\varphi_k-\psi_0=\psi_0'-\varphi_k$ and $c_{\psi_0'}\neq 0$. This implies $2\varphi_k=\psi_0+\psi_0'$, and consequently one of $\psi_0$ and $\psi_0'$ is smaller then $\varphi_k$. Then we have the following two cases:

(i) If $\psi_0<\varphi_k$, then we find $\psi_0<\varphi_k$ with $\psi_0\independent\varphi_i$ for all $1\leq i\leq k-1$, and this contradicts to the definition of $\varphi_k$ as the following
\begin{align*}
\varphi_k=\min\{\phi\in\Delta_{\mathfrak{p}}^+:\,\varphi\independent \varphi_j \text{ for } 1\leq j\leq k-1\}.
\end{align*}

(ii) If $\psi_0'<\varphi_k$, since we have $c_{\psi_0'}\neq 0$, we have
\begin{align*}
k_{{\psi_0'}}=\min_{1\leq i\leq r}\{i: \psi_0' \text{ is not strongly orthogonal to } \varphi_i \}.
\end{align*} Then by the definition of $k$ in \eqref{definitionofk}, we have $k_{{\psi_0'}}\geq k$. Therefore we found $\psi_0'<\varphi_k$ such that $\psi_0'<\varphi_i$ for any $1\leq i\leq k-1< k_{\psi_0'}$, and this contradicts with the definition of $\varphi_k$.

Therefore in both cases, we found contradictions. Thus we conclude that $\mathfrak{A}_0$ is a maximal abelian subspace of $\mathfrak{p}_0$.
\end{proof}
For further use, we also state a proposition about the maximal abelian subspaces of $\mathfrak{p}_0$ according to \cite[Ch.~V]{Hel},
\begin{proposition}\label{adjoint} Let $\mathfrak{A}_0'$ be an arbitrary maximal abelian subspaces of $\mathfrak{p}_0$, then there exists an element $k\in K$ such that $\emph{Ad}(k)\cdot \mathfrak{A}_0=\mathfrak{A}'_0$. Moreover, we have $$\mathfrak{p}_{0}=\bigcup_{k\in K}\emph{Ad}(k)\cdot\mathfrak{A}_0,$$ where $\emph{Ad}$ denotes the adjoint action of $K$ on $\mathfrak{A}_0$.
\end{proposition}

\subsection{Boundedness of the period maps}\label{boundedness of Phi}
Now let us fix the base point $p\in \mathcal{T}$ with $\Phi(p)=o\in D$. Then according to Remark \ref{identify n+}, $N_+$ can be viewed as a subset in $\check{D}$ by identifying it with its orbit in $\check{D}$ with the base point $\Phi(p)=o$.
  Let us also fix an adapted basis $(\eta_0, \dots, \eta_{m-1})$ for the Hodge decomposition of the base point $\Phi(p)\in D$. Then we identify elements in $N_+$ with nonsingular block lower triangular matrices whose diagonal blocks are all identity submatrix. We define
\begin{align*}
\check{\mathcal{T}}=\Phi^{-1}(N_+\cap D).
\end{align*}

At the base point $\Phi(p)=o\in N_+\cap D$, we have identifications of the tangent spaces $$\text{T}_{o}^{1,0}N_+=\text{T}_o^{1,0}D\simeq \mathfrak{n}_+\simeq N_+.$$ Then the Hodge metric on $\text{T}_o^{1,0}D$ induces an Euclidean metric on $N_+$. In the proof of the following lemma, we require all the root vectors to be unit vectors with respect to this Euclidean metric.

Let $$\mathfrak{p}_+=\mathfrak{p}/(\mathfrak{p}\cap
\mathfrak{b})=\mathfrak{p}\cap\mathfrak{n}_+ \subseteq
\mathfrak{n}_+$$  denote a subspace of $\text{T}_{o}^{1,0}D\simeq
\mathfrak{n}_+$, and $\mathfrak{p}_+$ can be viewed as an Euclidean
subspace of $\mathfrak{n}_+$. Similarly $\mathrm{exp}(\mathfrak{p}_+)$
can be viewed as an Euclidean subspace of $N_{+}$ with the
induced metric from $N_{+}$. Define the projection map

$$P_+:\, N_+\cap D\to \mathrm{exp}(\mathfrak{p}_+) \cap D$$ by
\begin{align}\label{definition of P_+}
P_{+}=\text{exp}\circ p_{+}\circ \text{exp}^{-1}
\end{align}
where $\text{exp}^{-1}:\, N_+ \to \mathfrak{n}_+$ is the inverse of the isometry $\text{exp}:\, \mathfrak{n}_+ \to N_+$, and $$p_{+}:\, \mathfrak{n}_+\to \mathfrak{p}_+$$ is the projection map from the complex Euclidean space $\mathfrak{n}_+$  to its  Euclidean subspace $\mathfrak{p}_+$.

The restricted period map $\P : \, \check{\T}\to N_+\cap D$, composed with the projection map $P_{+}$, gives a holomorphic map
\begin{equation}\label{maptoA}
\P_{+}:\,  \check{\T} \to \mathrm{exp}(\mathfrak{p}_+)\cap D,
\end{equation}
where $\P_{+}=P_{+} \circ \P|_{\check{\T}}$.



Because the period map is a horizontal map, and the geometry in the horizontal direction of the period domain $D$ is similar to Hermitian symmetric space as discussed in detail in \cite{GS},  the proof of the following lemma is basically an analogue of the proof of the Harish-Chandra embedding theorem for Hermitian symmetric spaces, see for example \cite{Mok}.
\begin{lemma}\label{abounded}The image of the holomorphic map $$\P_{+} :\, \check{\T} \to \mathrm{exp}(\mathfrak{p}_+)\cap D$$ is
bounded in $\mathrm{exp}(\mathfrak{p}_+)$ with respect to the Euclidean metric on $\mathrm{exp}(\mathfrak{p}_+)\subseteq
N_+$.
\end{lemma}

\begin{proof} 
We need to show that there exists $0\leq C<\infty$ such that for any
  $q\in \check{\mathcal{T}}$, $d_{E}({\P_{+}}({p}), {\P_{+}}(q))\leq C$, where $d_E$ is the Euclidean distance on $\exp(\mathfrak{p}_+)$. In fact the proof shows that $$ \mathrm{exp}(\mathfrak{p}_+)\cap D\subset \mathrm{exp}(\mathfrak{p}_+)$$ is a bounded subset in the complex Euclidean space $\mathrm{exp}(\mathfrak{p}_+)$.
    
By the definition of $\exp(\mathfrak{p}_+)$, for any $t\in \exp(\mathfrak{p}_+)$ 
there is a unique $Y \in \mathfrak{p}_{+}$
such that the left translation $\exp (Y)\bar{o}=t$, where $\bar{o}=P_{+}(o)$ is the base point in $\exp(\mathfrak{p}_+)\cap D$.
On the other hand, for any $s\in \exp(\mathfrak{p}_+)\cap D$, there also exists an $X \in \mathfrak{p}_{0}$
such that $\exp (X)\bar{o}=s$.

Next we analyze the point  $\exp (X)\bar{o}$ considered in $\exp(\mathfrak{p}_+)$ by using the method of Harish-Chandra's proof of his famous embedding theorem for Hermitian symmetric spaces. See pages 94--97 in \cite{Mok}.





Let $\Lambda=\{\varphi_1, \dots, \varphi_r\}\subseteq \Delta^+_{\mathfrak{p}}$ be a set of strongly orthogonal roots given in Lemma \ref{stronglyortho}. We denote $x_{\varphi_i}=e_{\varphi_i}+e_{-\varphi_i}$ and $y_{\varphi_i}=\sqrt{-1}(e_{\varphi_i}-e_{-\varphi_i})$ for any $\varphi_i\in \Lambda$. Then
\begin{align*} \mathfrak{a}_0=\mathbb{R} x_{\varphi_{_1}}\oplus\cdots\oplus\mathbb{R}x_{\varphi_{_r}},\quad\text{and}\quad\mathfrak{a}_c=\mathbb{R} y_{\varphi_{_1}}\oplus\cdots\oplus\mathbb{R}y_{\varphi_{_r}},
\end{align*}
are maximal abelian spaces in $\mathfrak{p}_0$ and $\sqrt{-1}\mathfrak{p}_0$ respectively.

Since $X\in \mathfrak{p}_0$, by Proposition \ref{adjoint}, there exists $k\in K$ such that
$ X\in Ad(k)\cdot\mathfrak{A}_0$. As the adjoint action of $K$ on $\mathfrak{p}_0$ is unitary action and we are considering the length in this proof, we may simply assume that $X\in \mathfrak{A}_0$ up to a unitary transformation. With this assumption, there exists $\lambda_{i}\in \mathbb{R}$ for $1\leq i\leq r$ such that
\begin{align*}
X=\lambda_{1}x_{{\varphi_1}}+\lambda_{2}x_{{\varphi_2}}+\cdots+\lambda_{r}x_{{\varphi_r}}
\end{align*}
Since $\mathfrak{A}_0$ is commutative, we have
\begin{align*}
\exp(tX)=\prod_{i=1}^r\exp (t\lambda_{i}x_{{\varphi_i}}).
\end{align*}

Now for each $\varphi_i\in \Lambda$, we have $\text{Span}_{\mathbb{C}}\{e_{{\varphi_i}}, e_{{-\varphi_i}}, h_{{\varphi_i}}\}\simeq \mathfrak{sl}_2(\mathbb{C}) \text{ with}$
\begin{align*}
h_{{\varphi_i}}\mapsto \left[\begin{array}[c]{cc}1&0\\0&-1\end{array}\right], \quad &e_{{\varphi_i}}\mapsto \left[\begin{array}[c]{cc} 0&1\\0&0\end{array}\right], \quad e_{{-\varphi_i}}\mapsto \left[\begin{array}[c]{cc} 0&0\\1&0\end{array}\right];
\end{align*}
and  $\text{Span}_\mathbb{R}\{x_{{\varphi_i}}, y_{{\varphi_i}}, \sqrt{-1}h_{{\varphi_i}}\}\simeq \mathfrak{sl}_2(\mathbb{R})\text{ with}$
\begin{align*}
\sqrt{-1}h_{{\varphi_i}}&\mapsto \left[\begin{array}[c]{cc}\sqrt{-1}&0\\0&-\sqrt{-1}\end{array}\right], \quad &x_{{\varphi_i}}\mapsto \left[\begin{array}[c]{cc} 0&1\\1&0\end{array}\right], \\[1ex]
y_{{\varphi_i}}&\mapsto \left[\begin{array}[c]{cc} 0&\sqrt{-1}\\-\sqrt{-1}&0\end{array}\right].
\end{align*}

Since $\Lambda=\{\varphi_1, \dots, \varphi_r\}$ is a set of strongly orthogonal roots, we have that
\begin{align*}&\mathfrak{g}_{\mathbb{C}}(\Lambda)=\text{Span}_\mathbb{C}\{e_{{\varphi_i}}, e_{{-\varphi_i}}, h_{{\varphi_i}}\}_{i=1}^r\simeq (\mathfrak{sl}_2(\mathbb{C}))^r,\\\text{and} \quad&\mathfrak{g}_{\mathbb{R}}(\Lambda)=\text{Span}_\mathbb{R}\{x_{{\varphi_i}}, y_{{\varphi_i}}, \sqrt{-1}h_{{\varphi_i}}\}_{i=1}^r\simeq (\mathfrak{sl}_2(\mathbb{R}))^r.
\end{align*}
In fact, we know that for any $\varphi, \psi\in \Lambda$ with $\varphi\neq\psi$, $[e_{{\pm\varphi}}, \,e_{{\pm\psi}}]=0$ since $\varphi$ is strongly orthogonal to $\psi$; $[h_{\phi},\,h_{{\psi}}]=0$, since $\mathfrak{h}$ is abelian; and
$$[h_{{\varphi}},\, e_{{\pm\psi}}]=[[e_{{\varphi}},\,e_{{-\varphi}}],\,e_{{\pm\psi}}]=-[[e_{-\phi},\ e_{\pm\psi}],\ e_\phi]-[[e_{\pm\psi},\ e_\phi],\ e_{-\phi}]=0.$$

Let us denote $G_{\mathbb{C}}(\Lambda)=\exp(\mathfrak{g}_{\mathbb{C}}(\Lambda))\simeq (SL_2(\mathbb{C}))^r$ and $G_{\mathbb{R}}(\Lambda)=\linebreak \text{exp}(\mathfrak{g}_{\mathbb{R}}(\Lambda))=(SL_2(\mathbb{R}))^r$, which are subgroups of $G_{\mathbb{C}}$ and $G_{\mathbb{R}}$ respectively.
With the fixed reference point $o=\P(p)$, we denote $D(\Lambda)=G_\mathbb{R}(\Lambda)(o)$ and $S(\Lambda)=G_{\mathbb{C}}(\Lambda)(o)$ to be the corresponding orbits of these two subgroups, respectively. Then we have the following isomorphisms,
\begin{align}
&D(\Lambda)=G_{\mathbb{R}}(\Lambda)\cdot B/B\simeq G_{\mathbb{R}}(\Lambda)/G_{\mathbb{R}}(\Lambda)\cap V,\label{isomorphism1}\\
&S(\Lambda)\cap (N_+B/B)=(G_{\mathbb{C}}(\Lambda)\cap N_+)\cdot B/B\simeq G_{\mathbb{C}}(\Lambda)\cap N_+.
\label{isomorphism2}
\end{align}
With the above notations, we will show that
\begin{itemize}
\item[(i)] $D(\Lambda)\subseteq S(\Lambda)\cap (N_+B/B)\subseteq \check{D}$;
\item[(ii)] $D(\Lambda)$ is bounded inside $S(\Lambda)\cap(N_+B/B)$.
\end{itemize}

By Lemma \ref{puretype}, we know that for each pair of roots $\{e_{{\varphi_i}}, e_{{-\varphi_i}}\}$, there exists a positive integer $k$ such that either $e_{{\varphi_i}}\in \mathfrak{g}^{-k,k}\subseteq \mathfrak{n}_+$ and $e_{{-\varphi_i}}\in \mathfrak{g}^{k,-k}$, or $e_{{\varphi_i}}\in \mathfrak{g}^{k,-k}$ and $e_{{-\varphi_i}}\in \mathfrak{g}^{-k,k}\subseteq\mathfrak{n}_+$. For the simplicity of notations, for each pair of root vectors $\{e_{{\varphi_i}}, e_{{-\varphi_i}}\}$, we may assume the one in $\mathfrak{g}^{-k,k}\subseteq \mathfrak{n}_+$ to be $e_{{\varphi_i}}$ and denote the one in $\mathfrak{g}^{k,-k}$ by $e_{{-\varphi_i}}$. In this way, one can check that $\{\varphi_1, \dots, \varphi_r\}$ may not be a set in $\Delta^+_{\mathfrak{p}}$, but it is a set of strongly orthogonal roots in $\Delta_{\mathfrak{p}}$. In this case, for any two different vectors $e_{{\varphi_i}},e_{{\varphi_j}}$  in $\{e_{{\varphi_1}}, e_{{\varphi_2}}, \dots, e_{{\varphi_r}}\}$, the Hermitian inner product
\begin{align*}
-B(\theta(e_{\phi_i}),\bar{e_{\phi_j}})&=-B(-e_{\phi_i},e_{-\phi_j})\\
&=B(1/2[h_{\phi_i},e_{\phi_i}],e_{-\phi_j})\\
&=-B(1/2e_{\phi_i},[h_{\phi_i},e_{-\phi_j}])=0.
\end{align*} 
Hence the basis $\{e_{{\varphi_1}}, e_{{\varphi_2}}, \dots,
e_{{\varphi_r}}\}$ can be chosen as an orthonormal basis.

Therefore, we have the following description of the above groups,
\begin{align*}
G_{\mathbb{R}}(\Lambda)&=\exp(\mathfrak{g}_{\mathbb{R}}(\Lambda))\\
&=\exp(\text{Span}_{\mathbb{R}}\{ x_{{\varphi_1}}, y_{{\varphi_1}},\sqrt{-1}h_{{\varphi_1}}, \dots, x_{{\varphi_r}}, y_{{\varphi_r}}, \sqrt{-1}h_{{\varphi_r}}\})\\
G_{\mathbb{R}}(\Lambda)\cap V&=\exp(\mathfrak{g}_{\mathbb{R}}(\Lambda)\cap \mathfrak{v}_0)=\exp(\text{Span}_{\mathbb{R}}\{\sqrt{-1}h_{{\varphi_1}}, \cdot,\sqrt{-1} h_{{\varphi_r}}\})\\
G_{\mathbb{C}}(\Lambda)\cap N_+&=\exp(\mathfrak{g}_{\mathbb{C}}(\Lambda)\cap \mathfrak{n}_+)=\exp(\text{Span}_{\mathbb{C}}\{e_{{\varphi_1}}, e_{{\varphi_2}}, \dots, e_{{\varphi_r}}\}).
\end{align*}
Thus by the isomorphisms in \eqref{isomorphism1} and \eqref{isomorphism2}, we have
\begin{align*}
&D(\Lambda)\simeq \prod_{i=1}^r\exp(\text{Span}_{\mathbb{R}}\{x_{{\varphi_i}}, y_{{\varphi_i}}, \sqrt{-1}h_{{\varphi_i}}\})/\exp(\text{Span}_{\mathbb{R}}\{\sqrt{-1}h_{{\varphi_i}}\},\\
&S(\Lambda)\cap (N_+B/B)\simeq\prod_{i=1}^r\exp(\text{Span}_{\mathbb{C}}\{e_{{\varphi_i}}\}).
\end{align*}
Let us denote $G_{\mathbb{C}}(\varphi_i)=\exp(\text{Span}_{\mathbb{C}}\{e_{{\varphi_i}}, e_{{-\varphi_i}}, h_{{\varphi_i}})\simeq SL_2(\mathbb{C}),$ $S(\varphi_i)=\linebreak G_{\mathbb{C}}(\varphi_i)(o)$, and $G_{\mathbb{R}}(\varphi_i)\!=\!\exp(\text{Span}_{\mathbb{R}}\{x_{{\varphi_i}}, y_{{\varphi_i}}, \sqrt{-1}h_{{\varphi_i}}\})\!\simeq\! SL_2(\mathbb{R})$, $D(\varphi_i)\!=\!G_{\mathbb{R}}(\varphi_i)(o)$.

Now each point in $S(\varphi_i)\cap (N_+B/B)$ can be represented by
\begin{align*}
\text{exp}(ze_{{\varphi_i}})=\left[\begin{array}[c]{cc}1&z\\ 0& 1\end{array}\right] \quad \text{for some } z\in \mathbb{C}.
\end{align*}
Thus $S(\varphi_i)\cap (N_+B/B)\simeq \mathbb{C}$. In order to see $D(\varphi_i)$ in $G_\C/B$, we decompose each point in $D(\varphi_i)$ as follows. Let $z=a+bi$ for some $a,b\in\mathbb{R}$, then
\begin{align}\label{x to y}
\exp(ax_{{\varphi_i}}+by_{{\varphi_i}})&=\left[\begin{array}[c]{cc}\cosh |z|&\frac{z}{|z|}\sinh |z|\\ \frac{\bar{z}}{|z|}\sinh |z|& \cosh |z|\end{array}\right] \\
&=\left[\begin{array}[c]{cc}1&\frac{z}{|z|}\tanh |z|\\ 0&1\end{array}\right]\left[\begin{array}[c]{cc}(\cosh |z|)^{-1}&0\\0&\cosh |z|\end{array}\right]\nonumber\\
&\quad \left[\begin{array}[c]{cc}1& 0 \\ \frac{\bar{z}}{|z|}\tanh |z|&1\end{array}\right]\nonumber\\
\notag &=\exp\left[(\frac{z}{|z|}\tanh |z| )e_{{\varphi_i}}\right]\exp\left[-\log (\cosh |z|)h_{{\varphi_i}}\right]\\
\notag &\quad \exp\left[(\frac{\bar{z}}{|z|}\tanh |z|) e_{{-\varphi_i}}\right]\\
&\equiv \exp\left[(\frac{z}{|z|}\tanh |z| )e_{{\varphi_i}}\right] \ (\text{mod }B).\nonumber
\end{align}
So the elements of $D(\varphi_i)$ in $G_\C/B$ can be represented by $\exp[(z/|z|)(\tanh |z|) e_{{\varphi_i}}]$, i.e.
\begin{align*}
\left[\begin{array}[c]{cc}1&  \frac{z}{|z|}\tanh |z|\\0&1\end{array}\right],
\end{align*}
in which $\frac{z}{|z|}\tanh |z|$ is a point in the unit disc $\mathfrak{D}$ of the complex plane. Therefore  $D(\varphi_i)$ is a unit disc $\mathfrak{D}$ in the complex plane $S(\varphi_i)\cap (N_+B/B)$.
 Therefore $$D(\Lambda)\simeq \mathfrak{D}^r\quad\text{and}\quad S(\Lambda)\cap N_+\simeq \mathbb{C}^r. $$ So we have obtained both (i) and (ii).
As a consequence, we get that for any $q\in \check{\T}$, $\P_{+}(q)\in D(\Lambda)$. This implies
\begin{align*}d_E(\P_{+}(p), \P_{+}(q))\leq \sqrt{r}
\end{align*} where $d_E$ is the Eulidean distance on $S(\Lambda)\cap (N_+B/B)$.

To complete the proof, we only need to show that $S(\Lambda)\cap (N_+B/B)$ is totally geodesic in $N_+B/B$. In fact, the tangent space of $N_+$ at the base point is $\mathfrak{n}_+$ and the tangent space of $S(\Lambda)\cap (N_+B/B)$ at the base point is $\text{Span}_{\mathbb{C}}\{e_{{\varphi_1}}, e_{{\varphi_2}}, \dots, e_{{\varphi_r}}\}$. Since $\text{Span}_{\mathbb{C}}\{e_{{\varphi_1}}, e_{{\varphi_2}}, \dots, e_{{\varphi_r}}\}$ is a Lie subalgebra of $\mathfrak{n}_+$, the corresponding orbit $S(\Lambda)\cap (N_+B/B)$ is totally geodesic in $N_+B/B$. Here recall that the basis $\{e_{{\varphi_1}}, e_{{\varphi_2}}, \dots, e_{{\varphi_r}}\}$ is an orthonormal basis.
\end{proof}

Although not needed in the proof of the above theorem, we also show that the above inclusion of $D(\varphi_i)$ in $ D$ is totally geodesic in $D$ with respect to the Hodge metric. In fact, the tangent space of $D(\varphi_i)$ at the base point is  $\text{Span}_{\mathbb{R}}\{x_{{\varphi_i}}, y_{{\varphi_i}}\}$ which satisfies
\begin{align*}
&[x_{{\varphi_i}}, [x_{{\varphi_i}}, y_{{\varphi_i}}]]=4y_{{\varphi_i}},\\
&[y_{{\varphi_i}}, [y_{{\varphi_i}}, x_{{\varphi_i}}]]=4x_{{\varphi_i}}.
\end{align*}
So the tangent space of $D(\varphi_i)$ forms a Lie triple system, and consequently $D(\varphi_i)$ gives a totally geodesic submanifold of $D$. 

The fact that the exponential map of a Lie triple system gives a totally geodesic submanifold of $D$ is from (cf. \cite[Ch4, $\S$7]{Hel}), and we note that this result still holds true for locally homogeneous spaces instead of only for symmetric spaces. 
The pull-back of the Hodge metric on $D(\varphi_i)$ is $G(\varphi_i)$ invariant metric, therefore must be the Poincare metric on the unit disc. In fact, more generally, we have
\begin{lemma}
If $\tilde{G}$ is a subgroup of $G_{\mathbb{R}}$, then the orbit $\tilde{D}=\tilde{G}(o)$ is a totally geodesic submanifold of $D$, and the induced metric on $\tilde{D}$ is $\tilde{G}$ invariant.
\end{lemma}
\begin{proof}
Firstly, $\tilde{D}\simeq \tilde{G}/(\tilde{G}\cap V)$ is a quotient space. The induced metric of the Hodge metric from $D$ is $G_{\mathbb{R}}$-invariant, and therefore $\tilde{G}$-invariant.
Now let $\gamma:\,[0,1]\to \tilde{D}$ be any geodesic, then there is a local one parameter subgroup $S:\,[0,1] \to\tilde{G}$ such that, $\gamma(t)=S(t)\cdot \gamma(0)$. On the other hand, because $\tilde{G}$ is a subgroup of $G_{\mathbb{R}}$, we have that $S(t)$ is also a one parameter subgroup of $G_{\mathbb{R}}$, therefore the curve $\gamma(t)=S(t)\cdot \gamma(0)$ also gives a geodesic in $D$.
Since geodesics on $\tilde{D}$ are also geodesics on $D$, we have proved $\tilde{D}$ is totally geodesic in $D$.
\end{proof}
The following corollary is important to us.
\begin{corollary}\label{X to Y}
The underlying real manifold of $\mathrm{exp}(\mathfrak{p}_+)\cap D$ is diffeomorphic to $G_{\mathbb{R}}/K\simeq \exp(\mathfrak{p}_{0})$.
\end{corollary}
\begin{proof}
As discussed at the beginning of the proof of Lemma \ref{abounded}, for any $s\in \exp(\mathfrak{p}_+)\cap D$, there is an $X \in \mathfrak{p}_{0}$ such that $\exp (X)\bar{o}=s$, where $\bar{o}=P_{+}(o)$.
From equation \eqref{x to y}, one sees that there exists $Y\in \mathfrak{p}_+$, satisfying $$X =T_0( Y +\tau_0(Y))$$ for a unique real number $T_0$, such that $\exp (X)\bar{o} = \exp (Y)bar{o}$. Hence we have a diffeomorphism $$\exp(\mathfrak{p}_{+}) \cap D\to \exp(\mathfrak{p}_0)\simeq G_{\mathbb{R}}/K,$$ by mapping $\exp (Y) \bar{o}$ to $\exp (X)\bar{o}$ with the relation that $X =T_0( Y +\tau_0(Y))$.
For a more detailed description of the above explicit correspondence from $\mathfrak{p}_0$ to $\mathfrak{p}_{+}$ in proving the Harish-Chandra embedding, please see Lemma~7.11 in pages 390--391 in \cite{Hel}, pages 94--97 of \cite{Mok}, or the discussion in pages 463--466 in \cite{Xu}. 
\end{proof}

As proved in Proposition 3 of Chapter 2 in \cite{Schwartz}, for the extended period map $\PP:\, \TT \to D$, there is  a Whitney stratification $$\TT=\cup_{i}\T_i$$ such that if $\Phi_i=\PP|_{\T_i}$, the rank of the tangent map $d\Phi_i$ is constant on $\T_i$. Note that 
the stratification is narrow in the sense that for any open neighborhood $U$ of any point in $\TT$, it induces a Whitney stratification of $U$. We will define the image $\PP(U)\subset D$ of a small open neighborhood $U$ of  $\TT$  under the period map $\PP$ as a {\em horizontal slice}, which is given by the union of the image of each $\T_i$ restricted to the neighborhood $U$. 

Note that we can always take $U$ arbitrarily small as needed. Clearly we have $$\PP(U) =\cup_i L_i$$ where each $L_i= \PP(\T_i\cap U)$  is a smooth manifold when $U$ is small enough, and they induce a Whitney stratification of $\PP(U)$ by the continuity of the tangent map of $\PP$. We remark that locally there are only finitely many strata $\T_i$, since $\TT$ is finite dimensional and each $L_i$ is an integral submanifold of the horizontal distribution induced by the period map.  See for example,  page 36 in \cite{Pflaum} about details related to the Whitney stratifications.  

In fact here we can also directly use the Whitney stratification of $\PP(U)$ for the proof of the following Lemma, while using the Whitney stratification for $\TT$ makes the geometric picture of the period map more transparent. 

As local Torelli theorem holds for Calabi--Yau manifolds, the extended period map $\PP$ is nondegenerate on $\T_{m}\subset \TT$. 
Since $i_m:\, \T \to \T_m$ is a covering, $i_{m}(U)$ is a neighborhood of $i_m(q)$ in $\T_m\subset \TT$ if $U$ is taken as a small enough neighborhood of a point $q$ in $\T$. Therefore we also call the image $\P(U)=\P^H_m(i_m(U))$ a horizontal slice.





Note that, at any point  $t\in L_{i}$, by the Griffiths transversality, we know that the corresponding real tangent spaces satisfy $$\text{T}_tL_{i} \subset \text{T}_{h,t}D\subset \text{T}_{\bar t}G_\mathbb{R}/K\simeq \mathfrak{p}_0 $$ where $\bar{t}=\pi(t)$. Here the inclusion $\text{T}_{h,t}D\subset \text{T}_{\bar t}G_\mathbb{R}/K$ is induced by the tangent map of $\pi$ at $t$.
Therefore the tangent map of $\pi|_{L_i}:\, L_i \to G_{\mathbb R}/K $ at $t\in L_i$ is injective, and $\pi$ is injective in a small neighborhood of $t$ in $L_i$. From this one can see that the following lemma is a straightforward corollary of the Griffiths transversality.

\begin{lemma}\label{locally injective}
The projection map $\pi:\, D \to G_\mathbb{R}/K$ is injective on horizontal slices. That is to say that  for any horizontal slice $\PP(U)$ around any point $s=\PP(q)$ with $\PP(U) =\cup_i L_i$, we can take the open neighborhood $U$ of $q\in \T$  small enough such that $\pi$ is injective on each $L_{i}$. 
\end{lemma}
\begin{proof} 
First from the above discussion, we see that the lemma is an obvious corollary from the Griffiths transversality,  if $\PP(U)$ is smooth. The 
proof for general case is essentially the same, except that we need to use the Whitney stratification of $\PP(U)$ and apply the Griffiths transversality on each stratum. This should be standard in stratified spaces as discussed, for example,  in Section 3.8 of Chapter 1 in \cite{Pflaum}. We give the detailed argument for reader's  convenience.

Let $s=\PP(q)\in D$ and $U$ be a small  open neighborhood of $q$. As described above, we have the Whitney stratification $\PP(U)=\cup_{i} L_{i}$ and each $L_i$ can be identified to the image $\PP(\T_i\cap U)$ of the stratum $\T_i$.  

From Theorem 2.1.2 of \cite{Pflaum}, we know that the tangent bundle $\text{T}\PP(U)$\linebreak is  a stratified space with a smooth structure, such that the projection\linebreak $\text{T}\PP(U)\to \PP(U)$ is smooth and a morphism of stratified spaces.
For any sequence of points $\{s_k\}$ in a horizontal slice $L_i$ converging to $s$, the limit of the tangent spaces, $$\lim_{k\to \infty} \text{T}_{s_k}L_i=\text{T}_sL_i$$ exists by the Whitney conditions, and is defined as the generalized tangent space at $s$ in page 44 of \cite{GM}. Also see the discussion in page 64 of \cite{Pflaum}. Denote $\bar{s} =\pi(s)$. With these notations understood, and by the Griffiths transversality, we get the following relations for the corresponding real tangent spaces,
\begin{align*}
\text{T}_s\PP(U)=\cup_i\text{T}_sL_i\simeq (d\PP)_{q}(\text{T}_q\T) &\subset (\mathfrak{g}^{-1,1}\oplus \mathfrak{g}^{1,-1})\cap \mathfrak{g}_{0}\\
&\subset \text{T}_{\bar s}G_\mathbb{R}/K\simeq \mathfrak{p}_0.
\end{align*}
This implies that the tangent map of $$\pi|_{\PP(U)}:\, \PP(U) \to G_{\mathbb R}/K $$ at $s$ is injective in the sense of stratified space, or equivalently it is injective on each $\text{T}_sL_i$ considered as generalized tangent space.

Therefore we can find a small open neighborhood $V$ of $s$ in $D$, such that the restriction of $\pi$ to $\PP(U)\cap V$,  $$\pi|_{\PP(U)\cap V}:\, \PP(U)\cap V \to G_{\mathbb R}/K,$$ is an immersion in the sense of stratified spaces, or equivalently injective on each stratum $L_i\cap U$. 
Now we take $U$ in $\T$ containing $q$ small enough such that $\PP(U)\subset V$.
With such a choice of $U$,  $\pi$ is injective on the horizontal slice $\PP(U)$ in the sense of stratified spaces, and hence injective on each stratum $L_{i}$.
\end{proof}

\begin{lemma}\label{lemma of locallybounded}
For any $z\in \P_{+}({\check{\T}})\subset \exp(\mathfrak{p}_{+})\cap D$, we have $$P_{+}^{-1}(z)\cap\linebreak \P({\check{\T}})=\pi^{-1}(z')\cap \P({\check{\T}}),$$ where $z'=\pi(z)\in G_\mathbb{R}/K$.
\end{lemma}
\begin{proof}
From Corollary \ref{X to Y}, 
the underlying real manifold of $\mathrm{exp}(\mathfrak{p}_+)\cap D$ is diffeomorphic to $G_{\mathbb{R}}/K\simeq \exp(\mathfrak{p}_{0})$. 
As described there, this diffeomorphism is given explicitly by identifying the point 
$$\text{exp}(Y)\bar{o}\in \mathrm{exp}(\mathfrak{p}_+)\cap D$$
with the point $\text{exp}(X)\in \exp(\mathfrak{p}_{0})$, where the vectors $X\in \mathfrak{p}_{0}$ and $Y\in \mathfrak{p}_+$ satisfy the relation that $X =T_0( Y +\tau_0(Y))$ for certain real number $T_{0}$.   

On the other hand, from the definition of the Hodge metric on $D$ in page~297 of \cite{GS}, we know that the natural projection $\pi:\, D\to G_{\mathbb{R}}/K$ is a Riemannian submersion with the natural homogeneous metrics on $D$ and $G_{\mathbb{R}}/K$.  For more details about this, see also Section 2 of \cite{JostYang}.

Then the real geodesic $$c(t)= \exp(tX)$$ in $\mathrm{exp}(\mathfrak{p}_+)\cap D$ with $X\in \mathfrak{p}_{0}$ connecting the based point $\bar{o}$ and any point $z\in \mathrm{exp}(\mathfrak{p}_+)\cap D$ is the horizontal lift of the geodesic $\pi(c(t))$ in $G_{\mathbb{R}}/K$. This is a basic fact in Riemannian submersion as given in, for example, Corollary~26.12 in page 339 of \cite{Michor}. 

Hence the natural projection $\pi:\,D\to G_\mathbb{R}/K$ maps $c(t)$ isometrically to $\pi(c(t))$.
From this one sees that the projection map $\pi$, when restricted
to the underlying real manifold of $\mathrm{exp}(\mathfrak{p}_+)
\cap D$, is given by the diffeomorphism 
$$ \pi_{+}:\, \mathrm{exp}(\mathfrak{p}_+) \cap D\longrightarrow \mathrm{exp}(\mathfrak{p}_0) \stackrel{\simeq}{\longrightarrow} G_\mathbb{R}/K,$$
and the diagram
$$
\xymatrix{ N_{+}\cap D \ar[r]^-{\pi} \ar[d]^-{P_{+}} & G_\mathbb{R}/K \\
\mathrm{exp}(\mathfrak{p}_+)\cap D \ar[ur]_{\pi_{+}}& ,
}$$
is commutative.
Therefore one concludes that any two points in $N_+\cap D$ are mapped to the same point in $\mathrm{exp}(\mathfrak{p}_+) \cap D$ via $P_+$, if and only if they are are mapped to the same point in $G_\mathbb{R}/K$ via $\pi$. Hence for any $$z\in \P_{+}({\check{\T}})\subset \exp(\mathfrak{p}_{+})\cap D,$$ the projection map $P_+$ maps the fiber $\pi^{-1}(z')\cap \P({\check{\T}}) $ onto the point $z\in \exp(\mathfrak{p}_+) \cap D$, where $z'=\pi(z)\in G_\mathbb{R}/K$.
\end{proof}



\begin{theorem}\label{locallybounded}
The image of the restriction of the period map $$\Phi :\, \check{\T}
\to N_+\cap D$$ is bounded in $N_+$ with respect to the Euclidean
metric on $N_+$.
\end{theorem}
\begin{proof}
In Lemma \ref{abounded}, we have already proved that the image of
$\P_{+}=P_{+}\circ \P$ is bounded with respect to the Euclidean metric on
$\exp(\mathfrak{p}_{+})\subseteq N_+$. Now together with the Griffiths transversality, we
will deduce the boundedness of the image of $\P :\, \check{\T} \to
N_+\cap D$ from the boundedness of the image of $\P_{+}$.

Our proof can be divided into two steps. It is an elementary argument to apply the Griffiths transversality on $\TT$.

(i) We claim that there are only finite points in the inverse image $$(P_{+}|_{\P(\check{\T})})^{-1}(z)$$ for any $z\in \P_{+}(\check{\T})$. Here $P_{+}|_{\P(\check{\T})}$ denotes the restriction of $P_+$ to $\P(\check{\T})$.

Otherwise, by Lemma \ref{lemma of locallybounded}, we have $\{q_i\}_{i=1}^{\infty}\subseteq \check{\T}$ and 
$$\{y_i=\P(q_i)\}_{i=1}^{\infty}\subseteq (P_{+}|_{\P(\check{\T})})^{-1}(z)$$ with limiting point $y_{\infty}\in \pi^{-1}(z')\simeq K/V$, since $K/V$ is compact. We project the points $q_i$ to $q'_i\in \Z$ via the universal covering map $\pi_{m}:\, \T \to \Z$. There must be infinite many $q'_i$'s. Otherwise, we have a subsequence $\{q_{j_k}\}$ of $\{q_j\}$ such that $\pi_{m}(q_{j_k})=q'_{i_0}$ for some $i_0$ and
$$y_{j_k}=\P(q_{j_k})=\gamma_{k}\P(q_{j_0})=\gamma_{k}y_{j_0},$$
where $\gamma_k \in \Gamma$ is the monodromy action. Since $\Gamma$ is discrete, the subsequence $\{y_{j_k}\}$ is not convergent, which is a contradiction.

Now we project the points $q_i$ on $\Z$ via the universal covering map $\pi_{m}:\, \T \to \Z$ and still denote them by $q_i$ without confusion. Then the sequence $\{q_i\}_{i=1}^{\infty}\subseteq \Z$ has a limiting point $q_\infty$ in $\bar{\mathcal{Z}}_{m}$, where $\bar{\mathcal{Z}}_{m}$ is the compactification of $\Z$.
By continuity the period map $\Phi :\,  \Z \to D/\Gamma$ can be extended over $q_\infty$ with $$\Phi(q_\infty)=\pi_D(y_\infty) \in D/\Gamma,$$ where $\pi_D :\, D\to D/\Gamma$ is the projection map. Thus $q_\infty$ lies $\ZZ$. 

Now we regard the sequence $\{q_i\}_{i=1}^{\infty}$ as a convergent sequence in $\ZZ$ with limiting point $q_\infty \in \ZZ$. Let $V$ be a small open neighborhood of $q_{\infty}$, and $\tilde{q}_{\infty}$ be its lifting to $\T^{H}_{m}$. Let $U$ be a small open neighborhood of $\tilde{q}_{\infty}$ such that $\pi_{m}^{H}:\, U\to V$ is a diffeomorphism.

 We can choose a sequence $\{\tilde{q}_i\}_{i=1}^{\infty}\subseteq \TT$ with limiting point $\tilde{q}_\infty\in \TT$ such that $\tilde{q}_i\in U$ maps to $q_i\in V$ for $i$ large via the universal covering map $\pi_{m}^{H} :\,  \TT\to \ZZ$ and $$\PP(\tilde{q}_i)=y_i \in D,$$ for $i\ge 1$ and $i=\infty$. Since the extended period map $$\PP :\, \TT \to D$$ still satisfies the Griffiths transversality by Lemma \ref{extendedtransversality}, we can choose the neighborhood $U$ of $\tilde{q}_\infty$ small enough such that $\PP(U)=\cup_{i}L_{i}$ is a disjoint union of Whitney stratifications in Lemma \ref{locally injective} and $\pi$ is injective on each $L_{i}$.
By passing to a subsequence, we may assume that the points $\tilde{q}_i$ for $i$ sufficiently large are mapped into some stratum $L_i$, which is a contradiction.

Denote $P_{+}|_{\PP(\TT)}$ to be the restricted map $$P_{+}|_{\PP(\TT)}:\, N_{+}\cap \PP(\TT) \to \exp(\mathfrak{p}_{+})\cap D.$$
In fact, a similar argument also proves that there are only finite points in\linebreak $(P_{+}|_{\PP(\TT)})^{-1}(z)$, for any $z\in \P_{+}(\check{\T})$. Furthermore, we have the following conclusion.

(ii) The restricted map $$P_{+}|_{\PP(\TT)}:\, N_{+}\cap \PP(\TT) \to \exp(\mathfrak{p}_{+})\cap D$$ is a finite holomorphic ramified covering map onto its image. The proof is a direct application of some basic results in the book of Grauert--Remmert \cite{GR}.

From (i) and the definition of finite map in analytic geometry as given in page 47 of  \cite{GR}, we only need to show that $P_{+}|_{\PP(\TT)}$ is closed.
In fact, we know that $${\Phi}_{{\mathcal{Z}^H_m}}:\, {\mathcal{Z}}^H_m\to D/\Gamma$$ is a proper map by the result of Griffiths, and hence ${\Phi}_{{\mathcal{Z}^H_m}}( {\mathcal{Z}^H_m})$ is closed in $D/\Gamma$. So $$\PP(\TT)=\pi_{D}^{-1}({\Phi}_{{\mathcal{Z}^H_m}}( {\mathcal{Z}^H_m}))$$ is also closed in $D$, where $\pi_{D} : \, D \to D/\Gamma$ is the projection map.
Hence any closed subset $E$ of $\PP(\TT)$ is also a closed subset of $D$. 
Since the natural projection map $$\pi:\, D \to G_{\mathbb{R}}/K$$ is a proper map, one sees that $\pi$ is a closed map, which implies that $\pi(E)$ is closed in $G_{\mathbb{R}}/K$. Moreover, from the proof of Lemma \ref{lemma of locallybounded}, one  sees that $P_{+}(E)$ is diffeomorphic to $\pi(E)$ through the diffeomorphism $$\exp(\mathfrak{p}_{+})\cap D\simeq \mathfrak{p}_0\simeq G_{\mathbb R}/K,$$ which implies that $P_{+}(E)$ is closed in $\exp(\mathfrak{p}_{+})\cap D$.
Therefore we have proved that $P_{+}|_{\PP(\TT)}$ is a closed map.

As proved in page 171 of \cite{GR},  the image of an irreducible complex variety under holomorphic map is still irreducible. 
Since $\TT$ is irreducible, we know that $\PP(\TT)$ is an irreducible analytic subvariety of $\check{D}$. 
The intersection $N_{+}\cap \PP(\TT)$ is equal to $\PP(\TT)$ minus the proper analytic subvariety $\PP(\TT)\cap (\check{D}\setminus N_{+})$.
Hence, from the results in page 171 of  \cite{GR}, we get that  $N_{+}\cap \PP(\TT)$ and $P_{+}(N_{+}\cap \PP(\TT))$ are both irreducible. 

 By the result in page 179  of \cite{GR}, the projection map 
$$P_{+}|_{\PP(\TT)} :\, N_{+}\cap \PP(\TT)\to P_{+}(N_{+}\cap \PP(\TT))$$ 
is a finite holomorphic map, or equivalently, finite ramified covering map. Let $r(z)$ be the cardinality of the fiber $(P_{+}|_{\PP(\TT)})^{-1}(z)$ for any $$z\in P_{+}(N_{+}\cap \PP(\TT))$$ outside the ramification locus. From the result in page 135 in \cite{GR}, we know that $r(z)=r$ is constant on $P_{+}(N_{+}\cap \PP(\TT))$ outside the ramified locus which is an analytic subset. From the above proof of (i), we also know that $r$ is finite. Hence $P_{+}|_{\PP(\TT)}$ is an $r$-sheeted ramified covering,
which together with Lemma \ref{abounded},  implies that the image $\P(\check{\T})\subseteq N_+\cap D$ is bounded.
\end{proof}

In early versions of this paper we wrote a more elementary proof of (ii), which only used the Griffiths transversality and a simple limiting argument similar to the proof of (i). The new proof we give here is more illuminating, since it has the advantage of involving more geometric structures of the image of period map and period domain.

\begin{lemma}\label{transversal}Let $p\in\mathcal{T}$ be the base point with $\Phi(p)=\{F^n_p\subseteq F^{n-1}_p\subseteq \cdots\subseteq F^0_p\}.$ Let $q\in \mathcal{T}$ be any point with $\Phi(q)=\{F^n_q\subseteq F^{n-1}_q\subseteq \cdots\subseteq F^0_q\}$, then $\Phi(q)\in N_+$ if and only if $F^{k}_q$ is isomorphic to $F^k_p$ for all $0\leq k\leq n$.
\end{lemma}
\begin{proof}For any $q\in \mathcal{T}$, we choose an arbitrary adapted basis $\{\zeta_0, \dots, \zeta_{m-1}\}$ for the given Hodge filtration $\{F^n_q\subseteq F^{n-1}_q\subseteq\cdots\subseteq F^0_q\}$. Recall that $\{\eta_0, \dots,\linebreak \eta_{m-1}\}$ is the adapted basis for the Hodge filtration $\{F^n_p \subseteq F^{n-1}_p\subseteq \cdots\subseteq F^0_p\}$ for the base point $p$. Let $[A^{i,j}(q)]_{0\leq i,j\leq n}$ be the transition matrix between the basis $\{\eta_0,\dots, \eta_{m-1}\}$ and $\{\zeta_0, \dots, \zeta_{m-1}\}$ for the same vector space $H^{n}(M, \mathbb{C})$, where $A^{i,j}(q)$ are the corresponding blocks.

Recall that elements in $N_+$ and $B$ have matrix representations with the fixed adapted basis at the base point: elements in $N_+$ can be realized as nonsingular block lower triangular matrices with identity blocks in the diagonal; elements in $B$ can be realized as nonsingular block upper triangular matrices. Therefore  $$\Phi(q)\in N_+=N_+B/B\subseteq \check{D}$$ if and only if its matrix representation $[A^{i,j}(q)]_{0\leq i,j\leq n}$ can be decomposed as $L(q)\cdot U(q)$, where $L(q)$ is a nonsingular block lower triangular matrix with identities in the diagonal blocks, and $U(q)$ is a nonsingular block upper triangular matrix. 

By basic linear algebra, we know that $[A^{i,j}(q)]$ has such decomposition if and only if $\det[A^{i,j}(q)]_{0\leq i, j\leq k}\neq 0$ for any $0\leq k\leq n$. In particular, we know that $[A(q)^{i,j}]_{0\leq i,j\leq k}$ is the transition map between the bases of $F^k_p$ and $F^k_q$. Therefore, $\det([A(q)^{i,j}]_{0\leq i,j\leq k})\neq 0$ if and only if $F^k_q$ is isomorphic to $F^k_p$.
\end{proof}\begin{lemma}\label{codimension}The subset $\check{\mathcal{T}}$ is an open dense submanifold in $\mathcal{T}$, and $\mathcal{T}\backslash \check{\mathcal{T}}$ is an analytic subvariety of $\mathcal{T}$ with $\text{codim}_{\mathbb{C}}(\mathcal{T}\backslash \check{\mathcal{T}})\geq 1$.
\end{lemma}
\begin{proof}
From Lemma \ref{transversal}, one can see that $\check{D}\setminus N_+\subseteq \check{D}$ is defined as an analytic subvariety by equation
\begin{align*}
\Pi_{0\leq k\leq n}\det [A^{i,j}(q)]_{0\leq i,j\leq k}=0 .
\end{align*}
Therefore $N_+$ is dense in $\check{D}$, and that $\check{D}\setminus N_+$ is an analytic subvariety, which is closed in $\check{D}$ and $\text{codim}_{\mathbb{C}}(\check{D}\backslash N_+)\geq 1$. We consider the period map  $\Phi:\,\mathcal{T}\to \check{D}$ as a holomorphic map to $\check{D}$, then $$\mathcal{T}\setminus \check{\mathcal{T}}=\Phi^{-1}(\check{D}\setminus N_+)$$ is the preimage of $\check{D}\setminus N_+$ of the holomorphic map $\Phi$. Therefore $\mathcal{T}\setminus \check{\mathcal{T}}$ is also an analytic subvariety and a closed set in $\mathcal{T}$. Because $\mathcal{T}$ is smooth and connected, $\mathcal{T}$ is irreducible. If $\dim(\mathcal{T}\setminus \check{\mathcal{T}})=\dim\mathcal{T}$, then $\mathcal{T}\setminus \check{\mathcal{T}}=\mathcal{T}$ and $\check{\mathcal{T}}=\emptyset$, but this contradicts to the fact that the reference point $p$ is in $\check{\mathcal{T}}$. Thus we conclude that $\dim(\mathcal{T}\!\setminus\! \check{\mathcal{T}})\!<\!\dim\mathcal{T}$, and consequently $\text{codim}_{\mathbb{C}}(\mathcal{T}\backslash \check{\mathcal{T}})\!\geq\! 1$.
\end{proof}
\begin{remark}
We can also prove this lemma in a more direct manner. By using notation in the proof of Lemma \ref{transversal}, for any $q\in\mathcal{T}$, let $U_q$ be a neighborhood of $q$ such that all Hodge bundles $\{F^k\}_{0\leq k\leq n}$ are trivial over $U_q$. For any $r\in U_q$, let $\{\zeta(r)=\{\zeta_0(r), \dots, \zeta_{m-1}(r)\}\}_{r\in U_q}$ be a holomorphic family of adapted bases for the Hodge filtrations over $U_q$, where $\zeta(r)$ is an adapted basis for the Hodge filtration at $r\in U_q$. Let $\{A(r)\}_{r\in U_q}$ be the holomorphic family of transition matrices, where $A(r)$ is the transition matrix between the adapted basis $\eta$ to the Hodge filtration at the reference point $p$ and the adapted basis $\zeta(r)$ to the Hodge filtration at any point $r\in U_q$. 

From the definition of $\check{\mathcal{T}}$, we get that $r\in U_q\setminus (U_q\cap \check{\mathcal{T}})$ if and only if $r$ satisfies the following local holomorphic equation,
\begin{align*}
\Pi_{0\leq k\leq n}\det [A^{i,j}(r)]_{0\leq i,j\leq k}=0.
\end{align*}
Since $\mathcal{T}$ is irreducible and $\check{\mathcal{T}}\neq \emptyset$, we have that $\mathcal{T}\setminus\check{\mathcal{T}}$ is a divisor on $\mathcal{T}$. Therefore the complex codimension of $\mathcal{T}\setminus\check{\mathcal{T}}$ in $\mathcal{T}$ is greater than or equal to $1$.
\end{remark}
\begin{corollary}\label{image}The image of $$\Phi:\mathcal{T}\rightarrow D$$ lies in $N_+\cap D$ and is bounded with respect to the Euclidean metric on $N_+$.
\end{corollary}
\begin{proof}According to Lemma \ref{codimension}, $\mathcal{T}\backslash\check{\mathcal{T}}$ is an analytic subvariety of $\mathcal{T}$ and the complex codimension of $\mathcal{T}\backslash\check{\mathcal{T}}$ is at least one. By Theorem \ref{locallybounded}, the holomorphic map $\Phi:\,\check{\mathcal{T}}\rightarrow N_+\cap D$ is bounded in $N_+$ with respect to the Euclidean metric. Thus by the Riemann extension theorem, there exists a holomorphic map $\Phi': \,\mathcal{T}\rightarrow N_+\cap D$ such that $$\Phi'|_{{\check{\mathcal{T}}}}=\Phi|_{{\check{\mathcal{T}}}}.$$ Since as holomorphic maps, $\Phi'$ and $\Phi$ agree on the open subset $\check{\mathcal{T}}$, they must be the same on the entire $\mathcal{T}$. Therefore, the image of $\Phi$ is in $N_+\cap D$, and the image is bounded with respect to the Euclidean metric on $N_+.$ As a consequence, we also get $\mathcal{T}=\check{\mathcal{T}}=\Phi^{-1}(N_+).$ 
\end{proof}

Recall that in Section \ref{Hodge metric completion}, we have proved that $\Phi_{m}:\, \T_{m}\to D$ is holomorphic with $$\Phi_{m}(\T_{m})=\PP(i_{m}(\T))=\Phi(\T)$$ and $\text{codim}_{\C}(\TT\setminus \T_{m})\ge 1$. Corollary \ref{image} implies that the image of $$\P_m:\, \T_m \to N_+\cap D$$ is bounded in $N_+$.

\begin{corollary}\label{extended-boundedness}The image of $$\PP:\TT\!\rightarrow\! D$$ lies in $N_+\!\cap\! D$ and is bounded with respect to the Euclidean metric on $N_+$.
\end{corollary}
\begin{proof}
This follows directly from boundedness of  $\Phi$, and that $\Phi_m^H$ is a continuous extension of $\Phi$ to $\T^H_m$.  Indeed let $q\in \T^H_m$ and $\{ q_n\}$ be a sequence of points in $\T$ with limit $q$.
Then we have $$\P^H_m(q)=\lim_{n\to\infty}\, \P(q_n).$$ Therefore the boundedness of $\P$ gives the boundedness of $\P^H_m$.
\end{proof}

\section{Affine structures and injectivity of extended period map}\label{Holomorphic affine}
In Section \ref{affine on T},  we introduce the abelian subalgebra $\mathfrak{a}$, the abelian Lie group $A=\text{exp}(\mathfrak{a})$, and the projection map $P:\, N_+\cap D\to A\cap D$.
Then from local Torelli for Calabi--Yau manifolds, we show that the holomorphic map $$\Psi:\, {\T}\to A\cap D \subset A\simeq \C^{N}$$ defines a global affine structure on $\T$.

In Section \ref{affine on TH}, we consider the extended period map $$\Psi_{m}^{H}:\, \TT \to A\cap D,$$ where $\Psi_{m}^{H}=P\circ \PP$. Then by using the affine structure, local Torelli for Calabi--Yau manifolds and extension of Hodge bundles we show that $\Psi_{m}^{H}$ is nondegenerate and hence defines a global affine structure on $\TT$. A lemma of Griffiths-Wolf in \cite{GW} tells us that the completeness of $\TT$ with Hodge metric implies that $$\Psi_{m}^{H}:\, \TT \to A\cap D$$ is a covering map.
In Section \ref{injective}, we prove that $\Psi^H_{m}$ is an injection by using the Hodge metric completeness of $\T^{H}_{m}$, and the global holomorphic affine structure on $\mathcal{T}^H_{m}$. As a corollary, we show that the holomorphic map $\Phi_{m}^H$ is an injection.

\subsection{Affine structure on the Teichm\"uller space}\label{affine on T}
Let us consider $$\mathfrak{a}=d\P_{p}(\text{T}_{p}^{1,0}\T)\subseteq \text{T}^{1,0}_oD\simeq \mathfrak{n}_+$$ where $p$ is the base point in $\T$ with $\P(p)=o$.
Then by Griffiths transversality, $\mathfrak{a}\subseteq \g^{-1,1}$ is an abelian subalgebra of $\mathfrak{n}_+$ determined by the tangent map of the period map
$$d\Phi :\, \text{T}^{1,0}{\T} \to \text{T}^{1,0}D.$$
 Consider the corresponding Lie group
$$A\triangleq \exp(\mathfrak{a})\subseteq N_+.$$
Then $A$ can be considered as a complex Euclidean subspace of $N_{+}$ with the induced Euclidean metric from $N_+$.



Define the projection map $P:\, N_+\cap D\to A \cap D$ by
$$ P=\text{exp}\circ p\circ \text{exp}^{-1}$$
where $\text{exp}^{-1}:\, N_+ \to \mathfrak{n}_+$ is the inverse of the isometry $\text{exp}:\, \mathfrak{n}_+ \to N_+$, and $p:\, \mathfrak{n}_+\to \mathfrak{a}$ is the projection map from the complex Euclidean space $\mathfrak{n}_+$  to its  Euclidean subspace $\mathfrak{a}$.

The period map $\P : \, {\T}\to N_+\cap D$ composed
with the projection map $P$ gives a holomorphic map
\begin{equation}\label{maptoA}
\Psi :\,  {\T} \to A\cap D
\end{equation}
where $\Psi=P \circ \P$.
Similarly we define $\Psi_{m}^{H} :\, \TT \to A\cap D$ by  $\Psi_{m}^{H}= P\circ \Phi_{m}^{H}$.
We will prove that the map in \eqref{maptoA} defines a global affine structure on the Teichm\"uller space $\T$. First we review the definition of complex affine structure on a complex manifold.

\begin{definition}
Let $M$ be a complex manifold of complex dimension $n$. If there
is a coordinate cover $\{(U_i,\,\phi_i);\, i\in I\}$ of M such
that $\phi_{ik}=\phi_i\circ\phi_k^{-1}$ is a holomorphic affine
transformation on $\mathbb{C}^n$ whenever $U_i\cap U_k$ is not
empty, then $\{(U_i,\,\phi_i);\, i\in I\}$ is called a complex
affine coordinate cover on $M$ and it defines a holomorphic affine structure on $M$.
\end{definition}

Now we are ready to prove the following theorem which is the main result of this subsection.
\begin{theorem}\label{affine structure}
The holomorphic map $\Psi:\, \T\to A\cap D\subset A\simeq \C^{N}$ defines a global holomorphic affine structure on $\T$.
\end{theorem}
\begin{proof} Let $\mathcal{H}$ denote the Hodge subbundle
$\text{Hom}({F}^n, {F}^{n-1}/{F}^n)$. Recall that we have used the same notations for Hodge bundles on $D$ and $\T$.
Then by local Torelli theorem
for Calabi--Yau manifolds, we have the isomorphism
\begin{align}\label{local assumption''}
d\P:\, \text{T}^{1,0}\T\stackrel{\sim}{\longrightarrow} \mathcal{H}.
\end{align}

Let $\mathcal{H}_{A}$ be the restriction of $\mathcal{H}$ to $A\cap D$. Since $\mathcal{H}_{A}$ is a Hodge subbundle, we have the natural commutative diagram
\begin{align*}
\xymatrix{
\mathcal{H}_{A}|_{o} \ar[rr]^-{d\exp(X)}_{\simeq} \ar[d]^-{\simeq} && \mathcal{H}_{A}|_{s}\ar[d]\\
\text{T}_{o}^{1,0}(A\cap D) \ar[rr]^-{d\exp(X)}_{\simeq} && \text{T}_{s}^{1,0}(A\cap D)}
\end{align*}
at any point $s=\exp(X)$ in $A\cap D$ with $X\in \mathfrak{a}$. Here $d\exp(X)$ denote the tangent map of the left translation by $\exp(X)$. Here we have used the fact that the Hodge subbundle $\mathcal{H}_{A}$ is a homogeneous subbundle of the tangent bundle $\text{T}^{1,0}D$ which is also homogeneous. See \cite {Griffiths63} for the basic properties of homogeneous vector bundles.

Hence we have the isomorphism of tangent spaces $\text{T}_{s}^{1,0}(A\cap D)\simeq \mathcal{H}_{A}|_{s}$ for any $s\in A\cap D$, which implies that $$\text{T}^{1,0}(A\cap D)\simeq
 \mathcal{H}_{A}$$ as holomorphic bundles on $A\cap D$.
 Note that the tangent map of the projection map $P:\, N_+\cap D \to A\cap D$ maps the
tangent bundle of $D$,
$$\text{T}^{1,0}D\simeq \bigoplus_{k=1}^{n}\bigoplus_{l=1}^{k}\text{Hom}({F}^k/{F}^{k+1}, {F}^{k-l}/{F}^{k-l+1})$$
onto its subbundle $\text{T}^{1,0}(A\cap D)\simeq \mathcal{H}_{A}$. Therefore the tangent map $$d\Psi=dP\circ d\P:\, \text{T}^{1,0}\T \to \text{T}^{1,0}(A\cap D)$$
at any point $q\in \T$ is explicitly given by
$$d\Psi_{q}:\, \text{T}_{q}^{1,0}\T \to \oplus_{k=1}^{n}\text{Hom}({F}_{q}^k/{F}_{q}^{k+1}, {F}_{q}^{k-1}/{F}_{q}^{k}) \to \mathcal{H}_{A}|_{\Psi(q)}\simeq \text{T}_{\Psi(q)}^{1,0}(A\cap D)$$
which is an isomorphism by the local Torelli theorem for Calabi--Yau manifolds as in (\ref{local assumption''}). Therefore we have proved that the holomorphic map $$\Psi:\, \T\to A\cap D$$ is nondegenerate which induces an
affine structure on $\T$ from $A$.
\end{proof}


Here recall that by pulling back the affine structure on $A\cap D$, we mean to pull back the affine coordinates of $A\simeq \C^N$ by the nondegenerate holomorphic map $\Psi$, and this defines an affine structure on $\T$. We remark that this affine structure on ${\mathcal{T}}$ depends on the choice of the base point $p$. Affine structures on ${\mathcal{T}}$ defined in this ways by fixing different base point may not be compatible with each other.

\subsection{Affine structure on the Hodge metric completion space}\label{affine on TH}

First recall the diagram \eqref{cover maps}:
\begin{align} \xymatrix{\mathcal{T}\ar[r]^{i_{m}}\ar[d]^{\pi_m}&\mathcal{T}^H_{m}\ar[d]^{\pi_{m}^H}\ar[r]^{{\Phi}^{H}_{m}}&D\ar[d]^{\pi_D}\\
\mathcal{Z}_m\ar[r]^{i}&\mathcal{Z}^H_{m}\ar[r]^{{\Phi}_{{\mathcal{Z}_m}}^H}&D/\Gamma.
}
\end{align}

By Corollary \ref{extended-boundedness},
we define the holomorphic map $$\Psi_m^H:\, \TT\to A\cap D$$ by composing
the extended period map $\Phi_m^H:\,\TT\to N_+\cap D$ with the projection map $P:\, N_+\cap D\to A\cap D$.
We also define the holomorphic map $$\Psi_m:\, \T_m\to A\cap D$$ by restricting $\Psi_m=\Psi_m^H|_{\T_m}$, which is given by  $\Psi_m=P\circ \Phi_m$. 

Next, recall that $\T_{m}\subset \TT$ is an open complex submanifold of $\TT$ with $\text{codim}_{\C}(\TT\setminus \T_{m})\ge 1$,  $i_m$ is a covering map onto $\T_m$, and $$\T_m=i_m(\T)=(\pi_m^H)^{-1}(\sZ).$$ We can choose a small neighborhood $U$ of any point in $\T_{m}$ such that
$$\pi_m^H:\, U \to V=\pi_m^H(U)\subset \sZ,$$
is a biholomorphic map. We can shrink $U$ and $V$ simultaneously such that $\pi_{m}^{-1}(V)=\cup_{\alpha}W_{\alpha}$ and $\pi_{m} :\, W_{\alpha}\to V$ is also
biholomorphic. Choose any $W_{\alpha}$ and denote it by $W=W_{\alpha}$. Then $i_{m} :\, W \to U$ is a biholomorphic map. Since $$\Psi=\Psi_{m}^{H}\circ
i_{m}=\Psi_{m}\circ i_{m},$$ we have $\Psi|_{W}=\Psi_{m}|_{U}\circ i_{m}|_{W}$. Theorem \ref{affine structure} implies that $\Psi|_{W}$ is biholomorphic onto its
image, if we shrink $W$, $V$ and $U$ again. Therefore $$\Psi_{m}|_{U} :\, U\to A\cap D$$ is biholomorphic onto its image. By pulling back the affine coordinate
chart in $A \simeq \C^N$, we get an induced affine structure on $\T_m$ such that $\Psi_m$ is an affine map.

In conclusion, we have proved the following lemma.

\begin{lemma}\label{affine on Tm}
The holomorphic map $$\Psi_m:\, \T_m\to A\cap D$$ is a local embedding. In particular, $\Psi_m$ defines a global holomorphic affine structure on $\T_m$.
\end{lemma}

Next we will prove that the affine structure induced by $\Psi_m:\, \T_m\to A\cap D$ can be extended to a global affine structure on $\TT$, which is precisely induced by the extended period map $\Psi_m^H:\, \TT\to A\cap D$.

\begin{definition}
Let $M$ be a complex manifold and $N\subset M$ a cloesd subset. Let $E_0\to M\setminus N$ be a holomorphic vector bundle.
Then $E_0$ is called holomorphically trivial along $N$, if for any point $x\in N$, there exists an open neighborhood $U$ of $x$ in $M$ such that $E_0|_{U\setminus N}$ is holomorphically trivial.
\end{definition}
We need to following elementary lemma, which is Proposition 4.4 from \cite{Mckay}, to proceed.
\begin{lemma}\label{unique-extension}
The holomorphic vector bundle $E_0\to M\setminus N$ can be extended to a unique holomorphic vector bundle $E\to M$ such that $E|_{M\setminus N}=E_0$, if and only if $E_0$ is holomorphically trivial along $N$.
\end{lemma}
\begin{proof}
The proof is taken from Proposition 4.4 of \cite{Mckay}. We include it here for convenience of the reader.

If such an extension $E\to M$ exists uniquely, then we can take a neighborhood $U$ of $x$ such that $E|_{U}$ is trivial, and hence $E_0|_{U\setminus N}$ is trivial.

Conversely, suppose that for any point $x\in N$, there exists an open neighborhood $U_x$ of $x$ in $M$ such that $E_0|_{U_x\setminus N}$ is trivial. We cover $N$ by the open sets $\{U_x:\, x\in N\}$. For any $y\in M\setminus N$, we can choose an open neighborhood $V_y$ of $y$ such that $V_y\cap N=\emptyset$.

We cover $M$ by the open sets in
$$\mathcal{C}=\{U_x:\, x\in N\}\cup \{V_y:\, y\in M\setminus N\},$$ and 
denote $\psi_{V}^{U}$ to be the transition function of $E_0$ on $U\cap V$ if $U\cap V\neq \emptyset$, where $U,V$ are open subsets of $M\setminus N$ and $E_0$ is trivial on $U$ and $V$.

If $V_y\cap V_{y'}\neq \emptyset$, then $\psi_{V_{y'}}^{V_y}$ is well-defined, since $V_y\cap V_{y'}\subset M\setminus N$. We can define the transition function $\phi_{V_{y'}}^{V_y}=\psi_{V_{y'}}^{V_y}$.
For the case of $U_x$ and $V_y$, we use the same argument to get the transition functions $\phi_{V_{y}}^{U_x}=\psi_{V_{y}}^{U_x}$ and
$\phi^{V_{y}}_{U_x}=\psi^{V_{y}}_{U_x}$.

The remaining case is when $U_x\cap U_{x'}\neq \emptyset$, where $x ,x' \in N$. Then the transition function $\psi_{U_{x'}\setminus N}^{U_x\setminus N}$ is well-defined and we define the transition function $\phi_{U_{x'}}^{U_x}=\psi_{U_{x'}\setminus N}^{U_x\setminus N}$. Clearly the newly defined transition functions $\phi_{V}^U$ with $U,V\in \mathcal{C}$ satisfy that
$$\phi_{V}^U \circ \phi_{U}^V=\text{id},\text{ and }\phi_{V}^U \circ \phi_{W}^V \circ \phi_{U}^W=\text{id}.$$
Therefore these transition functions determine a unique holomorphic vector bundle $E\to M$ such that $E|_{M\setminus N}=E_0$.
\end{proof}

On the period domain ${D}$ we have the Hodge bundle which is the horizontal subbundle of the tangent bundle of $D$,
\begin{equation}\label{Hodge-bundle}
\bigoplus_{k=1}^n \text{Hom}(F^k/F^{k+1},F^{k-1}/F^k),
\end{equation}
in which the differentials of the period maps $\Phi$, $\Phi_m^H$ and $\Phi_m$ take values.
From local Torelli for Calabi--Yau manifolds, we know that the image of the differential of the period map at each point  is  the Hodge subbundle $\text{Hom}(F^n,F^{n-1}/F^n)$.  On the other hand, as the variation of Hodge structure from geometry, the Hodge bundles naturally exist on $\T$ and $\T_m$.


 Let us denote the restriction to $A\cap D$ of the Hodge subbundle $\mathcal{H}=\text{Hom}(F^n,F^{n-1}/F^n)$ on $D$ by
$$\mathcal{H}_A=\text{Hom}(F^n,F^{n-1}/F^n)|_{A}.$$
By the definition of $A$, the holomorphic tangent bundle of $A\cap D$ is naturally isomorphic to $\mathcal{H}_A$ as proved in Theorem \ref{affine structure}.
Theorem \ref{affine structure} and Lemma \ref{affine on Tm} give the natural isomorphisms of the holomorphic vector bundles over $\T$ and $\T_m$ respectively
\begin{eqnarray*}
d\Psi:\,\text{T}^{1,0}\T \simeq \Psi^*\mathcal{H}_A,\\
d\Psi_{m}:\,\text{T}^{1,0}\T_m\simeq \Psi_m^*\mathcal{H}_A.
\end{eqnarray*}

We remark that the period map $\Phi_{\Z}^{H}:\, \ZZ\to D/\Gamma$ can be lifted to the universal cover to get  $\Phi_{m}^{H}: \, \TT\to D$ due to the fact that  around any point in $\ZZ\setminus \Z$, the Picard-Lefschetz transformation, or equivalently, the monodromy is trivial, which is proved in Lemma \ref{cidim}. Hence the Hodge bundle $(\Psi_m^H)^*\mathcal{H}_A$ is the natural extension of $\Psi_m^*\mathcal{H}_A$ over $\TT$. More precisely we have the following lemma.

\begin{lemma}\label{Hodge extension}
The isomorphism $\text{T}^{1,0}\T_m\simeq \Psi_m^*\mathcal{H}_A$ of holomorphic vector bundles over $\T_m$ has a unique extension to an isomorphism of holomorphic vector bundles over $\T^H_m$ with
$$\text{T}^{1,0}\TT\simeq (\Psi_m^H)^*\mathcal{H}_A.$$
\end{lemma}
\begin{proof} We give another proof by using the extension of the period map.

For any $p\in \TT\setminus \T_{m}$, we can choose an open neighborhood $U_{p}\subset \TT$ such that $\Psi_{m}^{H}(U_{p})$ is contained in an open neighborhood $W\subset A\cap D$ along which $\mathcal{H}_A$ is trivial. Let
$$\bar{o}=\Psi_{m}^{H}(p)\in W \text{ and } H_{A}=\mathcal{H}_A|_{\bar{o}}.$$ Then we have $$\mathcal{H}_A|_{W}\simeq W\times H_{A}$$ and
$$(\Psi_m^H)^*\mathcal{H}_A|_{U_{p}}\simeq(\Psi_m^H)^*(\mathcal{H}_A|_{W})|_{U_{p}}\simeq(\Psi_m^H)^*(W\times H_{A})|_{U_{p}}\simeq U_{p}\times H_{A}.$$
Let $U=U_{p}\cap \T_{m}$.
Then the restriction to $U$ of the pull-back Hodge bundle is $$\Psi_m^*\mathcal{H}_A|_{U}=((\Psi_m^H)^*\mathcal{H}_A|_{U_{p}})|_{U}\simeq U\times H_{A},$$ which means that $\Psi_m^*\mathcal{H}_A$ is holomorphically trivial along $\TT \setminus \T_{m}$.

Therefore by Lemma \ref{unique-extension}, the Hodge bundle $\Psi_m^*\mathcal{H}_A$ over $\T_m$ has a unique extension over $\TT$, which is
$(\Psi_m^H)^*\mathcal{H}_A$ by continuity. Then we apply Lemma \ref{unique-extension} again to conclude that the holomorphic tangent bundle $\text{T}^{1,0}\T_m$,
which is isomorphic to $\Psi_m^*\mathcal{H}_A$, is holomorphically trivial along $\TT\setminus \T_m$. Hence $\text{T}^{1,0}\T_m$ has a unique extension which is
obviously $\text{T}^{1,0}\TT$. By the uniqueness of such extension we conclude that
$$\text{T}^{1,0}\TT\simeq (\Psi_m^H)^*\mathcal{H}_A.$$
\end{proof}

With the same notation $\mathcal{H}=\text{Hom}(F^n,F^{n-1}/F^n)$ to denote the corresponding Hodge bundles on $\T_m $ and $\TT$,
the above lemma simply tells us that the isomorphism of bundles $\text{T}^{1,0}\T_{m}\simeq \mathcal{H}$ on $\T_{m}$ extends to isomorphism on $\TT$, $\text{T}^{1,0}\TT\simeq \mathcal{H}$ due to the trivial monodromy around   around $\ZZ\setminus \Z$.

With the above preparations, we are ready to prove the following theorem, which is crucial to the proofs of our main theorems.

\begin{thm}\label{THmaffine}
The holomorphic map $$\Psi_m^H:\, \TT\to A\cap D$$ is nondegenerate. Hence $\Psi_m^H$ defines a global affine structure on $\TT$.
\end{thm}
\begin{proof}
Without confusion we will use $\mathcal{H}_A$ to denote the tangent bundle of $A\cap D$, since they are naturally isomorphic.

For any point $q\in \TT\setminus \T_m$, we can choose a neighborhood $U_q$ of $q$ in $\TT$ such that $$\text{T}^{1,0}\TT\simeq (\Psi_m^H)^*\mathcal{H}_A$$ is trivial on $U_q$. Moreover, we can shrink $U_q$ so that $\Psi_m^H(U_q)\subseteq V$, where $V\subset A\cap D$ on which $\mathcal{H}_A$ is trivial.
We choose a basis $$\{\Lambda_1, \cdots ,\Lambda_N \}$$ of $\mathcal{H}_A|_{V}$, which is parallel with respect to the natural affine structure on $A\cap D$.
Let $$\mu_i=((\Psi_m^H)^*\Lambda_i)|_{U_{q}},\, 1\le i\le N.$$ Then $\{\mu_1,\cdots,\mu_n\}$ is a basis of $$\text{T}^{1,0}\TT|_{U_{q}}\simeq (\Psi_m^H)^*\mathcal{H}_A|_{U_{q}}.$$

Let $U=U_q\cap \T_m$ and $q_k\in U$ such that $q_k\longrightarrow q$ as $k\longrightarrow \infty$.
Since $$\Psi_m:\, \T_m\to A\cap D$$ defines a global affine structure on $\T_m$, for any $k\ge1$ we have
$$(d\Psi_m)(\mu_i|_{q_k})=\sum_{j} A_{ij}(q_k)\Lambda_j|_{o_k},$$
for some nonsingular matrix $(A_{ij}(q_k))_{1\le i,j\le N}$, where $o_k=\Psi_m(q_k)$.

Since the affine structure on $\T_m$ is induced from that of $A\cap D$, both the bases
$$\{\mu_1|_{U},\cdots,\mu_N|_{U} \}\text{ and }\{\Lambda_1, \cdots ,\Lambda_N \}$$
are parallel with respect to the affine structures on $\T_m$ and $A\cap D$ respectively.
Hence we have that the matrix  $(A_{ij}(q_k))=(A_{ij})$ is constant matrix.
Therefore
\begin{eqnarray*}
(d\Psi^H_m)(\mu_i|_{q})&=&\lim_{k\to \infty}(d\Psi^H_m)(\mu_i|_{q_k})\\
&=&\lim_{k\to \infty}(d\Psi_m)(\mu_i|_{q_k})\\
&=&\lim_{k\to \infty}\sum_{j} A_{ij}\Lambda_j|_{o_k}\\
&=&\sum_{j} A_{ij}\Lambda_j|_{o_{_\infty}},
\end{eqnarray*}
where $o_\infty =\Psi^H_m(q)$.
Since the matrix $(A_{ij})_{1\le i,j\le N}$ is nonsingular, the tangent map $d\Psi^H_m$ is nondegenerate.
\end{proof}

Note that in the proof of Theorem \ref{THmaffine}, we have used substantially the identification
$$d\Psi_{m}:\, \text{T}^{1,0}\T_m\simeq \Psi_m^*\mathcal{H}_A$$
on $\T_m$,
which is explicitly given by the contraction map $\kappa(v)\lrcorner$ at a point $(q,v) \in \text{T}_q \T_m$. Here recall that $$\kappa:\,T^{1,0}_q\mathcal{T}_{m}\to {H}^{0,1}(M_q,T^{1,0}M_q)$$ is the Kodaira-Spencer map for any $q\in \T_m$. From the explicit expression, one can see that this identification of bundles $$\text{T}^{1,0}\T_m\simeq \Psi_m^*\mathcal{H}_A\simeq \mathcal{H}$$ on $\T_m$ only depends on the tangent vector $v$.

Moreover, since $\Psi_m^*\mathcal{H}_A$ extends trivially to $(\Psi_m^H)^*\mathcal{H}_A$ on $\TT$ due to the trivial monodromy along $\ZZ\setminus \Z$,
it gives the identification of the tangent bundle $\text{T}^{1,0}\TT$ with the Hodge subbundle $$(\Psi_m^H)^*\mathcal{H}_A\simeq \mathcal{H}$$ on $\TT$ by Lemma \ref{Hodge extension}.
From this special feature of the period map, one can see that the identification of the tangent bundle of the Torelli space with the Hodge subbundle is compatible with both the affine structures on $\T_m$ and $A$, as well as the affine map $\Psi_m$.

Now we recall a lemma due to Griffiths and Wolf, which is proved as Corollary 2 in \cite{GW}.
\begin{lemma}\label{covering-lemma}
Let $f:\, X\to Y$ be a local diffeomorphism of connected Riemannian manifolds. Assume that $X$ is complete for the induced metric. Then $f(X)=Y$, $f$ is a covering map and $Y$ is complete.
\end{lemma}

This lemma, together with Theorem \ref{THmaffine} proved above, gives the following corollary.

\begin{corollary}\label{cover-TmH}
The holomorphic map $$\Psi_m^H:\, \TT\to A\cap D$$ is a universal covering map, and the image $\Psi_{m}^{H}(\TT)=A\cap D$ is complete with respect to the Hodge metric.
\end{corollary}




It is important to note that the flat connections which correspond to the global holomorphic affine structures on $\mathcal{T}$, on $\mathcal{T}_m$ or on $\mathcal{T}_{m}^H$ are in general not compatible with the corresponding Hodge metrics on them.

\subsection{Injectivity of the period map on the Hodge metric completion space}\label{injective}
The main purpose of this section is to prove Theorem \ref{injectivityofPhiH} stated below. We will give two proofs of this theorem.
The first proof is to show directly that $A\cap D$ is simply connected, which together with Corollary \ref{cover-TmH} implies the theorem. 

The second proof uses the affine structures on $\T^H_m$ and $A\cap D$ in a more substantial way. Each proof reflects different geometric structures of the period domain and the period map which will be useful for further study, so we include both proofs in this section.


\begin{theorem}\label{injectivityofPhiH}For any $m\geq 3$, the holomorphic map $$\Psi^H_{m}:\, \mathcal{T}^H_{m}\to A\cap D$$ is an injection and hence a biholomorphic map.
\end{theorem}
\begin{proof} As explicitly described in the proof of Lemma \ref{lemma of locallybounded}, the natural projection $\pi:\,D\to G_\mathbb{R}/K$, when restricted
to the underlying real manifold of $\mathrm{exp}(\mathfrak{p}_+)
\cap D$, is given by the diffeomorphism
\begin{equation}\label{pi+}
\pi_{+}:\, \mathrm{exp}(\mathfrak{p}_+) \cap D\longrightarrow \mathrm{exp}(\mathfrak{p}_0) \stackrel{\simeq}{\longrightarrow} G_\mathbb{R}/K,
\end{equation}
which is defined by mapping $\exp (Y) \bar{o}$ to $\exp (X) \bar{o}$ for $Y\in \mathfrak{p}_+$ and $X\in \mathfrak{p}_0$ with the relation that
$$X =T_0( Y +\tau_0(Y))$$ for some real number $T_{0}$.
Here $\bar{o}$ is the base point in $\mathrm{exp}(\mathfrak{p}_+) \cap D$.

By Griffiths transversality, one has
$$\mathfrak{a}\subset\mathfrak{g}^{-1,1}\subset \mathfrak{p}_{+}\text{ and }\mathfrak{a}_0=\mathfrak{a} +\tau_0(\mathfrak{a})\subset \mathfrak{p}_0.$$
Then $A\cap D$ is a submanifold of $\text{exp}(\mathfrak{p}_+) \cap D$, and the diffeomorphism $\pi_{+}$ maps $A\cap D\subseteq \text{exp}(\mathfrak{p}_{+}) \cap D$ diffeomorphically to its image $\exp(\mathfrak{a}_0)$ inside $G_{\mathbb R}/K$, from which one has the diffeomorphism $$A\cap D\simeq \text{exp} (\mathfrak{a}_0)$$ induced by $\pi_+$. Since $\text{exp} (\mathfrak{a}_0)$ is simply connected, one concludes that $A \cap D$ is also simply connected.



Now since $\T^H_m$ is simply connected and $\Psi_m^H:\, \TT\to A\cap D$ is a covering map, we conclude that $\Psi^H_m$ must be a biholomorphic map.
\end{proof}

For the second proof, we will first prove the following elementary lemma, in which we mainly use the completeness with the Hodge metric on $\T^H_m$, the boundedness of the image of the extended period map $$\PP:\, \TT \to N_{+}\cap D$$ which is also an isometry in the Hodge metric,  the holomorphic affine structure on $\mathcal{T}^H_{m}$, and the affineness of $\Psi^H_{m}$. We remark that as $\mathcal{T}^H_{m}$ is a complex affine manifold, we have the notion of straight lines in it with respect to the affine structure.

\begin{lemma}\label{straightline} For any point in $\mathcal{T}^H_{m},$ there is a straight line segment in $\mathcal{T}^H_{m}$ connecting it to the base point $p$.
\end{lemma}
\begin{proof}The proof uses crucially the facts that the map $\Psi_m^H:\, \T_m^H\to A\cap D$ is local isometry with the Hodge metrics on $\T_m^H$ and $A\cap D$, and that it is also
an affine map with the induced affine structure on $\T_m^H$.

Let $q$ be any point in $\T_m^H$. As $\T_m^H$ is connected and complete with the Hodge metric, by Hopf-Rinow theorem, there exists a geodesic $\gamma$ in $\T_m^H$ connecting the base point $p$ to $q$. Since
$\Psi^{H}$ is a local isometry with the Hodge metric, $\tilde{\gamma}=\Psi^{H}(\gamma)$ is also a geodesic in $A\cap D$.

By Griffiths transversality, we have that $\mathfrak{a}\subset\mathfrak{g}^{-1,1}\subset \mathfrak{p}_{+}$ and $\mathfrak{a}_0=\mathfrak{a} +\tau_0(\mathfrak{a})\subset \mathfrak{p}_0$.
On the other hand,  from equation \eqref{pi+}, we know that any geodesic starting from the base point $\tilde{p}=\Psi^{H}(p)$
in $A\cap D$ is of the form $\exp(tX)\tilde{p}$ with $X\in \mathfrak{a}_{0}$ and $t\in \mathbb{R}$.  Also from the explicit computation in Lemma \ref{abounded} we have the relation
$$\exp(tX)\tilde{p}=\exp(T(t)Y)\tilde{p}$$ with $Y\in \mathfrak{a}$ satisfying $X=T_0(Y+\tau_{0}(Y))$ for some  $T_0\in \mathbb{R}$, and $T(t)$ a smooth real valued monotone function of $t$.

 Note that the affine structure on $A\cap D$ is induced from the affine structure on $\mathfrak{a}$ by the exponential map $$\text{exp}:\, \mathfrak{a}\to A,$$
therefore $\tilde{\gamma}=\exp(T(t)Y)\tilde{p}$ corresponds to a straight line with respect to the affine structure on $A\cap D$. Hence $\gamma$ is also a
straight line in $\T^H$ with respect to the induced affine structure, since $\Psi^H$ is an affine map.
\end{proof}

We remark that the above proof uses substantially the fact that the straight line segment connects the base point to any other point in $\TT$.

\begin{proof}[Second Proof of Theorem \ref{injectivityofPhiH}] We will prove by contradiction.
Let $q_{1}, q_{2}\in \mathcal{T}^H_{m}$ be two different points.
Suppose that $\Psi_{m}^{H}(q_{1})=\Psi_{m}^{H}(q_{2})$.

If one of $q_{1}, q_{2}$, say $q_{1}$, happens to be the base point $p$,
then Lemma \ref{straightline} implies that there is a straight line segment $l\subseteq \mathcal{T}^H_{m}$ connecting $p$ and $q_{2}$.
Since $\Psi_{m}^{H}$ is an affine map and is locally biholomorphic by local Torelli for Calabi--Yau manifolds, the straight line segment $l$ is mapped to a straight line segment $\tilde{l}=\Psi_{m}^{H}(l)$ in $$A\cap D \subset A\simeq \C^{N}.$$ But $\Psi_{m}^{H}(p)=\Psi^{H}(q_{2})$, $\tilde{l}$ is also a cycle in $A\simeq \C^{N}$, which is a contradiction.

Now if both $q_{1}$ and $q_{2}$ are different from the base point $p$, then by Lemma \ref{straightline}, there exist two different straight line segments $l_{1},l_{2}\subseteq \mathcal{T}^H_{m}$ connecting $p$ to $q_{1}$ and $q_{2}$ respectively.
Since $$\Psi_{m}^{H}(q_{1})=\Psi_{m}^{H}(q_{2}),$$ the two straight line segments $\tilde{l_{i}}=\Psi_{m}^{H}(l_{i})$, $i=1,2$ in $A\cap D$ must coincide, since they have the same end points.  This contradicts to the fact that $\Psi_{m}^{H}$ is locally biholomorphic at the base point $p$.
\end{proof}

Since $\Psi^H_m = P\circ \P^H_m$, we also have the following corollary.
\begin{corollary}\label{injective PhiHm}The extended period map $\Phi^H_{m}: \,\mathcal{T}^H_{m}\rightarrow N_+\cap D$ is  an injection.
\end{corollary}

To conclude this section we give another description of the affine structure induced by the map $\Psi :\,\TT \to A\cap D$.
With the adapted basis at the base point $p\in T$, we consider $$(d\Phi)_{p}(\text{T}^{1,0}_{p}\T)=\mathfrak{a}\subset \mathfrak{n}_{+}$$ as a block lower triangle matrix whose diagonal elements are zero. Moreover by the local Torelli theorem for Calabi--Yau manifolds, the differential map
$$(d\Phi)_{p}:\, \text{T}^{1,0}_{p}\T \to \text{Hom}(F_{p}^{n},F_{p}^{n-1}/F_{p}^{n})$$
is an isomorphism.
Hence we get  that $\mathfrak{a}$ is isomorphic to its $(1,0)$-block as vector spaces, see \eqref{block} for the definition. 

Let $(\tau_{1},\cdots ,\tau_{N})^{T}$ be the $(1,0)$-block of $\mathfrak{a}$.
Since the affine structure on $A$ is induced by $\exp:\, \mathfrak{a}\to A$
which is an isomorphism, $(\tau_{1},\cdots ,\tau_{N})^{T}$ also defines a global affine structure on A, and hence a global affine structure on $\T_{m}^{H}$, and we denote it by
\begin{equation}\label{tau affine}
\tau_{m}^{H}:\, \TT \to \C^{N},\quad  q\mapsto (\tau_{1}(q),\cdots ,\tau_{N}(q)).
\end{equation}
Note that from linear algebra, it is easy to see that the $(1,0)$-block of $A=\exp(\mathfrak{a})$ is still $(\tau_{1},\cdots ,\tau_{N})^{T}$. Hence the affine map \eqref{tau affine} can be constructed as the $(1,0)$-block of the image of the period map, which are realized as matrices in $N_{+}$.
To be precise, let $$P^{1,0}:\, N_{+}\to \C^{N}$$  be the projection of the matrices in $N_{+}$ onto their $(1,0)$-blocks. Then the affine map \eqref{tau affine} is
$$\tau_{m}^{H}=P^{1,0}\circ \PP : \, \TT \to \C^{N}\simeq H^{n-1,1}(M_{p}).$$
From Theorem \ref{injectivityofPhiH}, we know that $\tau_{m}^{H}$ is injective, and hence it also defines an embedding of $\TT$ into $\C^{N}$.


\section{Global Torelli and Domain of holomorphy}\label{completionspace}
In this section, we first show that $\mathcal{T}^H_{m}$ does not depend on the choice of the level $m$, so that we denote the Hodge metric completion space $\mathcal{T}^H$ by $\mathcal{T}^H=\mathcal{T}^H_{m}$, and the extended period map $\Phi^H$ by $\Phi^H=\Phi^H_{m}$ for any $m\geq 3$.  Therefore $\mathcal{T}^H$ is a complex affine manifold and that $\Phi^H$ is a holomorphic injection. We then prove Theorem \ref{main theorem}, which asserts that $\mathcal{T}^H$ is the completion space of the Torelli space $\mathcal{T}'$ with respect to the Hodge metric and it is a bounded domain of holomorphy in $\mathbb{C}^N$. 

As a direct corollary, we derive the global Torelli theorem for the injectivity of the period maps from the Torelli space $\T'$ and from its completion space $\T^H$ to the period domain.
Another corollary we will prove is the global Torelli theorem on the moduli space of polarized Calabi--Yau manifolds with level $m$ structure for any $m\geq 3$.

For any two integers $m, m'\geq 3$, let $\mathcal{Z}_m$ and $\mathcal{Z}_{m'}$ be the smooth quasi-projective manifolds as in Theorem \ref{Szendroi theorem 2.2} and let $\mathcal{Z}^H_{m}$ and $\mathcal{Z}^H_{m'}$ be their completions with respect to the Hodge metric. Let $\mathcal{T}^H_{m}$ and $\mathcal{T}^H_{m'}$ be the universal cover spaces of $\mathcal{Z}^H_{m}$ and $\mathcal{Z}_{{m'}}^H$ respectively.
From Theorem \ref{injectivityofPhiH}, we know that both $\mathcal{T}^H_{m}$ and $\mathcal{T}^H_{m'}$ are biholomorphic to $A\cap D$. Hence we have the following proposition.
\begin{proposition}\label{indepofm}
For any $m\geq 3$, the complex manifold $T^H_m$ is complete equipped with the Hodge metric,  and is biholomorphic to $A\cap D$. So for any integers $m, \, m'\geq 3$,  the complex manifolds $\mathcal{T}^H_m$ and $\mathcal{T}^H_{m'}$ are biholomorphic to each other.
\end{proposition}
Proposition \ref{indepofm} shows that $\mathcal{T}^H_{m}$ is independent of the choice of the level $m$ structure, and it allows us to introduce the following notations.

We define the complex manifold $\mathcal{T}^H=\mathcal{T}^H_{m}$, the holomorphic map $$i_{\mathcal{T}}: \,\mathcal{T}\to \mathcal{T}^H$$ by $i_{\mathcal{T}}=i_m$, and the extended period map $\Phi^H:\, \mathcal{T}^H\rightarrow D$ by $\Phi^H=\Phi^H_{m}$ for any $m\geq 3$.
In particular, with these new notations, we have the commutative diagram
\begin{align}\label{main diagram}
\xymatrix{\mathcal{T}\ar[r]^{i_{\mathcal{T}}}\ar[d]^{\pi_m}&\mathcal{T}^H\ar[d]^{\pi^H_{m}}\ar[r]^{{\Phi}^{H}}&D\ar[d]^{\pi_D}\\
\mathcal{Z}_m\ar[r]^{i}&\mathcal{Z}^H_{m}\ar[r]^{{\Phi}_{{\mathcal{Z}_m}}^H}&D/\Gamma.
}
\end{align}

The main result of this section is the following,
\begin{theorem}\label{main theorem}
The complex manifold $\mathcal{T}^H$ is a complex affine manifold which can be embedded into $A\simeq \mathbb{C}^N$ and it is the completion space of the Torelli space $\mathcal{T}'$ with respect to the Hodge metric. Moreover, the extended period map $$\Phi^H: \,\mathcal{T}^H\rightarrow N_+\cap D$$ is a holomorphic injection.
\end{theorem}
\begin{proof}By the definition of $\mathcal{T}^H$ and Theorem \ref{injectivityofPhiH}, it is easy to see that $\mathcal{T}^H$ is a complex affine manifold, which can be embedded into $A\simeq \mathbb{C}^N$.
It is also not hard to see that the injectivity of $\Phi^H$ follows directly from Corollary \ref{injective PhiHm} by the definition of $\Phi^H$. Now we only need to show that $\mathcal{T}^H$ is the Hodge metric completion space of $\mathcal{T}'$, which follows from the following lemma. \end{proof}
\begin{lemma}\label{injectivity of i}
Let $\T_{0}\subset \T^{H}$ be defined by $\T_{0}:=i_{\T}(\T)$. Then $\T_{0}$ is biholomorphic to the Torelli space $\T'$.
\end{lemma}
\begin{proof}  Level structures will play a crucial role in the proof. First note that diagram \eqref{main diagram} together with diagram \eqref{periods} in Section \ref{section period map} implies the following commutative diagram
$$\xymatrix{
\T \ar[dr]^-{\pi}\ar[rr]^-{i_{\T}} \ar[dd]^-{\pi_m} &&\T_{0}  \ar[dd]^-{\pi_{m}^{H}|_{\T_{0}}}\ar[rr]^-{\P^{H}|_{\T_{0}}}  &&D\ar[dd]^-{\pi_{D}}\\
&\T'\ar[dl]_-{\pi'_{m}}\ar[urrr]^-{\P'}&&&\\
\Z \ar[rr]^-{i} &&\ZZ \ar[rr]^-{\P_{\Z}^{H}} &&D/\Gamma .}$$

Now we construct a map $\pi_0:\, \T_{0}\to \T'$. For any point $o$ in $\T_{0}$, we can choose $p\in \T$ such that $i_{\T}(p)=o$. We define $\pi_0:\, \T_{0}\to \T'$ such that $\pi_0(o) = \pi (p)$. Then clearly $\pi_0$ fits in the following commutative diagram
$$\xymatrix{
\T \ar[dr]^-{\pi}\ar[rr]^-{i_{\T}} \ar[dd]^-{\pi_m} &&\T_{0} \ar[dl]_-{\pi^{0}} \ar[dd]^-{\pi_{m}^{H}|_{\T_{0}}}\ar[rr]^-{\P^{H}|_{\T_{0}}}  &&D\ar[dd]^-{\pi_{D}}\\
&\T'\ar[dl]_-{\pi'_{m}}\ar[urrr]^-{\P'}&&&\\
\Z \ar[rr]^-{i} &&\ZZ \ar[rr]^-{\P_{\Z}^{H}} &&D/\Gamma .}$$

We first show that the map $\pi_0:\, \T_0\to \T'$ defined this way is well-defined and satisfies $\P^{H}|_{\T_{0}}=\P'\circ \pi_{0}$.

Let $p\neq q\in \T$ be two points such that
$i_{\T}(p)=i_{\T}(q)\in \T_{0}$.
We choose some $m\ge 3$ such that $\T^{H}\simeq \TT$. Then $$i_m(p)=i_m(q)\in \T_{m} \subset \TT.$$
Let $[M_p, L_p, \gamma_p]$ and $[M_q, L_q, \gamma_q]$ denote the fibers over the points $p$ and $q$ of the analytic family $\phi':\, \U'\to \T'$ respectively, where $\gamma_p$ and $\gamma_q$ are two markings identifying the fixed lattice $\Lambda$ isometrically with $H^{n}(M_{p},\mathbb{Z})/\text{Tor}$ and $H^{n}(M_{q},\mathbb{Z})/\text{Tor}$ respectively.

Since $i_m(p)=i_m(q)$ and $i\circ \pi_m=\pi^H_{m}\circ i_m$, we have $$i\circ\pi_m(p)=i\circ\pi_m(q)\in \Z,$$ i.e. $\pi_m(p)=\pi_m(q)\in \Z$. By the definition of $\Z$, there exists a biholomorphic map $f:\, M_p\to M_q$ such that $f^*L_q=L_p$ and
$$f^*\gamma_q= \gamma_p\cdot A,$$
where $A\in \text{Aut}(H^{n}(M_{p},\mathbb{Z})/\text{Tor},Q))$ satisfies
$$A=(A_{ij})\equiv\text{Id}\quad(\text{mod } m), \text{ for }m\ge 3.$$

Let $m_0$ be an integer such that $m_0> |A_{ij}|$ for any $i,j$. Proposition \ref{indepofm} implies that $\T^{H}\simeq \TT\simeq \T_{m_0}^{H}$. The same argument as above implies that $\pi_{m_0}(p)=\pi_{m_0}(q)\in \mathcal{Z}_{m_{0}}$ and hence
$$A=(A_{ij})\equiv\text{Id}\quad(\text{mod } m_0).$$
Since each $m_0> |A_{ij}|$, we have $A=\text{Id}$.

Therefore, we have found a biholomorphic map $f:\, M_{p}\to M_q$ such that $f^*L_q=L_p$ and $f^*\gamma_q= \gamma_p$. This implies that $p$ and $q$ in $\T$ actually correspond to the same point in the Torelli space $\T'$, i.e. $\pi(p) =\pi(q)$ in $\T'$. Therefore we have proved that the map $$\pi_0:\, \T_0\to \T'$$ is well-defined.
From the definitions of the period map and that of the Torelli space, we deduce that the map $\pi_0:\, \T_0\to \T'$ satisfies that $$\P^{H}|_{\T_{0}}=\P'\circ \pi_{0}.$$

Clearly $\pi_0$ is surjective, because $\pi:\, \T \to \T'$ is surjective. Since $\P^{H}:\, \T^{H}\to D$ is injective, so is the restriction map $\P^{H}|_{\T_{0}}:\, \T_{0}\to D$, which implies the injectivity of the map $\pi_{0}:\, \T_{0}\to \T'$. Hence $\pi_{0}:\, \T_{0}\to \T'$ is in fact a biholomorphic map. This completes the proof of  the lemma.
\end{proof}

We remark that there is another approach to Lemma \ref{injectivity of i}, by using the level structures and fundamental groups.
Since $$\mathcal{T}_{0}\simeq \mathcal{T}_m=(\pi^H_{m})^{-1}(\mathcal{Z}_m)$$ for any $m\ge 3$, $\pi_m^H:\, \mathcal{T}_0\to \mathcal{Z}_m$ is a covering map. Thus the fundamental group of $\mathcal{T}_0$ is a subgroup of the fundamental group of $\mathcal{Z}_m$, that is, $ \pi_1(\mathcal{T}_0)\subseteq \pi_1(\mathcal{Z}_m)$ for any $m\geq 3$.
Moreover, the universal property of the universal covering map $\pi_m: \,\mathcal{T}\to \mathcal{Z}_m$ with the identity $$i_m\circ \pi_m=\pi_{m}^H|_{\mathcal{T}_0}\circ i_{\mathcal{T}}$$
implies that $i_{\mathcal{T}}:\,\mathcal{T}\to \mathcal{T}_0$ is also a covering map, as proved in Proposition \ref{opend}.

On the other hand, let $\{m_k\}_{k=1}^{\infty}$ be a sequence of positive integers such that $m_k< m_{k+1}$ and $m_k|m_{k+1}$ for each $k\geq 1$. From the discussion of Lecture 10 of \cite{Popp}, or Page 5 of \cite{sz}, there is a natural covering map from $\mathcal{Z}_{m_{k+1}}$ to  $\mathcal{Z}_{m_{k}}$ for each $k$.
Thus the fundamental group $\pi_1(\mathcal{Z}_{m_{k+1}})$ is a subgroup of $\pi_1(\mathcal{Z}_{m_{k}})$ for each $k$. Hence, the inverse system of fundamental groups
\begin{align*}
\xymatrix{
\pi_1(\mathcal{Z}_{m_1})&\pi_1(\mathcal{Z}_{m_2})\ar[l]&\cdots\cdots\ar[l]&\pi_1(\mathcal{Z}_{m_k})\ar[l]&\cdots\ar[l]
}
\end{align*}
has an inverse limit, which is the fundamental group of the Torelli space $\mathcal{T}'$. Since $\pi_1(\mathcal{T}_0)\subseteq \pi_1(\mathcal{Z}_{m_{k}})$ for any $k$, we have the inclusion $$\pi_1(\mathcal{T}_0)\subseteq\pi_1(\mathcal{T}').$$ This implies that $\T_0$ is a covering of $\T'$. Let $\pi_{0}:\, \T_{0}\to \T'$ be the covering map. The injectivity of $\pi_{0}$ follows as in the last step of the proof of Lemma \ref{injectivity of i}, from which we also get that $\pi_{0}:\, \T_{0}\to \T'$ is biholomorphic.

Define the injective map $$\pi^{0}:\, \T' \to \T^{H} \text{ with }\pi^{0}=(\pi_{0})^{-1},$$ we then have the relation that $$\P'=\P^{H}\circ \pi^{0}:\, \T' \to D.$$
From the injectivity of $\P^{H}$ and $\pi^{0}$, we deduce the global Torelli theorem on the Torelli space of Calabi--Yau manifolds as follows.

\begin{corollary}[Global Torelli theorem]\label{Global Torelli theorem}
The period map $$\Phi':\, \mathcal{T}'\rightarrow D$$ is injective.
\end{corollary}

As a corollary, we have the global Torelli theorem on the moduli space $\Z$ of polarized Calabi--Yau manifolds with level $m$ structure for any $m\geq 3$.

\begin{corollary}\label{Global Torelli theorem'}
The period map $$\Phi_{\Z}:\, \Z \rightarrow D/\Gamma$$ is injective.
\end{corollary}
\begin{proof}
Let $p_{1}$ and $p_{2}$ be two points in $\Z$ such that $\Phi_{\Z}(p_{1})=\Phi_{\Z}(p_{2})$ in $D/\Gamma$. Let $$[M_{i},L_{i},[\gamma_{i}]_{m}], \, i=1,2$$ be the fibers over $p_{1}$ and $p_{2}$ of the analytic family $f_{m} :\, \U_{m}\to \Z$ respectively.  From Lemma \ref{injectivity of i} we know that $$\pi^0:\, \T'\to \T^{H}$$ identifies $\T'$ to the Zariski open submanifold $$\T_0=i_\T(\T) \simeq i_m(\T)\subseteq \TT$$ of $\TT$, and that $\T_0$ is a cover of $\Z$ by Proposition \ref{opend}. In the following discussion we will use freely the identification $$\T_0\simeq \T'.$$

There exist two points ${q}_{1}$ and ${q}_{2}$ in $\T'$ over which are the fibers $$[M_{i},L_{i},\gamma_{i}], \, i=1,2$$ of the analytic family $\phi':\, \U'\to \T'$ respectively,
such that
$$\pi_{m}'({q}_{i})=p_{i}, \, i=1,2$$
under the covering map $\pi_{m}':\, \T' \to \Z$.

The condition $\Phi_{\Z}(p_{1})=\Phi_{\Z}(p_{2})$ implies that there exists $\gamma\in \Gamma$ such that
$$\Phi'({q}_{1})=\gamma\Phi'({q}_{2}).$$
Let ${q}_{1}'\in \T'$ correspond to the triple $[M_{1},L_{1},\gamma_{1}\gamma]$. Then by the definition of the period map $\Phi'$ in \eqref{defn of P'}, we have
\begin{eqnarray*}
\Phi'(q_{1}')&=&(\gamma_{1}\gamma)^{-1}(F^n(M_{1})\subseteq\cdots\subseteq F^0(M_{1}))\\
             &=&\gamma^{-1}\gamma_{1}^{-1}(F^n(M_{1})\subseteq\cdots\subseteq F^0(M_{1}))\\
             &=&\gamma^{-1}\Phi'(q_{1})\\
             &=&\Phi'(q_{2}).
\end{eqnarray*}
By Corollary \ref{Global Torelli theorem}, we see that $q_{1}'=q_{2}$ in $\T'$.

Now we have the inclusion $\pi^{0}:\, \T' \to \T^{H}$ and the inclusion $i:\, \Z \to \ZZ$, therefore we can view the points $p_{1}$ and $p_{2}$ as points in $\ZZ$ and the points $q_{1}$, $q'_{1}$ and $q_{2}$ as points in $\T^{H}$.

Since $\gamma$ lies in the image of the monodromy representation $\rho:\, \pi_{1}(\ZZ)\to \Gamma$, there exists a $\tilde{\gamma}\in \pi_{1}(\ZZ)$ such that
$$\rho(\tilde{\gamma})=\gamma \text{ and } q_{1}'=q_{1}\tilde{\gamma},$$ where $q_{1}\tilde{\gamma}$ is defined by the action of $\pi_{1}(\ZZ)$ on the universal cover $\T^{H}$, and the action of $\tilde{\gamma}$ on the fiber $[M_{1},L_{1},\gamma_{1}]$ over $q_{1}$ is defined by
$$[M_{1},L_{1},\gamma_{1}]\tilde{\gamma}= [M_{1},L_{1},\gamma_{1}\gamma].$$
So we have  $$p_{1}=\pi_{m}^{H}(q_{1})=\pi_{m}^{H}(q_{1}')=\pi_{m}^{H}(q_{2})=p_{2},$$
which proves the injectivity of $\Phi_{\Z}$.
\end{proof}

As another important consequence, we prove the following property of $\mathcal{T}^H$.
\begin{theorem}\label{important result}
The completion space $\mathcal{T}^H$ is a bounded domain of holomorphy in $\mathbb{C}^N$; thus there exists a complete K\"ahler--Einstein metric on $\mathcal{T}^H$.
\end{theorem}
We recall that a $\mathcal{C}^2$ function $\rho:\, \Omega\rightarrow \mathbb{R}$ on a domain $\Omega\subseteq\mathbb{C}^n$ is \emph{plurisubharmornic} if and only if its Levi form is positive definite at each point in $\Omega$. Given a domain $\Omega\subseteq \mathbb{C}^n$, a function $$f:\, \Omega\rightarrow \mathbb{R}$$ is called an \emph{exhaustion function} if for any $c\in \mathbb{R}$, the set
$$\{z\in \Omega\,|\, f(z)< c\}$$ is relatively compact in $\Omega$.
The following well-known theorem provides a definition for domains of holomorphy. For example, one may refer to \cite{Hom} for details.
\begin{proposition}
An open set $\Omega\subset \mathbb{C}^n$ is a domain of holomorphy if and only if there exists a continuous plurisubharmonic function $f:\, \Omega\to \mathbb{R}$ such that $f$ is also an exhaustion function.
\end{proposition}

The following theorem in \cite[Section 3.1]{GS} gives us the basic ingredients to construct a plurisubharmoic exhaustion function on $\mathcal{T}^H$.
\begin{theorem}\label{general f}
On every manifold $D$, which is dual to a K\"ahler C-space, there exists an exhaustion function $f: \, D\rightarrow \mathbb{R}$, whose Levi form, restricted to $T^{1, 0}_h(D)$, is positive definite at every point of $D$.
\end{theorem}
We remark that in this proposition, in order to show $f$ is an exhaustion function on $D$, Griffiths and Schmid showed that the set $f^{-1}(-\infty, c]$ is compact in $D$ for any $c\in \mathbb{R}$.

\begin{proof}[Proof of Theorem \ref{important result}]
By Theorem \ref{main theorem}, we see that $\mathcal{T}^H$ is a bounded domain in $\mathbb{C}^N$. Therefore, once we show $\mathcal{T}^H$ is domain of holomorphy, the existence of K\"ahler-Einstein metric on it follows directly from the well-known theorem by Mok--Yau in \cite{MokYau}.

In order to show that $\mathcal{T}^H$ is a domain of holomorphy in $\mathbb{C}^N$, it is enough to construct a plurisubharmonic exhaustion function on $\mathcal{T}^H$.

Let $f$ be the exhaustion function on $D$ constructed in Theorem \ref{general f}, whose Levi form, when restricted to the horizontal tangent bundle $T^{1,0}_hD$ of $D$, is positive definite at each point of $D$. We claim that the composition function $f\circ \Phi^H$ is the demanded plurisubharmonic exhaustion function on $\mathcal{T}^H$.

By the Griffiths transversality of $\Phi^H$, the composition function $$f\circ \Phi^H:\,\mathcal{T}^H\to \mathbb{R}$$ is a plurisubharmonic function on $\mathcal{T}^H$. Thus it suffices to show that the function $f\circ \Phi^H$  is an exhaustion function on $\mathcal{T}^H$, which is enough to show that for any constant $c\in\mathbb{R}$,
$$(f\circ\Phi^H)^{-1}(-\infty, c]=(\Phi^H)^{-1}\left(f^{-1}(-\infty, c]\right)$$ is a compact set in $\mathcal{T}^H$.
Indeed, it has already been shown in \cite{GS} that the set $f^{-1}(-\infty, c]$ is a compact set in $D$. Now for any sequence $$\{p_k\}_{k=1}^\infty\subseteq (f\circ\Phi^H)^{-1}(-\infty, c],$$ we have $$\{\Phi^H(p_k)\}_{k=1}^\infty\subseteq f^{-1}(-\infty,  c].$$ Since $f^{-1}(-\infty, c]$ is compact in $D$, the sequence $\{\Phi^H(p_k)\}_{k=1}^\infty$ has a convergent subsequence. We denote this convergent subsequence by $$\{\Phi^H(p_{k_n})\}_{n=1}^\infty\subseteq \{\Phi^H(p_k)\}_{k=1}^\infty$$ with $k_n<k_{n+1}$,
and denote $$\lim_{k\to \infty}\Phi^H(p_k)=o_\infty\in D.$$

On the other hand, since the map $\Phi^H$ is injective and the Hodge metric on $\mathcal{T}^H$ is induced from the Hodge metric on $D$, the extended period map $\Phi^H$ is actually a global isometry onto its image. Therefore the sequence $\{p_{k_n}\}_{n=1}^\infty$ is also a Cauchy sequence that converges to $(\Phi^H)^{-1}(o_{\infty})$ with respect to the Hodge metric in $$(f\circ\Phi^H)^{-1}(-\infty, c]\subseteq \mathcal{T}^H.$$ In this way, we have proved that any sequence in $(f\circ\Phi^H)^{-1}(-\infty, c]$ has a convergent subsequence. Therefore $(f\circ\Phi^H)^{-1}(-\infty, c]$ is compact in $\mathcal{T}^H$, as was needed to show.
\end{proof}
\appendix
\section{Two topological lemmas}\label{topological lemmas}In this appendix we first prove the existence of the choice of $i_{m}$ and $\Phi^H_{m}$ in diagram \eqref{cover maps} such that $\Phi=\Phi^H_m\circ i_{m}$. Then we show a lemma that relates the fundamental group of the moduli space of Calabi--Yau manifolds to that of completion space with respect to the Hodge metric on $\mathcal{Z}_m$. The arguments only use elementary topology and the results may be well-known. We include their proofs here for the sake of completeness.

\begin{lemma}\label{choice}
There exists a suitable choice of $\ i_{m}$ and $\Phi_{m}^H$ such that $\Phi_{m}^H\circ i_{m}=\Phi.$
\end{lemma}
\begin{proof} Recall the following commutative diagram:
\begin{align*} \xymatrix{\mathcal{T}\ar[r]^{i_{m}}\ar[d]^{\pi_m}&\mathcal{T}^H_{m}\ar[d]^{\pi_{m}^H}\ar[r]^{{\Phi}^{H}_{m}}&D\ar[d]^{\pi_D}\\
\mathcal{Z}_m\ar[r]^{i}&\mathcal{Z}^H_{m}\ar[r]^{{\Phi}_{{\mathcal{Z}_m}}^H}&D/\Gamma.
}
\end{align*}
Fix a reference point $p\in \mathcal{T}$. The relations $ i\circ\pi_m=\pi^H_{m}\circ  i_{m}$ and $$\Phi_{\mathcal{Z}_m}^H\circ \pi^H_{m} = \pi_D\circ\Phi^H_{m}$$
imply that $$\pi_D\circ\Phi^H_{m}\circ i_{m} = \Phi_{\mathcal{Z}_m}^H\circ i\circ\pi_m= \Phi_{\mathcal{Z}_m}\circ\pi_m.$$ Therefore $\Phi^H_{m}\circ i_{m}$ is a
lifting map of $\Phi_{\mathcal{Z}_m}$. On the other hand $\Phi: \, \mathcal{T}\to D$ is also a lifting of $\Phi_{\mathcal{Z}_m}$. In order to make
$$\Phi^H_{m}\circ i_{m}=\Phi,$$ one only needs to choose the suitable $ i_{m}$ and $\Phi^H_{m}$ such that these two maps agree on the reference point, i.e.
$\Phi^H_{m}\circ i_{m}(p)=\Phi(p).$

For an arbitrary choice of $ i_{m}$, we have $ i_{m}(p)\in {\mathcal{T}}_m^H$ and $$\pi^H_{m}( i_{m}(p))=i(\pi_m(p)).$$ Considering the point $ i_{m}(p)$ as a reference point in ${\mathcal{T}}_m^H$, we can choose $\Phi^H_{m}( i_{m}(p))$ to be any point from
$$\pi_D^{-1}(\Phi^H_{\mathcal{Z}_m}( i(\pi_m(p)))) = \pi_D^{-1}(\Phi_{\mathcal{Z}_m}(\pi_m(p))).$$
Moreover the relation $\pi_D(\Phi(p))=\Phi_{\mathcal{Z}_m}(\pi_m(p))$ implies that $$\Phi(p)\in \pi_D^{-1}(\Phi_{\mathcal{Z}_m}(\pi_m(p))).$$ Therefore we can choose $\Phi^H_{m}$ such that $\Phi^H_{m}( i_{m}(p))=\Phi(p).$ With this choice, we have $\Phi^H_{m}\circ i_{m}=\Phi$.
\end{proof}

\begin{lemma}\label{fund}
Let $\pi_1(\mathcal{Z}_m)$ and $\pi_1(\mathcal{Z}_m^H)$ be the fundamental groups of $\mathcal{Z}_m$ and $\mathcal{Z}_m^H$ respectively, and suppose the group morphism $$i_*: \, \pi_1(\mathcal{Z}_m) \to \pi_1(\mathcal{Z}_m^H)$$ is induced by the inclusion $i: \, \mathcal{Z}_m\to \mathcal{Z}_m^H$. Then $i_*$ is surjective.
\end{lemma}
\begin{proof}First notice that $\mathcal{Z}_m$ and $\mathcal{Z}_m^H$ are both smooth manifolds, and $\mathcal{Z}_m\subseteq \mathcal{Z}_m^H$ is open. Thus for each point $q\in \mathcal{Z}_m^H\setminus \mathcal{Z}_m$ there is a disc $D_q\subseteq \mathcal{Z}_m^H$ with $q\in D_q$. Then the union of these discs
$$\bigcup_{q\in \mathcal{Z}_m^H\setminus \mathcal{Z}_m}D_q$$
forms a manifold with open cover $\{D_q:~q\in \cup_q D_q\}$. Because both $\mathcal{Z}_m$ and $\mathcal{Z}_m^H$ are second countable spaces, there is a countable subcover $\{D_i\}_{i=1}^\infty$ such that
$\mathcal{Z}_m^H=\mathcal{Z}_m\cup\bigcup\limits_{i=1}^\infty D_i,$
where the $D_i$ are open discs in $\mathcal{Z}_m^H$ for each $i$. Therefore, we have $\pi_1(D_i)=0$ for all $i\geq 1$. Letting
$\mathcal{Z}_{m,k}=\mathcal{Z}_m\cup\bigcup\limits_{i=1}^k D_i$, we get
\begin{align*}\pi_1(\mathcal{Z}_{m,k})*\pi_1(D_{k+1})=\pi_1(\mathcal{Z}_{m,k})=\pi_1(\mathcal{Z}_{m,{k-1}}\cup D_k), \quad\text{ for any } k.
\end{align*}

We know that $\text{codim}_{\mathbb{C}}$($\mathcal{Z}_m^H\backslash \mathcal{Z}_m)\geq 1.$
Therefore, since $$D_{k+1}\backslash\mathcal{Z}_{m,k}\subseteq D_{k+1}\backslash \mathcal{Z}_m,$$ we have $\text{codim}_{\mathbb{C}}(D_{k+1}\backslash \mathcal{Z}_{m,k})\geq 1$ for any $k$. As a consequence we conclude that $D_{k+1}\cap \mathcal{Z}_{m,k}$ is path-connected. Hence we apply the Van Kampen Theorem on $X_k=D_{k+1}\cup \mathcal{Z}_{m,k}$ to conclude that for every $k$, the following group homomorphism is surjective:
\[\xymatrix{\pi_1(\mathcal{Z}_{m,k})=\pi_1(\mathcal{Z}_{m,k})*\pi_1(D_{k+1})\ar@{->>}[r]&\pi_1(\mathcal{Z}_{m,k}\cup D_{k+1})=\pi_1(\mathcal{Z}_{m,{k+1}}).}\]
Thus we get the directed system:
\[\xymatrix{\pi_1(\mathcal{Z}_m)\ar@{->>}[r]&\pi_1(\mathcal{Z}_{m,1})\ar@{->>}[r]&\cdots\cdots\ar@{->>}[r]&\pi_1(\mathcal{Z}_{m,k})\ar@{->>}[r]&\cdots\cdots}\]
By taking the direct limit of this directed system, we get the surjectivity of the group homomorphism $\pi_1(\mathcal{Z}_m)\to\pi_1(\mathcal{Z}_m^H)$.
\end{proof}

\vspace{+12 pt}

\noindent Center of Mathematical Sciences, Zhejiang University, Hangzhou, Zhejiang 310027, China;\\
Department of Mathematics, University of California at Los Angeles, Los Angeles, CA 90095-1555, USA\\
\noindent e-mail: liu@math.ucla.edu, kefeng@cms.zju.edu.cn

\vspace{+6pt}
\noindent Center of Mathematical Sciences, Zhejiang University, Hangzhou, Zhejiang 310027, China \\
\noindent e-mail: syliuguang2007@163.com

\end{document}